\documentclass[a4paper,fontsize=12pt]{scrartcl}
\usepackage[latin1]{inputenc}
\usepackage[T1]{fontenc}
\usepackage{amsmath}
\usepackage{amssymb}
\usepackage{amsthm}
\usepackage[nobysame]{amsrefs}
\usepackage{setspace} 
\usepackage{graphicx}
\usepackage{array}
\usepackage{numprint}
\usepackage{epstopdf}

\pdfoutput=1

\def\clap#1{\hbox to 0pt{\hss#1\hss}}

\addtokomafont{sectioning}{\rmfamily}
\addtokomafont{descriptionlabel}{\rmfamily}

\newcommand{\n}[2]{\frac{\numprint{#1}}{\numprint{#2}}}

\newcommand{\m}{\mu}

\renewcommand{\l}{\lambda}

\newcommand{\sub}{\subseteq}
\newcommand{\ds}{\displaystyle}
\newcommand{\ts}{\textstyle}
\newcommand{\s}{\mspace{-1mu}}
\newcommand{\N}{\mathbb{N}}
\newcommand{\R}{\mathbb{R}}
\newcommand{\C}{\mathbb{C}}

\renewcommand{\S}{\mathcal{S}}

\renewcommand{\dots}{.\,.\,.\,}
\renewcommand{\cdots}{\cdot \,\cdot \,\cdot \,}

\newcommand{\norm}[2]{\lVert #1 \rVert_{#2}}

\newcommand{\abs}[1]{\lvert #1 \rvert}
\newcommand{\bigabs}[1]{\bigl\lvert #1 \bigr\rvert}

\protect\renewcommand{\s}{\operatorname{s}}
\protect\renewcommand{\c}{\operatorname{c}}
\DeclareMathOperator{\e}{e}

\newtheorem{theo}{Theorem}[section]
\newtheorem{lem}[theo]{Lemma}
\newtheorem{cor}[theo]{Corollary}
\newtheorem{pro}[theo]{Proposition}

\theoremstyle{definition}
\newtheorem{defi}[theo]{Definition}
\newtheorem{exa}[theo]{Example}
\newtheorem{rem}[theo]{Remark}
\newtheorem{conj}{Conjecture}
\newtheorem{rem2}[conj]{Remark}

\title{Measure Theoretic Trigonometric Functions}
\author{Peter Arzt}

\begin{document}
\maketitle


\begin{abstract}
We study the eigenvalues and eigenfunctions of the Laplacian $\Delta_{\m}=\frac{d}{d\m}\frac{d}{dx}$ for a Borel probability measure $\m$ on the interval $[0,1]$ by a technique that follows the treatment of the classical eigenvalue equation $f'' = -\l f$ with homogeneous Neumann or Dirichlet boundary conditions. For this purpose we introduce generalized trigonometric functions that depend on the measure $\m$. 
In particular, we consider the special case where $\m$ is a self-similar measure like e.g. the Cantor measure. We develop certain trigonometric identities that generalize the addition theorems for the sine and cosine functions. In certain cases we get information about the growth of the suprema of normalized eigenfunctions. For several special examples of $\m$ we compute eigenvalues of $\Delta_{\m}$ and $L_{\infty}$- and $L_2$-norms of eigenfunctions numerically by applying the formulas we developed.  
\end{abstract}

\section{Introduction}

Assume that $\m$ is a Borel probability measure on the interval $[0,1]$.  We consider the  Laplacian $\Delta_{\m}$ on $[0,1]$ for the measure $\m$ and study the eigenvalue problem
\[ \Delta_{\m} f = -\l f\]
with either homogeneous Dirichlet boundary conditions
\[ f(a)=f(b)=0,\]
or homogeneous Neumann boundary conditions
\[ f'(a)=f'(b)=0.\]

The definition of $\Delta_{\m}$ involves the derivative with respect to the measure $\m$. If a function $g\colon [0,1]\to \R$ allows the representation 
\begin{equation}\label{intromder} 
g(x) = g(a) + \int_{[a,x]} \frac{dg}{d\m}\, d\m
\end{equation}
for all $x\in [0,1]$, then $\frac{dg}{d\m}$ is unique in $L_2(\m)$ and is called the \emph{$\m$-derivative} of $g$. In Freiberg \cite{freiberg:analytical} an analytic calculus of
the concept of $\m$ derivatives is developed.

The operator $\Delta_{\m}$ is then given by 
\[ \Delta_{\m}f = \frac{d}{d\m} f' \]
for all $f\in L_2(\m)$ for which $f'$ and the $\m$-derivative of $f'$ exist.

It is well known that if $\m$ is a non-atomic Borel measure, $\Delta_{\m}$ has a pure point spectrum consisting only of eigenvalues with multiplicity one, that accumulate at infinity, 
see Freiberg \cite{freiberg:analytical}*{Lem. 5.1 and Cor. 6.9} or Bird, Ngai and Teplyaev \cite{bird:laplacians}*{Th. 2.5}. Moreover, we have a pure point spectrum not only in the non-atomic case, see Vladimirov and Sheipak \cite{vladimirov:singular}.

This operator and the resulting eigenvalue problem  has been studied in numerous papers, for example in
 Feller \cite{feller:second},
 McKean and Ray \cite{mckean:spectral},
 Kac and Krein \cite{kac:spectral},
 Fujita \cite{fujita:diffusion}, 
 Naimark and Solomyak \cite{naimark:eigenvalue},
 Freiberg and Zähle \cite{freiberg:harmonic},
 Bird, Ngai and Teplyaev \cite{bird:laplacians},
 Freiberg \cites{freiberg:analytical, freiberg:asymptotics, freiberg:pruefer,  freiberg:dirichlet},
 Freiberg and Löbus \cite{freiberg:zeros},
 Hu, Lau and Ngai \cite{hu:laplace},
 Chen and Ngai \cite{chen:eigenvalues}, and
 Arzt and Freiberg \cite{arzt:asymptotics}.

In this paper we give a new technique of determining the eigenvalues and eigenfunctions of $\Delta_{\m}$ that involves a generalization of the sine and cosine functions.

In this we follow the classical case, where $\m$ is the Lebesgue measure. There, the Dirichlet eigenvalue problem reads
\begin{align*} f''=-\l f \\
f(0)=f(1)=0.
\end{align*}
Then, for every non-negative $\l$, $f(x)=\sin (\sqrt{\l}x)$ satisfies the equation as well as the boundary condition on the left-hand side. On the right-hand side,
the boundary condition is only met if $\sqrt{\l}$ is a zero point of the sine function, which are, indeed, very well known.

If we impose Neumann boundary conditions $f'(0)=f'(1)=0$, we take $f(x)=\cos (\sqrt{\l}x)$, because this complies automatically with the left-hand side condition.
The right-hand side condition again is satisfied if $\sqrt{\l}$ is a zero point of the sine function, which leads to the same eigenvalues as in the Dirichlet case (supplemented by zero). But here sine appears as the derivative of cosine, which will make a difference when we take more general measures.

Now let $\m$ be an arbitrary Borel probability measure on $[0,1]$. We construct functions $\s_{\l,\m}$ and $\c_{\l,\m}$ as a replacement for $\sin$ and $\cos$ by  generalizing the series
\[\sin(zx)=\sum_{n=0}^\infty (-1)^n \frac{(zx)^{2n+1}}{(2n+1)!}\]
and
\[\cos(zx)=\sum_{n=0}^\infty (-1)^n \frac{(zx)^{2n}}{(2n)!}.\]
There we replace $\frac{x^n}{n!}$ by appropriate functions $p_n(x)$ or $q_n(x)$, depending on whether we impose Neumann or Dirichlet boundary conditions.
These functions fulfill the eigenvalue equation and meet the left-hand side Dirichlet and Neumann boundary condition, respectively.

Putting $p_n:=p_n(1)$ and $q_n:=q_n(1)$, we define
\[ \sin_{\m}^D(z):= \sum_{n=0}^\infty (-1)^n q_{2n+1} z^{2n+1} \] 
and
\[ \sin_{\m}^N(z):= \sum_{n=0}^\infty (-1)^n p_{2n+1} z^{2n+1}. \] 
For $\s_{\l,\m}(z,\cdot)$ and $\c_{\l,\m}(z,\cdot)$ to also match the right-hand side conditions, $z$ has to be chosen as a zero point of $\sin_{\m}^D$ in the Dirichlet case and $\sin_{\m}^N$ in the Neumann case. All this is described in Section \ref{sect:trig}.

In Section \ref{sec:l2norm} we show how to compute the norms in $L_2(\m)$ of the eigenfunctions by using the sequences $p_n$ and $q_n$.

The functions $\c_{\l,\m}(z,\cdot)$ and $\s_{\l,\m}(z,\cdot)$ satisfy an identity that generalizes the classical trigonometric identity. This is established in Section 
\ref{sec:trigind}.

In Section \ref{sec:sym} we consider symmetric measures and get some symmetry results.

The main results are established in Section \ref{sec:ss}. We outline these briefly here. Since the functions $p_n(x)$ and $q_n(x)$ are determined in a process of iterative integration alternately with respect to $\m$ and the Lebesgue measure, the coefficients $p_n$ and $q_n$ are difficult to compute in general. But if $\m$ is a self-similar measure with respect to the mappings
\[ S_1(x) = r_1 x \qquad \text{and} \qquad S_2(x)= r_2 (x-1) + 1\]
as well as the weight factors $m_1$ and $m_2$, we develop a recursion formula for $p_n$ and $q_n$.

To illustrate the structure of this recursion formula, we consider again the classical Lebesgue case. There we have
\[ p_n=q_n=\frac{1}{n!}\]
which leads to $\sin_{\m}^N(z)=\sin_{\m}^D(z)=\sin(z)$. The sequence $p_n = \frac{1}{n!}$ can be viewed as the solution of 
the problem
\begin{equation}\label{intropn}
\begin{split}
2^n p_n & = \sum_{i=0}^n p_i\,p_{n-i},\\
p_0 & = p_1 = 1,
\end{split}
\end{equation}
which is derived from the equation $2^n = \sum\limits_{i=0}^n \binom{n}{i}$. Our recursion formula for self-similar $\m$ looks a little more involved, as it 
distinguishes between the two different kinds of boundary conditions. Additionally it is different for even and odd values of $n$, and it involves the
 parameters $r_1, r_2, m_1, m_2$ of the measure. However, it has the same basic structure as \eqref{intropn}.

Moreover, we establish functional equations involving $\sin_{\m}^N$ and $\sin_{\m}^D$ that can be viewed as generalizations of
the classical addition theorems.

In Section \ref{sec:r1m1gleich} we consider the especially  interesting case where $r_1m_1 = r_2m_2$. Then the  Neumann eigenvalues fulfill a renormalization formula
\[ \l_{2n} = R \,\l_n,\]
where $1/R = r_1 m_1$. This property has been established in a special case by Volkmer \cite{volkmer:eigenvalues} and in our setting by 
Freiberg \cite{freiberg:pruefer}. This formula allows us to investigate the growth of subsequences 
\[ \Bigl(\norm{\tilde{f}_{k2^n}}{\infty}\Bigr)_{n\in\N}, \qquad \text{for odd $k$},\]
where $\tilde{f}_n$ denotes an eigenfunction to the $n$th Neumann eigenvalue that is normalized to one in $L_2(\m)$.

We show in Section \ref{r1r2g1} that, if we assume $r_1+r_2=1$ in addition to  $r_1m_1 = r_2m_2$, the Dirichlet and Neumann eigenvalues coincide.

Finally, by using the formulas we developed in the course of our investigations, we compute approximations of eigenvalues for certain examples in Section \ref{numerical}.

Several remarks about possible further investigations are made in Section \ref{sec:rem}.


\section{Derivatives and the Laplacian with respect to a measure}
\label{se:der}

As in Freiberg \cites{freiberg:analytical, freiberg:asymptotics}, we define a derivative of a function with respect to a measure.

\begin{defi}\label{def:derivative}
	Let $\m$ be a non-atomic Borel probability measure on $[0,1]$ and let $f\colon [0,1]\to \R$. A function $h\in  L_2([0,1],\m)$ is called the \emph{$\m$-derivative of $f$}, if
	\[f(x)=f(a)+\int_a^x h\,d\m \qquad \text{for all } x\in [0,1]. \]
\end{defi}

As can be easily seen, the $\m$-derivative in Definition \ref{def:derivative} is  unique in $L_2(\m)$. We denote the $\m$-derivative of a function $f$ by $\frac{df}{d\m}$. The $\l $-derivative $\frac{df}{d\l }$ we denote by $f'$, where $\l$ denotes the Lebesgue measure on $[0,1]$.

We define $H^1([0,1],\m)=H^1(\m)$ to be the space of all $L_2(\m)$-functions whose $\m$-derivative exists. 
According to our definition, if it exists, the $\m$-derivative is always in $L_2(\m)$, and thus it is clear that, for every non-atomic measure $\m$, all functions in $H^1(\m)$ are continuous.
In case $\m=\l$ is the Lebesgue measure, the definition of $H^1(\l)$ is equivalent to the usual definition of the Sobolev space $H^1=W_2^1$.

The following useful lemma is an analogue to integration by parts and can be found in Freiberg \cite{freiberg:analytical}*{Prop. $3.1$}.
\begin{lem}\label{parts}
	For $c,d \in [0,1]$ with $c<d$ and functions $f\in H^1(\m)$ and $g \in H^1(\l)$ we have
	\[ \int_c^d \frac{df}{d\m}(t)\, g(t) \, d\m(t) = f\,g\Big|_c^d - \int_c^d f(t)\, g'(t) \, dt.\]
\end{lem}

Let $\nu$ be another non-atomic Borel probability measure on $[0,1]$.
Then the space $H^2(\nu,\m)$ is defined to be the collection of all functions in $H^1(\nu)$ whose $\nu$-derivative belongs to $H^1(\m)$.

Now we define the operator $\Delta_{\m}$ on which our investigations are focused as
\[ \Delta_\m f := \frac{d}{d\m} f'\]
for all $f\in H^2(\l ,\m)$.

\begin{rem}
In Freiberg \cite{freiberg:analytical}*{Cor. 6.4} is shown that $H^2(\l ,\m)$ is dense in $L_2(\m)$. Furthermore, it is well known (see e.g. \cite{freiberg:analytical}*{Cor. 3.2}) that $\Delta_{\m}$ is a negative symmetric operator on $L_2(\m)$.
\end{rem}

\section{Generalized trigonometric functions}
\label{sect:trig}

Let $\m$ be an atomless Borel probability measure on $[0,1]$. We construct sequences of functions $p_n(x)$ and $q_n(x)$ depending on $\m$.
\begin{defi}\label{defpq}
For $x\in[0,1]$ we set $p_0(x)=q_0(x)=1$ and, for $n\in\N$,
\[p_n(x):=\begin{cases} \int_0^x p_{n-1}(t)\, d\m(t)		& \text{, if $n$ is odd,}\\
																					\int_0^x p_{n-1}(t)\, dt						& \text{, if $n$ is even,}
								\end{cases}\]

and
\[q_n(x):=\begin{cases} \int_0^x  q_{n-1}(t)\, dt					& \text{, if $n$ is odd,}\\
																					\int_0^x  q_{n-1}(t)\,d\m(t)		& \text{, if $n$ is even.}
								\end{cases}\]
\end{defi}

Then, for $n\in \N_0$, we have by definition  $p_{2n}, q_{2n+1} \in H^1(\l)$, $p_{2n+1}, q_{2n} \in H^1(\m)$ and
\[\frac{d}{d\m}p_{2n+1}	 =  p_{2n}, \quad	 q_{2n+1}'	 = q_{2n}, \quad p_{2n}'  =  p_{2n-1} \quad \text{and} \quad	 \frac{d}{d\m}q_{2n}		 = q_{2n-1}.\]

\begin{rem}
	\begin{enumerate}
	\item
	If we take $\m$ to be the Lebesgue measure, then
	\[p_n(x)=q_n(x)=\dfrac{x^n}{n!}.\]
	In the following, we will transfer classical concepts and techniques to a general measure $\m$ by replacing $\dfrac{x^n}{n!}$ by $p_n(x)$ or $q_n(x)$. In 
	this sense, we can look at $p_n(x)$ or $q_n(x)$ as a kind of generalized monomials.
	\item
	It is easy to see that for all $n\in \N$ and $x\in [0,1]$,
	\[q_{n+1}(x) \leq p_n(x)\qquad \text{and} \qquad p_{n+1}(x) \leq q_n(x).\]
	\end{enumerate}
\end{rem}

To prove convergence of the series defined below, we will need the following lemma. 
\begin{lem}\label{convabsch}
	For all $x\in [0,1]$, $z\in \R$ and $n\in \N_0$ holds
	\begin{align*}
	p_{2n+1}(x) &\leq \frac{1}{n!} \, q_2(x)^n, &   q_{2n+1}(x) &\leq \frac{1}{n!} \, p_2(x)^n,\\
	p_{2n}(x) &\leq \frac{1}{n!} \, p_2(x)^n, &   q_{2n}(x) &\leq \frac{1}{n!} \, q_2(x)^n.
	\end{align*}
\end{lem}
\begin{proof}
	The estimates for $q_{2n+1}(x)$ and $q_{2n}(x)$ are proved in Freiberg and Löbus \cite{freiberg:zeros}*{Lemma 2.3} with complete induction. 
	The proof of the other estimates works analogously.
\end{proof}

\begin{defi}\label{defi:sums}
Using the functions $p_n(x)$ and $q_n(x)$ we now define for $x\in[0,1]$ and $z\in\R$:
\begin{align*}
\s_{\m,\l}(z,x)&:= \sum_{n=0}^\infty (-1)^n\, z^{2n+1}p_{2n+1}(x),&\qquad \s_{\l,\m}(z,x)&:= \sum_{n=0}^\infty (-1)^n\, z^{2n+1}q_{2n+1}(x),\\
\c_{\l,\m}(z,x)&:= \sum_{n=0}^\infty (-1)^n\, z^{2n}p_{2n}(x),&\qquad \c_{\m,\l}(z,x)&:= \sum_{n=0}^\infty (-1)^n\, z^{2n}q_{2n}(x).
\end{align*}
\end{defi}
Note that for every $z\in \R$,
\[\c_{\l,\m}(z,\cdot),\,\s_{\l,\m}(z,\cdot) \in H^2(\l,\m)\qquad \text{and} \qquad  \s_{\m,\l}(z,\cdot),\, \c_{\m,\l}(z,\cdot) \in H^2(\m,\l).\]

\begin{rem}
\begin{enumerate}
\item
	If $\m$ is the Lebesgue measure, then
	\[ \s_{\m,\l}(z,x) = \s_{\l,\m}(z,x) = \sin (zx), \qquad \c_{\l,\m}(z,x) = \c_{\m,\l}(z,x) = \cos (zx).\]
\item
	Functions corresponding to $\s_{\l,\m}(z, \cdot)$ and $\c_{\m,\l}(z,\cdot)$ have also been constructed in Freiberg and Löbus \cite{freiberg:zeros},
	where they are used to determine the number of zeros of Dirichlet eigenfunctions.
\end{enumerate}
\end{rem}

\begin{lem}\label{diffrules}
	For every $z\in \R$ the series in Definition \ref{defi:sums} converge uniformly absolutely on $[0,1]$ and the following differentiation rules hold:
\begin{align*}
\frac{d}{d\m} \s_{\m,\l}(z, \cdot) &= z \c_{\l,\m}(z,\cdot), &  \s'_{\l,\m}(z,\cdot) &= z \c_{\m,\l}(z,\cdot),\\
 \c'_{\l,\m}(z,\cdot) & = -z \s_{\m,\l}(z, \cdot), & \frac{d}{d\m} \c_{\m,\l}(z,\cdot) &= -z\s_{\l,\m}(z,\cdot).
\end{align*}
\end{lem}
\begin{proof}
	Let $z\in\R$. Since $q_2(x)\leq 1$ for $x\in [0,1]$, we get by Lemma \ref{convabsch} for $N\in \N$
	\[\sup_{x\in[0,1]}\sum_{n=N}^\infty \abs{z}^{2n+1}p_{2n+1}(x) \leq \sup_{x\in[0,1]} \sum_{n=N}^\infty \frac{\abs{z}^{2n+1}\,q_2(x)^n}{n!} \leq \sum_{n=N}^\infty \frac{\abs{z}^{2n+1}}{n!}. \]
	Hence, for every $z\in \R$ the series $\sum\limits_{n=0}^\infty \abs{z}^{2n+1}p_{2n+1}(x)$ converges uniformly in $x$.	The proof for the other series works analogously with the estimates in Lemma \ref{convabsch}. Thus, we can differentiate term by term and get the above rules.
\end{proof}

Now we show the relation between $\c_{\l,\m}(z,\cdot)$ and $\s_{\l,\m}(z,\cdot)$ to the eigenvalue problem for $\Delta_{\m}$.
Consider the Neumann problem
\begin{gather*}
\frac{d}{d\m} f'=-\l f\\
f'(0)=f'(1)=0.
\end{gather*}

It is well known that the eigenvalues can be sorted according to size such that
\[ \l_{N,0} < \l_{N,1} < \l_{N,2} < \dotsm, \]
where $\l_{N,0} =0$ and $\lim\limits_{m\to \infty} \l_{N,m} = \infty$.

\label{NandDProb}
\begin{pro}\label{pro:neumanneigenvalues}
	The Neumann eigenvalues $\l_{N,m}$, $m\in \N_0$, are the squares of the non-negative zeros of the function $\sin_{\m}^N$ given by
	\[\sin_{\m}^N(z):=\s_{\m,\l}(z,1)=\sum_{n=0}^\infty (-1)^n p_{2n+1} z^{2n+1},\qquad \text{for $z\in\R$},\]
	where we write $p_n$ instead of $p_n(1)$ for simplicity.	The corresponding eigenfunctions $f_{N,m}$ are given by
	\[ f_{N,m} (x) := \c_{\l,\m}\bigl( \sqrt{\l_{N,m}}, x\bigr) = \sum_{n=0}^\infty (-1)^n\, \l_{N,m}^n\, p_{2n}(x), \quad x\in [0,1].\]
\end{pro}
\begin{proof}
	Using the differentiation rules from Lemma \ref{diffrules} it is easy to see that $\c_{\l,\m}(z,\cdot)$ satisfies the eigenvalue equation if $\l=z^2$, while it also fulfills
	the left boundary condition $\c'_{\l,\m}(z,0)=-z\s_{\m,\l}(z,0)=0$. Here, the dash refers to the second argument of $\c_{\l,\m}$.
	In order that $\c_{\l,\m}(z,\cdot)$ satisfies the right boundary condition, too, $z$ has to be zero itself or it must be
	chosen such that $\s_{\m,\l}(z,1)=0$. It is known (see Freiberg \cite{freiberg:analytical} p.40) that the solution of the above problem is unique up to a
	multiplicative constant. So $z$ is a zero point of $\sin_{\m}^N$ if and only if $z^2$ is a Neumann eigenvalue of $-\frac{d}{d\m}\frac{d}{dx}$.

	Thus, for $m\in \N_0$, $f_{N,m} = \c_{\l,\m}\bigl( \sqrt{\l_{N,m}}, x\bigr) $	is an eigenfunction to the $m$th Neumann eigenvalue $\l_{N,m}$. 
\end{proof}

We treat the Dirichlet eigenvalue problem
\begin{gather*}
\frac{d}{d\m} f'=-\l f\\
f(0)=f(1)=0
\end{gather*}
similarly. We denote the Dirichlet eigenvalues such that
\[ \l_{D,1} < \l_{D,2} < \l_{D,3} < \dotsm \]
where $\l_{D,1} > 0$ and $\lim\limits_{n\to \infty} \l_{D,n} = \infty$.

\begin{pro}\label{pro:dirichleteigenvalues}
	The Dirichlet eigenvalues $\l_{D,m}$, $m\in \N$, are the squares of the positive zeros of the function $\sin_{\m}^D$ given by
	\[ \sin_{\m}^D(z) :=\s_{\l,\m}(z,1)=\sum_{n=0}^\infty (-1)^n q_{2n+1} z^{2n+1}, \qquad \text{for $z\in\R$}\]
	where, as above, $q_n$ stands for $q_n(1)$. The corresponding eigenfunctions $f_{D,m}$ are given by
	\[ f_{D,m} (x) = \s_{\l,\m}\bigl( \sqrt{\l_{D,m}}, x \bigr) = \sqrt{\l_{D,m}} \sum_{n=0}^\infty (-1)^n\, \l_{D,m}^{n}\, q_{2n+1}(x), \quad x\in [0,1].\]
\end{pro}
\begin{proof}
	The function $\s_{\l,\m}(z,\cdot)$ satisfies the equation if $\l=z^2$ and also the left boundary condition  $\s_{\l,\m}(z,0)=0$ . The right boundary condition gives $\s_{\l,\m}(z,1)=0$.
	So $z^2$ is a Dirichlet eigenvalue of $-\frac{d}{d\m}\frac{d}{dx}$ if and only if $z$ is a  zero point of $\sin_{\m}^D$  and $z \neq 0$.
	
	Thus, for $m\in \N$, the function $f_{D,m}= \s_{\l,\m}\bigl( \sqrt{\l_{D,m}}, x\bigr)$ is an eigenfunction to the $m$th Dirichlet eigenvalue $\l_{D,m}$. 
\end{proof}

So if we only know the sequences $\big(p_n(1)\big)_n$ and $\big(q_n(1)\big)_n$, we can determine the Neumann and Dirichlet eigenvalues by means of the functions $\sin_{\m}^N$ and $\sin_{\m}^D$.

\begin{rem}
\begin{enumerate}
	\item
	As was pointed out to me only recently by V. Kravchenko, a construction analogous to that in Definitions \ref{defpq} and \ref{defi:sums} has also been done in 
	 \cite{kravchenko:representation}	for Sturm-Liouville equations of the form
	\[ (pu')' + q u = z^2\, u.\]
	There, the corresponding spectral problem is also transformed to the problem of finding zeros of a power series as
	in Propositions \ref{pro:neumanneigenvalues} and \ref{pro:dirichleteigenvalues}. See also Kravchenko and  Porter \cite{kravchenko:spps}.
	\item
	An eigenfunction is only unique up to a multiplicative constant. Throughout the chapter we will use the notations $f_{N,m}$ and $f_{D,m}$
	for the eigenfunctions as  constructed above. One would also get these by imposing the additional conditions $f_{N,m}(0)=1$ and $f'_{D,m}(0)=\sqrt{\l_{D,m}}$.
\end{enumerate}
\end{rem}

Analogously to $\sin_{\m}^N$ and $\sin_{\m}^D$ we define 
\[ \cos_{\m}^N(z):= \c_{\l,\m}(z,1) = \sum_{n=0}^\infty (-1)^n p_{2n}z^{2n}\]
and
\[ \cos_{\m}^D(z):= \c_{\m,\l}(z,1) = \sum_{n=0}^\infty (-1)^n q_{2n} z^{2n}\]
for $z\in\R$.

\label{NDDNproblem}
These functions are linked with the eigenvalue problems with mixed boundary conditions
\begin{align*}
\hspace{-3cm}(ND)&\hspace{1.5cm}
\begin{gathered}
\frac{d}{d\m} f'=-\l f\\
f'(0)=0, \quad f(1)=0,
\end{gathered}\\
\intertext{and}
\hspace{-3cm}(DN)&\hspace{1.5cm}
\begin{gathered}
\frac{d}{d\m} f'=-\l f\\
f(0)=0, \quad f'(1)=0.
\end{gathered}
\end{align*}
We treat these problems as the problems in the above Propositions \ref{pro:neumanneigenvalues} and \ref{pro:dirichleteigenvalues}.
If $\l>0$ is chosen such that $\cos_{\m}^N\bigl(\sqrt{\l}\bigr)=0$, the solutions to (ND) are multiples of $\c_{\l,\m}\bigl( \sqrt{\l}, \cdot \bigr)$, because 
\[ \c'_{\l,\m}\bigl( \sqrt{\l}, 0 \bigr)=-\sqrt{\l}\s_{\m,\l}(\sqrt{\l},0)=0\]
and
\[ \c_{\l,\m}\bigl( \sqrt{\l}, 1 \bigr) =\cos_{\m}^N\bigl(\sqrt{\l}\bigr).\]
Similarly, if $\l>0$ satisfies $\cos_{\m}^D\bigl(\sqrt{\l}\bigr)=0$, the solutions to (DN) are multiples of $\s_{\l,\m}(\sqrt{\l}, \cdot)$, because 
\[ \s_{\l,\m}(\sqrt{\l}, 0)=0\]
and
\[ \s_{\l,\m}'(\sqrt{\l}, 1)=\sqrt{\l}\c_{\m,\l}(\sqrt{\l}, 1)=\sqrt{\l} \cos_{\m}^D\bigl(\sqrt{\l}\bigr),\]
where the derivative refers to the second argument of $\s_{\l,\m}$.
Therefore, the (ND) eigenvalues are the squares of the zeros of $\cos_{\m}^N$ and the (DN) eigenvalues are the squares of the zeros of $\cos_{\m}^D$.

In the course of the following sections we will often use some of the following easy to prove multiplication formulas which we state here for easy reference. For absolutely summable sequences $(a_n)_n$ and $(b_n)_n$ holds
\begin{align}
\label{formgg}\bigg( \sum_{j=0}^\infty a_{2j} \bigg) \cdot \bigg( \sum_{k=0}^\infty b_{2k} \bigg) 		& = \sum_{n=0}^\infty \sum_{k=0}^n a_{2k}\, b_{2n-2k}, \\
\label{formgu}\bigg( \sum_{j=0}^\infty a_{2j} \bigg) \cdot \bigg( \sum_{k=0}^\infty b_{2k+1} \bigg) 	&	= \sum_{n=0}^\infty \sum_{k=0}^n a_{2k}\, b_{2n+1-2k}, \\
\label{formuu}\bigg( \sum_{j=0}^\infty a_{2j+1} \bigg) \cdot \bigg( \sum_{k=0}^\infty b_{2k+1} \bigg) &	= \sum_{n=0}^\infty \sum_{k=0}^n a_{2k+1}\, b_{2n+1-2k}.
\end{align}

\section{Calculation of $L_2$-norms}
\label{sec:l2norm}

It turns out that by knowing the sequences $(p_n)_n$ and $(q_n)_n$ we can not only determine the Neumann and Dirichlet eigenvalues, but also the $L_2(\m)$-norms of the eigenfunctions $f_{N,m}$ and $f_{D,m}$. We will need the following lemma to achieve that.

\begin{lem}\label{lem:intpp}
	For $k, n \in \N_0$ with $k\leq n$ and for all $x\in [0,1]$ we have
	\begin{equation}\label{intpp} \int_0^x p_{2k}(t) \, p_{2n-2k}(t) \, d\m(t) = \sum_{j=0}^{2k} (-1)^j p_j(x) \, p_{2n+1-j}(x)\end{equation}
	and
	\begin{equation}\label{intqq} \int_0^x q_{2k+1}(t) \, q_{2n+1-2k}(t) \, d\m(t) = \sum_{j=0}^{2k+1} (-1)^{j+1} q_j(x) \, q_{2n+3-j}(x).\end{equation}
\end{lem}
\begin{proof}
	We prove \eqref{intpp} by induction on $k$. If $k=0$ and $n\geq 0$, we have
	\[ \int_0^x p_{0}(t) \, p_{2n}(t) \, d\m(t) = p_{2n+1}(x)\]
	and so the assertion holds. Now, take $k\in \N_0$ and assume that the assertion holds for $k$ and all $n\geq k$. Then, for all $n \geq k+1$, 
	\begin{align*}
	&\int_0^x p_{2k+2}(t) \, p_{2n-2k-2}(t) \, d\m(t)		= p_{2k+2}(x)\,p_{2n-2k-1}(x) - \int_0^x p_{2k+1}(t)\,p_{2n-2k-1}(t)\,dt \\
	&																										= p_{2k+2}(x)\,p_{2n-2k-1}(x) - p_{2k+1}(x)\,p_{2n-2k}(x) + \int_0^x p_{2k}(t)\,p_{2n-2k}(t)\,d\m(t),
	\end{align*}
	by Lemma \ref{parts}.
	Thus, by the induction hypothesis, we have for all $n\geq k+1$
	\[ \int_0^x p_{2k+2}(t) \, p_{2n-2k-2}(t) \, d\m(t)		= \sum_{j=0}^{2k+2} (-1)^j p_j(x) \, p_{2n+1-j}(x),\]
	which proves \eqref{intpp}.
	
	The proof of \eqref{intqq} works the same way.
\end{proof}

\begin{pro}\label{L2norm}
	Let $z\in \R$. Then
	\begin{equation}\label{L2normcp}
		\norm{\c_{\l,\m}(z,\cdot)}{L_2(\m)}^2 = \sum_{n=0}^\infty (-1)^n z^{2n} \sum_{k=0}^n (n+1 - 2k) \, p_{2k}\, p_{2n+1-2k},
	\end{equation}
	and
	\begin{equation}\label{L2normsq}
		\norm{\s_{\l,\m}(z,\cdot)}{L_2(\m)}^2 = \sum_{n=0}^\infty (-1)^n z^{2n+2} \sum_{k=0}^{n+1} (n+1 - 2k) \, q_{2k+1}\, q_{2n+2-2k},
	\end{equation}
	where $p_j=p_j(1)$ and $q_j=q_j(1)$.
\end{pro}
\begin{proof}
	First we prove \eqref{L2normcp}. Using \eqref{formgg} we get for all $x\in [0,1]$ and $z\in\R$
	\begin{align*}
	\c_{\l,\m}(z,x)^2		&	= \bigg( \sum_{j=0}^\infty (-1)^j z^{2j} p_{2j}(x)\bigg)\bigg( \sum_{k=0}^\infty (-1)^k z^{2k} p_{2k}(x)\bigg)\\
								& = \sum_{n=0}^\infty (-1)^n\,z^{2n}\,\sum_{k=0}^n p_{2k}(x)\,p_{2n-2k}(x).
	\end{align*}
	Consequently, applying \eqref{intpp},
	\begin{align*}
	\norm{\c_{\l,\m}(z,\cdot)}{L_2(\m)}^2 & = \int_0^1 \c_{\l,\m}(z,t)^2\,d\mu(t) \\
													& = \sum_{n=0}^\infty (-1)^n\,z^{2n}\,\sum_{k=0}^n \int_0^1 p_{2k}(t)\,p_{2n-2k}(t)\,d\m(t)\\
													& = \sum_{n=0}^\infty (-1)^n\,z^{2n}\,\sum_{k=0}^n \sum_{j=0}^{2k} (-1)^j p_j \, p_{2n+1-j}.
	\end{align*}
	Note that for any sequence $a=(a_j)_{j\in\N_0}$ holds
	\begin{align*}
		\sum_{k=0}^n \sum_{j=0}^{2k} a_j	&	= \sum_{k=0}^n \sum_{j=0}^k a_{2j} + \sum_{k=1}^n \sum_{j=0}^{k-1} a_{2j+1}\\
																			& = \sum_{j=0}^n \sum_{k=j}^n a_{2j} + \sum_{j=0}^{n-1} \sum_{k=j+1}^n a_{2j+1}\\
																			& = \sum_{j=0}^n (n-j+1)\, a_{2j} + \sum_{j=1}^n (n-j+1)\, a_{2j-1}
	\end{align*}
	and thus,
	\begin{align*}
		\sum_{k=0}^n \sum_{j=0}^{2k} (-1)^j p_j \, p_{2n+1-j}	& = \sum_{k=0}^n (n-k+1)\, p_{2k}\, p_{2n+1-2k} - \sum_{k=1}^n (n-k+1)\, p_{2k-1}\, p_{2n+2-2k}\\
																													& = \sum_{k=0}^n (n-k+1)\, p_{2k}\, p_{2n+1-2k} - \sum_{k=1}^n k \, p_{2k} \, p_{2n+1-2k}\\
																													& = \sum_{k=0}^n (n+1 - 2k)\, p_{2k}\, p_{2n+1-2k},
	\end{align*}
	which proves \eqref{L2normcp}.
	
	The proof of \eqref{L2normsq} works analogously.
\end{proof}

We put $z=\sqrt{\l_{N,m}}$ and $z= \sqrt{\l_{D,m}}$ to get the following corollary.

\begin{cor}\label{cor:l2norm}
	The $L_2(\m)$-norm of the Neumann eigenfunction $f_{N,m}$ is given by
	\[ \norm{f_{N,m}}{L_2(\m)}^2 = \sum_{n=0}^\infty (-1)^n \l_{N,m}^n \sum_{k=0}^n (n+1 - 2k) \, p_{2k}\, p_{2n+1-2k}\]
	and of the Dirichlet eigenfunction $f_{D,m}$ by
	\[\norm{f_{D,m}}{L_2(\m)}^2 = \sum_{n=0}^\infty (-1)^n \l_{D,m}^{n+1} \sum_{k=0}^{n+1} (n+1 - 2k) \, q_{2k+1}\, q_{2n+2-2k}.\]
\end{cor}

\section{A trigonometric identity}
\label{sec:trigind}

As in the previous section, we consider an atomless Borel probability measure $\m$ on $[0,1]$. We prove a formula that links the functions $\c_{\l,\m}, \c_{\m,\l}, \s_{\m,\l}$, and $\s_{\l,\m}$ generalizing the trigonometric identity $\sin^2 + \cos^2 = 1$. For this we need the following lemma.

\begin{lem}\label{lem:intqp}
	For $k, n \in \N$ with $k\leq n$ and for all $x\in [0,1]$ we have
	\[ \int_0^x q_{2k-1}(t) \, p_{2n-2k}(t) \, d\m(t) = \sum_{j=0}^{2k-1} (-1)^{j+1} q_j(x) \, p_{2n-j}(x).\]
\end{lem}
\begin{proof}
	We prove this by induction on $k$. For $k=1$ and $n\geq 1$, we get by Lemma \ref{parts}
	\[ \int_0^x q_{1}(t) \, p_{2n-2}(t) \, d\m(t) = q_1(x)\,p_{2n-1}(x) - \int_0^x p_{2n-1}(t)\, dt = q_1(x)\,p_{2n-1}(x) - p_{2n}(x),\]
	and so the assertion holds. Now, take $k\in \N$, and assume that the assertion holds for $k$ and all $n\geq k$. Then, again by using Lemma \ref{parts}, we get
	\begin{align*}
	&\int_0^x q_{2k+1}(t) \, p_{2n-2k-2}(t) \, d\m(t)		= q_{2k+1}(x)\,p_{2n-2k-1}(x) - \int_0^x q_{2k}(t)\,p_{2n-2k-1}(t)\,dt \\
	&																										= q_{2k+1}(x)\,p_{2n-2k-1}(x) - q_{2k}(x)\,p_{2n-2k}(x) + \int_0^x q_{2k-1}(t)\,p_{2n-2k}(t)\,d\m(t).
	\end{align*}
	Thus, by the induction hypothesis, for all $n\geq k+1$,
	\[ \int_0^x q_{2k+1}(t) \, p_{2n-2k-2}(t) \, d\m(t)		= \sum_{j=0}^{2k+1} (-1)^{j+1} q_j(x) \, p_{2n-j}(x),\]
	which finishes the proof.
\end{proof}

\begin{cor}\label{sumpqgleich0}
	If we set $n=k$ in Lemma \ref{lem:intqp}, we get the formula
	\[ \sum_{j=0}^{2n} (-1)^j\,q_j(x)\,p_{2n-j}(x) =0,\]
	which holds for all $n\in\N$ and $x\in [0,1]$.
\end{cor}

With the above corollary we can prove the following theorem.

\begin{theo}\label{pythallg}
	For all $x\in[0,1]$ and $z\in \R$ holds
	\[ \c_{\m,\l}(z,x)\,\c_{\l,\m}(z,x) + \s_{\l,\m}(z,x)\, \s_{\m,\l}(z,x) = 1.\]
\end{theo}
\begin{proof}
	Take $x\in [0,1]$ and $z\in \R$. Then, by Corollary \ref{sumpqgleich0},
	\begin{align*}
		&\c_{\m,\l}(z,x)\,\c_{\l,\m}(z,x) + \s_{\l,\m}(z,x)\, \s_{\m,\l}(z,x)	\\
		&	= \sum_{n=0}^\infty (-1)^n z^{2n} \sum_{k=0}^n q_{2k}(x)\,p_{2n-2k}(x) + \sum_{n=0}^\infty (-1)^n z^{2n+2}\sum_{k=0}^n q_{2k+1}(x)\,p_{2n+1-2k}(x)\\
		& = 1 + \sum_{n=1}^\infty (-1)^n z^{2n} \bigg[ \sum_{k=0}^n q_{2k}(x)\,p_{2n-2k}(x) - \sum_{k=0}^{n-1} q_{2k+1}(x)\, p_{2n-(2k+1)}(x)\bigg]\\
		& = 1 + \sum_{n=1}^\infty (-1)^n z^{2n} \sum_{k=0}^{2n} (-1)^k \, q_k(x) \, p_{2n-k}(x)\\
		& = 1.
	\end{align*}
\end{proof}

\section{Symmetric measures}
\label{sec:sym}

In this section we consider symmetric measures $\m$ on $[0,1]$, meaning that, additionally to being an atomless Borel probability measure, $\m$ shall satisfy
\[ \m \big( [0,x] \big) = \m \big( [1-x, 1] \big) \]
for all $ x \in [0,1] $. 

\begin{pro}\label{pro:symm}
	Let $\m$ be symmetric and let $x \in [0,1]$. Then, for $n \in \N_0$ holds
	\begin{equation}\label{eqsymmungp}
	p_{2n+1}(x) = \sum_{k=0}^n p_{2k+1}\, q_{2n-2k}(x) - \sum_{k=1}^n  p_{2k}\, p_{2n-2k+1}(x) - p_{2n+1}(1-x),
	\end{equation}
	and for $n \in \N$
	\begin{equation}\label{eqsymmgerp}
	p_{2n}(x)	= \sum_{k=0}^{n-1} p_{2k+1}\, q_{2n-2k-1}(x) - \sum_{k=1}^n p_{2k}\, p_{2n-2k}(x) + p_{2n}(1-x).
	\end{equation}
\end{pro}
\begin{proof}
	For $p_1(x)$ the formula reduces to $p_1(x) = p_1 - p_1(1-x)$. This holds since 
	\[ p_1(x) = \m \big( [0,x] \big) = \m \big( [1-x,1] \big) = \int_0^1 \, d\mu - \int_0^{1-x}\,d\mu= p_1 - p_1(1-x).\]
	Assume $p_{2n+1}(x)$ satisfies the above formula for some $n \in \N_0$. Then
	\begin{align*}
	p_{2n+2}(x)	&	= \int_0^x p_{2n+1}(t)\, dt \\
												&	= \sum_{k=0}^n  p_{2k+1} \int_0^x q_{2n-2k}(t)\, dt - \sum_{k=1}^n  p_{2k} \int_0^x p_{2n-2k+1}(t) \, dt - \int_0^x  p_{2n+1}(1-t)\, dt\\
												&	=  \sum_{k=0}^n p_{2k+1}\, q_{2n-2k+1}(x) - \sum_{k=1}^n p_{2k}\, p_{2n-2k+2}(x) - \int_{1-x}^1 p_{2n+1}(t)\, dt \\
												&	=  \sum_{k=0}^n  p_{2k+1}\, q_{2n-2k+1}(x) - \sum_{k=1}^n  p_{2k}\, p_{2n-2k+2}(x)- p_{2n+2}(1) + p_{2n+2}(1-x)\\
												&	=  \sum_{k=0}^n  p_{2k+1}\, q_{2n-2k+1}(x) - \sum_{k=1}^{n+1} p_{2k}\, p_{2n-2k+2}(x)+ p_{2n+2}(1-x).
	\end{align*}
	Now, let the assertion be true for some $2n$, $n\in \N$. Since $\mu$ is symmetric, we have that $d\mu (t) = d \mu (1-t)$. Thus,
	\begin{align*}
	p_{2n+1}(x)	&	= \int_0^x p_{2n}(t)\, d\mu(t) \\
													&	= \sum_{k=0}^{n-1} p_{2k+1} \int_0^x q_{2n-2k-1}(t)\, d\mu(t) - \sum_{k=1}^n p_{2k} \int_0^x p_{2n-2k}(t) \, d\mu(t) + \int_0^x p_{2n}(1-t)\, d\mu(t)\\
													&	=  \sum_{k=0}^{n-1}p_{2k+1}\, q_{2n-2k}(x) - \sum_{k=1}^n p_{2k}\, p_{2n-2k+1}(x) + \int_{1-x}^1 p_{2n}(t)\, d\mu(t)\\
													&	=  \sum_{k=0}^{n-1}p_{2k+1}\, q_{2n-2k}(x)	- \sum_{k=1}^n p_{2k}\, p_{2n-2k+1}(x)  + p_{2n+1}(1) - p_{2n+1}(1-x)\\
													&	=  \sum_{k=0}^n  p_{2k+1}\, q_{2n-2k}(x)- \sum_{k=1}^n  p_{2k}\, p_{2n-2k+1}(x)  - p_{2n+1}(1-x).
	\end{align*}
\end{proof}

\begin{cor}\label{cor:sym}
	Let $\m$ be symmetric. Then, for $n\in \N$,
	\begin{equation} \label{symform1}
		\sum_{k=0}^{n} p_{2k}\, p_{2n-2k+1} = \sum_{k=0}^n p_{2k+1} \, q_{2n-2k}.
	\end{equation}
\end{cor}
\begin{proof}
	This follows from Proposition \ref{pro:symm} by putting $x=1$ in \eqref{eqsymmungp}.
\end{proof}
\begin{rem}
	In the special case where $\m$ is the Lebesgue measure, the above formula reduces to $\sum\limits_{k=0}^n (-1)^k \binom{n}{k}=0$.
\end{rem}
\begin{cor}\label{cor:symm}
	Let $\m$ be symmetric. Then the following statements hold.
	\begin{enumerate}
		\item \label{eveq}
			$\ds p_{2n}=q_{2n}$ for all $n \in \N$.
		\item \label{cospeqcosq}
			$\cos_{\m}^N(z)=\cos_{\m}^D(z)$ for all $z\in \R$.
		\item \label{cospquadplsinpsinq}
			$\cos_{\m}^N(z)^2 + \sin_{\m}^N(z)\,\sin_{\m}^D(z) = 1$	for all $z\in \R$.
		\item \label{evrecit}
			We have the recursion formula
			\begin{equation}\label{evreceq}
			p_{2n} = \frac{1}{2} \sum_{k=1}^{2n-1} (-1)^{k+1}\, p_k\, q_{2n-k}.
			\end{equation}
	\end{enumerate}
\end{cor}
\begin{proof}
	We prove \ref{eveq} by induction. By putting $n=1$ in \eqref{symform1}, we find that 
	\[ p_3+ p_2 p_1 = p_1 q_2 + p_3,\]
	which implies $p_2=q_2$. Assume that $p_{2k}=q_{2k}$ for all $k$ smaller than some $n\in \N$, $n\geq 2$. We reverse the order of the summands in the second sum of \eqref{symform1} to get
	\[ \sum_{k=0}^{n-1} p_{2k}\,p_{2n-2k+1} + p_{2n}\,p_1 = \sum_{k=0}^{n-1} p_{2n-2k+1}\, q_{2k} + p_1\,q_{2n}.\]
	Now it follows from the induction hypothesis that $p_{2n}=q_{2n}$. Then, \ref{cospeqcosq} follows immediately and by Proposition \ref{pythallg}
	also \ref{cospquadplsinpsinq}.
	
	Clearly, \ref{evrecit} follows from \ref{eveq} and Corollary \ref{sumpqgleich0}.
\end{proof}

\begin{pro}\label{cpz1minx}
	Let $\m$ be symmetric. Then, for all $z\in\R$ and $x\in [0,1]$,
	\[ \c_{\l,\m}(z,1-x) = \cos_{\m}^N(z) \c_{\l,\m}(z,x) + \sin_{\m}^N(z) \s_{\l,\m}(z,x).\]
\end{pro}
\begin{proof}
	Rearranging \eqref{eqsymmgerp} gives
	\[ p_{2n}(1-x)	= \sum_{k=0}^n p_{2k}\, p_{2n-2k}(x) - \sum_{k=0}^{n-1} p_{2k+1}\, q_{2n-2k-1}(x).\]
	We multiply the equation with $(-1)^n z^{2n}$ and sum from $n=0$ to infinity to get
	\begin{align*}
	\c_{\l,\m}(z,1-x) 	& = \sum_{n=0}^\infty \sum_{k=0}^n (iz)^{2k} p_{2k}\cdot (iz)^{2n-2k} p_{2n-2k}(x) \\
							& \qquad\qquad - \sum_{n=1}^\infty \sum_{k=0}^{n-1} (iz)^{2k+1}p_{2k+1}\cdot(iz)^{2n-2k-1} q_{2n-2k-1}(x)\\
							& = \sum_{n=0}^\infty (-1)^n z^{2n} p_{2n} \cdot \sum_{k=0}^\infty (-1)^k z^{2k} p_{2k}(x)\\
							& \qquad\qquad + \sum_{n=1}^\infty (-1)^n z^{2n+1} p_{2n+1} \cdot \sum_{k=0}^\infty (-1)^k z^{2k+1} q_{2k+1}(x)\\
							& = \cos_{\m}^N(z) \c_{\l,\m}(z,x) + \sin_{\m}^N(z) \s_{\l,\m}(z,x).
	\end{align*}
\end{proof}

\begin{cor}
	Let $\m$ be symmetric. Then the Neumann eigenfunctions $f_{N,m}$ are either symmetric or antisymmetric, that is, either
	\[f_{N,m}(x)=f_{N,m}(1-x) \qquad \text{or} \qquad f_{N,m}(x) = - f_{N,m}(1-x)\]
	for all $x\in [0,1]$.
\end{cor}
\begin{proof}
	Let $z^2$ be a Neumann eigenvalue. Then, by Proposition \ref{pro:neumanneigenvalues}, $\sin_{\m}^N(z)=0$ and hence, by Corollary \ref{cor:symm} \ref{cospquadplsinpsinq},
	$\abs{\cos_{\m}^N(z)}=1$. Thus, by Proposition \ref{cpz1minx}, we get
	\[ \c_{\l,\m}(z,1-x) = \pm \c_{\l,\m}(z,x).\]
	Since $\c_{\l,\m}(z,\cdot) = f_{N,m}$ for $z^2=\l_m$ the corollary is proved.
\end{proof}

Analogous to \eqref{eqsymmungp} there is a formula relating $q_{2n+1}(x)$ to $q_{2n+1}(1-x)$, namely
\begin{equation}
	q_{2n+1}(x)= \sum_{k=0}^n q_{2k+1}\, p_{2n-2k}(x) - \sum_{k=1}^n q_{2k}\,q_{2n+2k+1}(x) - q_{2n+1}(1-x).
\end{equation}
The proof is exactly like the proof of Proposition \ref{pro:symm}. As in the proof of Proposition \ref{cpz1minx}, we rearrange, multiply with $(-1)^n z^{2n+1}$, and sum up to get
\[\s_{\l,\m}(z,1-x)=\sin_{\m}^D(z) \c_{\l,\m}(z,x) - \cos_{\m}^D(z) \s_{\l,\m}(z,x).\]
If now $z^2$ is a Dirichlet eigenvalue, then $\sin_{\m}^D(z)=0$ and $\cos_{\m}^D(z)=\cos_{\m}^N(z)= \pm 1$ and it follows that
\[\s_{\l,\m}(z,1-x) = \mp \s_{\l,\m}(z,x).\]
Thus, we have the following proposition.
\begin{pro}
	Let $\m$ be symmetric. Then the Dirichlet eigenfunctions $f_{D,m}$ are either symmetric or antisymmetric, that is, either
	\[f_{D,m}(x)=f_{D,m}(1-x) \qquad \text{or} \qquad f_{D,m}(x) = - f_{D,m}(1-x)\]
	for all $x\in [0,1]$.
\end{pro}

\section{Self-similar measures}
\label{sec:ss}

In this section we impose that the measure $\m$ has a self-similar structure. For definitions of the concept of iterated function systems and self-similar measures, see Hutchinson \cite{hutchinson:fractals}. For reasons of simplicity, we take an IFS consisting only of two mappings, but it should not raise considerable problems to generalize this to an arbitrary number.

Let $r_1$, $r_2$, $m_1$ and $m_2$ be positive numbers satisfying  $r_1+r_2 \leq 1$ and $m_1+m_2=1$. Let $\S=(S_1, S_2)$ be the IFS given by
\[ S_1(x)=r_1 x \qquad \text{and} \qquad S_2(x)=r_2 x + 1 - r_2, \qquad x\in[0,1].\]

By $K$ we denote the invariant set of $\S$  and by $\m$ its invariant measure with vector of weights $(m_1,m_2)$. 

In this case we are able to prove several properties of the functions $p_n(x)$ and $q_n(x)$ that resemble corresponding ones of $\frac{x^n}{n!}$. These we will employ to examine the Neumann and Dirichlet eigenfunctions and eigenvalues of $-\frac{d}{d\m} \frac{d}{dx}$. In particular, we will develop a recursion law for $p_n(1)$ and $q_n(1)$.

The self-similar structure of the measure can be used in integral transformations to receive derivation rules like the following.
\begin{lem}\label{lem1selfsim}
 Let $F\in H^1(\m)$ and $f=\frac{dF}{d\m}$. Then
\[ \frac{d}{d\m}F(r_1 x) = m_1 f(r_1 x)\]
and
\[ \frac{d}{d\m}F(1-r_2+r_2 x) = m_2 f(1-r_2+r_2 x).\]
\end{lem}
\begin{proof}
	Since $F\in H^1(\m)$, it can be written as
	\[ F(r_1 x) = F(0) + \int_0^{r_1 x} f(t)\,d\m (t).\]
	The measure $\m$ is invariant with respect to $S_1$ and $S_2$ which means that
	\[ \m = m_1 (S_1 \m) + m_2 (S_2 \m).\]
	Consequently, if restricted to $[0,r_1]$, we have
	\[ \m = m_1 (S_1 \m),\]
	and hence
	\[ F(r_1 x) = F(0) + \int_0^{r_1 x} m_1 f(t)\,d(S_1\m)(t) = F(0) + \int_0^x m_1 f(r_1 t) \,d\m(t).\]
	Thus, the first assertion follows.
	
	Analogously, it follows that on $[1-r_2,1]$ we have
	\[ \m = m_2 (S_2 \m) \]
	and thus,
	\begin{align*}
	F(1-r_2+r_2 x)	&	= F(0)+\int_0^{1-r_2} f(t)\,d\m(t)+\int_{1-r_2}^{1-r_2+r_2x} f(t)\,d\m(t)\\
									&	= F(1-r_2)+\int_{1-r_2}^{1-r_2+r_2x} m_2 f(t)\,d(S_2\m)(t)\\
									& = F(1-r_2)+\int_0^x m_2 f(1-r_2+r_2t)\,d\m(t),
	\end{align*}
	which proves the second assertion.
\end{proof}

In the following proposition we present a formula that can be viewed as an analogue of the binomial theorem, adapted to the self-similar measure $\m$. It relates values on the left part of $K$, contained in $[0,r_1]$, to values on the right part, contained in $[1-r_2,1]$.

\begin{pro}\label{pro:pformel}
	For $x \in [0,1]$ and $n \in \N_0$,
	\begin{equation}
	\begin{split}
	p_{2n+1}(1-r_2 + r_2 x)	&= \sum_{i=0}^n p_{2i+1}(r_1) (\tfrac{r_2 m_2}{r_1 m_1})^{n-i} q_{2n-2i}(r_1 x) \\
	& + \sum_{i=0}^n p_{2i}(r_1)(\tfrac{r_2}{r_1})^{n-i}(\tfrac{m_2}{m_1})^{n-i+1} p_{2n-2i+1}(r_1 x) \\ 
	&  \hspace{-2.2cm} + [1-(r_1+r_2)] \sum_{i=0}^{n-1} p_{2i+1}(r_1)(\tfrac{r_2}{r_1})^{n-i-1} (\tfrac{m_2}{m_1})^{n-i}p_{2n-2i-1}(r_1 x),
	\end{split}
	\end{equation}
	where a sum from $0$ to $-1$ is regarded as zero, and, for $n \in \N$,
	\begin{equation}\label{gerpselfsim}
	\begin{split}
	p_{2n}(1-r_2 + r_2 x)	&	= \sum_{i=0}^n p_{2i}(r_1) (\tfrac{r_2 m_2}{r_1 m_1})^{n-i} p_{2n-2i}(r_1 x) \\
	&	+ \sum_{i=0}^{n-1} p_{2i+1}(r_1)(\tfrac{r_2}{r_1})^{n-i}(\tfrac{m_2}{m_1})^{n-i-1} q_{2n-2i-1}(r_1 x) \\
	& \hspace{-2.2cm} + [1-(r_1+r_2)] \sum_{i=0}^{n-1} p_{2i+1}(r_1) (\tfrac{r_2 m_2}{r_1 m_1})^{n-i-1}p_{2n-2i-2}(r_1 x).
	\end{split}
	\end{equation}
\end{pro}
\begin{rem}
	If $r_1=m_1$ and $r_2=m_2$ and $r_1+r_2=1$ (and hence, $\m$ is the Lebesgue measure), the above formulas reduce to 
	\[ \Big(r_1+r_2 x\Big)^n=\sum_{i=0}^n \binom{n}{i} \, r_1^i \, (r_2 x)^{n-i}, \qquad  n \in \N.\]
\end{rem}
\begin{proof}
	We prove the proposition by induction. As seen in the proof of Lemma \ref{lem1selfsim} we have $\m= m_1 (S_1\m)$ on $[0,r_1]$ and
	$\m = m_2 (S_2 \m) $ on $[1-r_2,1]$.	
	Therefore, 
	\begin{align*}
	p_1(1-r_2 + r_2 x)	&	= \int_0^{1-r_2+ r_2 x} d\m		= \int_0^{r_1} d\m + \int_{1-r_2}^{1 - r_2 + r_2 x}d\m \\
											& = p_1(r_1) + m_2 \int_{1-r_2}^{1 - r_2 + r_2 x} d(S_2\m)  = p_1(r_1) + m_2 \int_0^x d\m \\
											& = p_1(r_1) + m_2 \int_0^{r_1 x} d(S_1\m) = p_1(r_1) + \frac{m_2}{m_1} \int_0^{r_1 x} d\m \\ 
											& = p_1(r_1) + \frac{m_2}{m_1} p_1(r_1 x),
	\end{align*}
	which proves the assertion for $p_1$.
	
	Assume that the formula for $p_{2n+1}$ holds for some $n\in \N_0$. Then
	\begin{align*}
		& p_{2n+2}(1-r_2+ r_2 x)	= \int_0^{r_1} p_{2n+1}(t)\,dt+\int_{r_1}^{1-r_2} p_{2n+1}(t)\,dt + \int_{1-r_2}^{1-r_2+ r_2 x} p_{2n+1}(t)\,dt\\
		& \quad = p_{2n+2}(r_1) + [1-(r_1+r_2)]p_{2n+1}(r_1) + r_2 \int_0^x p_{2n+1}(1-r_2+ r_2 t)\,dt.
	\end{align*}
	Applying the induction hypothesis, we get
	\begin{align*}
	& p_{2n+2}(1-r_2+ r_2 x) = p_{2n+2}(r_1) + [1-(r_1+r_2)]p_{2n+1}(r_1)\\
	& \qquad + \sum_{i=0}^n p_{2i+1}(r_1) (\tfrac{r_2 m_2}{r_1 m_1})^{n-i} r_2 \int_0^x q_{2n-2i}  (r_1 t)\,dt \\
	& \qquad + \sum_{i=0}^n p_{2i}(r_1)(\tfrac{r_2}{r_1})^{n-i}(\tfrac{m_2}{m_1})^{n-i+1} r_2 \int_0^x p_{2n-2i+1}(r_1 t) \,dt \\ 
	& \qquad + [1-(r_1+r_2)] \sum_{i=0}^{n-1} p_{2i+1}(r_1)(\tfrac{r_2}{r_1})^{n-i-1} (\tfrac{m_2}{m_1})^{n-i} r_2 \int_0^x p_{2n-2i-1}(r_1 t)\, dt \\
	& = p_{2n+2}(r_1) + [1-(r_1+r_2)]p_{2n+1}(r_1) \\
	& \qquad + \sum_{i=0}^n p_{2i+1}(r_1) (\tfrac{r_2}{r_1})^{n-i+1}(\tfrac{m_2}{m_1})^{n-i} q_{2n-2i+1}(r_1 x) \\
	& \qquad + \sum_{i=0}^n p_{2i}(r_1) (\tfrac{r_2 m_2}{r_1 m_1})^{n-i+1} p_{2n-2i+2}(r_1 x) \\
	&	\qquad + [1-(r_1+r_2)] \sum_{i=0}^{n-1} p_{2i+1}(r_1) (\tfrac{r_2 m_2}{r_1 m_1})^{n-i}p_{2n-2i}(r_1 x)\\
	& = \sum_{i=0}^{n+1} p_{2i}(r_1) (\tfrac{r_2 m_2}{r_1 m_1})^{n-i+1} p_{2n-2i+2}(r_1 x)+\sum_{i=0}^n p_{2i+1}(r_1) (\tfrac{r_2}{r_1})^{n-i+1}(\tfrac{m_2}{m_1})^{n-i} q_{2n-2i+1}(r_1 x)\\
	& \quad + [1-(r_1+r_2)] \sum_{i=0}^{n} p_{2i+1}(r_1) (\tfrac{r_2 m_2}{r_1 m_1})^{n-i}p_{2n-2i}(r_1 x),
	\end{align*}
	which is the formula for $p_{2n+2}$.
	
	Furthermore, suppose that the assertion is true for $p_{2n}$ for some $n\in\N$. Then, transforming $\m$ as in the proof of the initial step
	and applying the induction hypothesis in the same way	as above,
	\begin{align*}
	& p_{2n+1}(1-r_2+r_2 x)		= \int_0^{r_1} p_{2n}(t)\,d\m(t)+\int_{r_1}^{1-r_2} p_{2n}(t)\,d\m(t) + \int_{1-r_2}^{1-r_2+r_2 x} p_{2n}(t)\,d\m(t)\\
	& = p_{2n+1}(r_1) + m_2 \int_0^{x} p_{2n}(1-r_2+ r_2 t)\,d\m(t)\\
	& = p_{2n+1}(r_1) +\sum_{i=0}^n p_{2i}(r_1) (\tfrac{r_2}{r_1})^{n-i} (\tfrac{m_2}{m_1})^{n-i+1} p_{2n-2i+1}(r_1 x) + \sum_{i=0}^{n-1} p_{2i+1}(r_1) (\tfrac{r_2m_2}{r_1 m_1})^{n-i} q_{2n-2i}(r_1 x) \\ 
	& \quad + [1-(r_1+r_2)]\sum_{i=0}^{n-1} p_{2i+1}(r_1) (\tfrac{r_2}{r_1})^{n-i-1}(\tfrac{m_2}{m_1})^{n-i}p_{2n-2i-1}(r_1 x)\\
	& =\sum_{i=0}^{n} p_{2i+1}(r_1) (\tfrac{r_2 m_2}{r_1 m_1})^{n-i} q_{2n-2i}(r_1 x) + \sum_{i=0}^n p_{2i}(r_1) (\tfrac{r_2}{r_1})^{n-i} (\tfrac{m_2}{m_1})^{n-i+1} p_{2n-2i+1}(r_1 x) \\
	& \quad +  [1-(r_1+r_2)] \sum_{i=0}^{n-1} p_{2i+1}(r_1)(\tfrac{r_2}{r_1})^{n-i-1} (\tfrac{m_2}{m_1})^{n-i}p_{2n-2i-1}(r_1 x),
	\end{align*}
	which is the formula for $p_{2n+1}$.
\end{proof}

Analogous formulas hold for the functions $q_n$.
\begin{pro}\label{pro:qformel}
	For $x \in [0,1]$ and $n \in \N_0$,
	\begin{equation}\label{ungerqselfsim}
	\begin{split}
	q_{2n+1}(1-r_2 + r_2 x)	&	= \sum_{i=0}^n q_{2i+1}(r_1) (\tfrac{r_2 m_2}{r_1 m_1})^{n-i} p_{2n-2i}(r_1 x) \\
	&	+ \sum_{i=0}^n q_{2i}(r_1)(\tfrac{r_2}{r_1})^{n-i+1}(\tfrac{m_2}{m_1})^{n-i} q_{2n-2i+1}(r_1 x) \\
	& \hspace{-2.2cm} + [1-(r_1+r_2)] \sum_{i=0}^{n} q_{2i}(r_1) (\tfrac{r_2 m_2}{r_1 m_1})^{n-i}p_{2n-2i}(r_1 x),
	\end{split}
	\end{equation}
	and, for $n \in \N$,
	\begin{equation}
	\begin{split}
	q_{2n}(1-r_2 + r_2 x)	&	= \sum_{i=0}^n q_{2i}(r_1) (\tfrac{r_2 m_2}{r_1 m_1})^{n-i} q_{2n-2i}(r_1 x) \\
	& + \sum_{i=0}^{n-1} q_{2i+1}(r_1)(\tfrac{r_2}{r_1})^{n-i-1}(\tfrac{m_2}{m_1})^{n-i} p_{2n-2i-1}(r_1 x) \\ 
	& \hspace{-2.2cm} + [1-(r_1+r_2)] \sum_{i=0}^{n-1} q_{2i}(r_1) (\tfrac{r_2}{r_1})^{n-i-1}(\tfrac{m_2}{m_1})^{n-i}p_{2n-2i-1}(r_1 x).
	\end{split}
	\end{equation}
\end{pro}
\begin{proof}
	The proof works by induction analogously to that of Proposition \ref{pro:pformel}.
\end{proof}

We translate the formulas about the functions $p_n(x)$ and $q_n(x)$ into formulas about $\c_{\l,\m}(z,x)$ and $\s_{\l,\m}(z,x)$. In the Lebesgue case, these are the usual addition theorems for $\cos ( r_1 z + r_2 x z) $ and $\sin(r_1z + r_2 x z)$.
\begin{cor}
	Let $z\in \R$ and $x\in [0,1]$. With the abbreviation $\bar{z}:= \sqrt{\frac{r_2 m_2}{r_1 m_1}}z$ we get
	\begin{equation}\label{lrcpz}
	\begin{split}
		\c_{\l,\m}(z,1-r_2+r_2x)	&	= \c_{\l,\m}(z,r_1) \c_{\l,\m}(\bar{z}, r_1 x)-\sqrt{\tfrac{r_2m_1}{r_1m_2}}\s_{\m,\l}(z,r_1) \s_{\l,\m}(\bar{z}, r_1x)\\
											& \quad - [1-(r_1+r_2)]z\s_{\m,\l}(z,r_1)\c_{\l,\m}(\bar{z}, r_1 x)
	\end{split}
	\end{equation}
	and
	\begin{equation}\label{lrsqz}
	\begin{split}
	\s_{\l,\m}(z,1-r_2+r_2x)	&	= \s_{\l,\m}(z,r_1) \c_{\l,\m}(\bar{z}, r_1 x)+ \sqrt{\tfrac{r_2m_1}{r_1m_2}}\c_{\m,\l}(z,r_1) \s_{\l,\m}(\bar{z}, r_1 x)\\
										& \quad +[1-(r_1+r_2)]z\c_{\m,\l}(z,r_1)\c_{\l,\m}(\bar{z}, r_1 x).
	\end{split}
	\end{equation}
\end{cor}
\begin{proof}
	We prove \eqref{lrsqz}. We multiply \eqref{ungerqselfsim} with $(-1)^n z^{2n+1}=\frac{1}{i}(iz)^{2n+1}$, sum from $n=0$ to infinity and get
	\begin{align*}
	&\s_{\l,\m}(z,1-r_2+r_2x) = \frac{1}{i} \sum_{n=0}^\infty \sum_{k=0}^n (iz)^{2k+1} q_{2k+1}(r_1)\,\Big(i\sqrt{\tfrac{r_2m_2}{r_1m_1}}z\Big)^{2n-2k}p_{2n-2k}(r_1x)\\
	& \quad +\sqrt{\tfrac{r_2m_1}{r_1m_2}}\frac{1}{i}\sum_{n=0}^\infty \sum_{k=0}^n(iz)^{2k} q_{2k}(r_1)\,\Big(i\sqrt{\tfrac{r_2m_2}{r_1m_1}}z\Big)^{2n-2k+1}q_{2n-2k+1}(r_1x)\\
	&	\quad	 + [1-(r_1+r_2)]z\sum_{n=0}^\infty \sum_{k=0}^n (iz)^{2k}q_{2k}(r_1)\,\Big(i \sqrt{\tfrac{r_2m_2}{r_1m_1}}z\Big)^{2n-2k} p_{2n-2k}(r_1x)\\
	&  = \frac{1}{i} \bigg(\sum_{n=0}^\infty (iz)^{2n+1} q_{2n+1}(r_1)\bigg) \bigg( \sum_{k=0}^\infty \Big(i\sqrt{\tfrac{r_2m_2}{r_1m_1}}z\Big)^{2k}p_{2k}(r_1x)\bigg)\\
	& \quad +\sqrt{\tfrac{r_2m_1}{r_1m_2}}\frac{1}{i}\bigg(\sum_{n=0}^\infty (iz)^{2n} q_{2n}(r_1) \bigg)\bigg( \sum_{k=0}^\infty \Big(i\sqrt{\tfrac{r_2m_2}{r_1m_1}}z\Big)^{2k+1}q_{2k+1}(r_1x)\bigg) \\
	& \quad +  [1-(r_1+r_2)]z \bigg(\sum_{n=0}^\infty (iz)^{2n}q_{2n}(r_1)\bigg) \bigg(\sum_{k=0}^\infty \Big(i \sqrt{\tfrac{r_2m_2}{r_1m_1}}z\Big)^{2k} p_{2k}(r_1x)\bigg)\\
	&	 = \s_{\l,\m}(z,r_1) \c_{\l,\m}(\bar{z}, r_1 x)+\sqrt{\tfrac{r_2m_1}{r_1m_2}}\c_{\m,\l}(z,r_1) \s_{\l,\m}(\bar{z}, r_1x) + [1-(r_1+r_2)]z\c_{\m,\l}(z,r_1)\c_{\l,\m}(\bar{z}, r_1 x).
	\end{align*}
By multiplying \eqref{gerpselfsim} with $(-1)^n z^{2n}$ and summing up, \eqref{lrcpz} is proved in the same way.
\end{proof}

The following scaling properties hold that are a replacement of the property $\bigl(\frac{1}{2}x)^n = \frac{1}{2^n}x^n$  for $p_n$ and $q_n$.
\begin{pro}\label{pqscaling}
For $x\in[0,1]$ and $n\in \N_0$ we have
\begin{align*}
p_{2n+1}\big( r_1 x \big)		&	=	r_1^n m_1^{n+1} \, p_{2n+1}(x),	& q_{2n+1}\big( r_1 x \big)		&	=	r_1^{n+1} m_1^{n}\, q_{2n+1}(x),\\
\intertext{and, {for $n\in \N$,}}
p_{2n}\big( r_1 x \big)		&	=	(r_1 m_1)^{n}\, p_{2n}(x),	&	 q_{2n}\big( r_1 x \big)		&	=	(r_1 m_1)^{n}\, q_{2n}(x).
\end{align*}
\end{pro}
\begin{proof}
	We prove the asserted property for $p_n$ by induction on $n\in \N$. Since $\m$ satisfies $\m (B)=m_1 (S_1\m)(B)$ for all Borel sets $B \sub [0,r_1 ]$, we have
	\[ p_1(r_1 x) = \int_0^{r_1 x} d\m = m_1 \int_0^{r_1 x} d(S_1 \m) = m_1 \int_0^x d \m =m_1 p_1(x).\]
	Suppose the assertion is true for $p_{2n+1}$ for some $n\in \N_0$. Then
	\[	p_{2n+2}(r_1 x)=\int_0^{r_1 x} p_{2n+1}(t)\,dt = r_1 \int_0^x p_{2n+1}(r_1 t)\,dt	= (r_1 m_1)^{n+1}p_{2n+2}(x).\]
	If we assume that the formula holds for $p_{2n}$ for some $n\in\N$, then, transforming $\m$ as above,
	\[ p_{2n+1}(r_1 x)=\int_0^{r_1 x} p_{2n}(t)\,d\m(t) = m_1 \int_0^x p_{2n}(r_1 t)\,d\m (t)= r_1^n m_1^{n+1} p_{2n+1}(x).\]
	The formula for $q_n$ is proved analogously.
	\end{proof}

Next, we deduce formulas corresponding to those in Proposition \ref{pqscaling} that relate values of $\c_{\l,\m}(z,\cdot)$ and $\s_{\l,\m}(z,\cdot)$ at $S_1(x)=r_1x$ to values of $\c_{\l,\m}\bigl( \sqrt{r_1m_1}z, \cdot \bigr)$ and $\s_{\l,\m}\bigl(\sqrt{r_1m_1}z, \cdot\bigr)$ at $x$.
\begin{pro}\label{cpzsqzS1}
	For all $x\in [0,1]$ and $z\in\R$ we have
	\begin{equation}\label{cpzS1}
	\c_{\l,\m}\bigl(z,S_1(x)\bigr) = \c_{\l,\m}\bigl( \sqrt{r_1m_1}z, x\bigr)
	\end{equation}
	and
	\begin{equation}\label{sqzS1}
	\s_{\l,\m}\big(z, S_1(x)\big) = \sqrt{\frac{r_1}{m_1}} \s_{\l,\m}\bigl( \sqrt{r_1m_1}z, x \bigr).
	\end{equation}
	Furthermore, we have
	\[ \s_{\m,\l}\bigl(z, S_1(x)\bigr) =  \sqrt{\frac{m_1}{r_1}} \s_{\m,\l}\bigl(\sqrt{r_1m_1}z,x\bigr)\]
	and
	\[ \c_{\m,\l}\bigl(z, S_1(x)\bigr) = \c_{\l,\m}\bigl( \sqrt{r_1m_1}z, x\bigr).\]
\end{pro}
\begin{proof}
	With Proposition \ref{pqscaling} we get
	\[ \c_{\l,\m}(z,r_1 x) = \sum_{n=0}^\infty (-1)^n z^{2n} p_{2n}(r_1x) = \sum_{n=0}^\infty (-1)^n (\sqrt{r_1m_1}z)^{2n} p_{2n}(x)=\c_{\l,\m}\bigl( \sqrt{r_1m_1}z, x\bigr)\]
	and
	\begin{align*}
	\s_{\l,\m}(z,r_1 x) & = \sum_{n=0}^\infty (-1)^n z^{2n+1} q_{2n+1}(r_1x) = \sqrt{\frac{r_1}{m_1}}\sum_{n=0}^\infty (-1)^n (\sqrt{r_1m_1}z)^{2n+1}q_{2n+1}(x)\\
							 & = \sqrt{\frac{r_1}{m_1}} \s_{\l,\m}\bigl( \sqrt{r_1m_1}z, x \bigr).
	\end{align*}
	The other two equations are obtained by deriving.
\end{proof}

The counterparts of \eqref{cpzS1} and \eqref{sqzS1} are the following formulas for $\c_{\l,\m}\bigl(z,S_2(x)\bigr)$ and $\s_{\l,\m}\big(z,S_2(x)\big)$.
\begin{pro}\label{cpzsqzS2}
	For all $x\in [0,1]$ and $z\in\R$ we have
	\begin{equation}\label{cpzS2}
	\begin{split}
	\c_{\l,\m}\bigl(z,S_2(x)\bigr)	& = \cos_{\m}^N(\sqrt{r_1m_1}z) \c_{\l,\m}\bigl( \sqrt{r_2m_2}z, x\bigr) - \sqrt{\frac{r_2m_1}{r_1m_2}}\sin_{\m}^N(\sqrt{r_1m_1}z)
	\s_{\l,\m}\bigl( \sqrt{r_2m_2}z, x \bigr)\\
												& \quad - [1-(r_1+r_2)]\sqrt{\frac{m_1}{r_1}} z \sin_{\m}^N(\sqrt{r_1m_1}z) \c_{\l,\m}\bigl( \sqrt{r_2m_2}z, x\bigr)
	\end{split}
	\end{equation}
	and
	\begin{equation}\label{sqzS2}
	\begin{split}
	\s_{\l,\m}\big(z,S_2(x)\big)	&	= \sqrt{\frac{r_1}{m_1}} \sin_{\m}^D(\sqrt{r_1m_1}z)\c_{\l,\m}\bigl( \sqrt{r_2m_2}z, x\bigr)
														+ \sqrt{\frac{r_2}{m_2}}\cos_{\m}^D(\sqrt{r_1m_1}z) \s_{\l,\m}\bigl( \sqrt{r_2m_2}z, x \bigr)\\
												& \quad +[1-(r_1+r_2)]z\cos_{\m}^D(\sqrt{r_1m_1}z)\c_{\l,\m}\bigl( \sqrt{r_2m_2}z, x\bigr).
	\end{split}
	\end{equation}
	Furthermore, we have
	\begin{align*}
		\s_{\m,\l}\bigl(z,S_2(x)\bigr)	&	=\sqrt{\frac{m_1}{r_1}}\sin_{\m}^N(\sqrt{r_1m_1}z)\c_{\m,\l}(\sqrt{r_2 m_2}z, x)+
																			\sqrt{\frac{m_2}{r_2}}\cos_{\m}^N(\sqrt{r_1m_1}z) \s_{\m,\l}\bigl(\sqrt{r_2m_2}z,x\bigr)\\
													&	- [1-(r_1+r_2)]\sqrt{\frac{m_1m_2}{r_1r_2}}\, z \sin_{\m}^N(\sqrt{r_1 m_1}z) \s_{\m,\l}\bigl(\sqrt{r_2m_2}z,x\bigr)
	\end{align*}
	and
	\begin{align*}
		\c_{\m,\l}\bigl(z, S_2(x)\bigr) & = \cos_{\m}^D(\sqrt{r_1m_1}z)\c_{\m,\l}(\sqrt{r_2 m_2}z, x)-\sqrt{\frac{r_1 m_2}{r_2 m_1}}\sin_{\m}^D(\sqrt{r_1m_1}z)\s_{\m,\l}\bigl(\sqrt{r_2m_2}z,x\bigr)\\
													& - [1-(r_1+r_2)]\, z \sqrt{\frac{m_2}{r_2}}\cos_{\m}^D(\sqrt{r_1 m_1}z) \s_{\m,\l}\bigl(\sqrt{r_2m_2}z,x\bigr).\\
	\end{align*}
\end{pro}
\begin{proof}
	By \eqref{lrcpz} and Proposition \ref{cpzsqzS1} we get
	\begin{align*}
	\c_{\l,\m}(z,1-r_2+r_2x)	&	= \c_{\l,\m}(z,r_1) \c_{\l,\m}\bigl(\sqrt{\frac{r_2 m_2}{r_1 m_1}}z, r_1 x\bigr)-\sqrt{\tfrac{r_2m_1}{r_1m_2}}\s_{\m,\l}(z,r_1) 
	\s_{\l,\m}\bigl(\sqrt{\frac{r_2 m_2}{r_1 m_1}}z, r_1x \bigr)\\
										& \quad - [1-(r_1+r_2)]z\s_{\m,\l}(z,r_1)\c_{\l,\m}\bigl(\sqrt{\frac{r_2 m_2}{r_1 m_1}}z, r_1 x\bigr)\\
										& = \cos_{\m}^N(\sqrt{r_1m_1}z) \c_{\l,\m}\bigl( \sqrt{r_2m_2}z, x\bigr) - \sqrt{\frac{r_2m_1}{r_1m_2}}\sin_{\m}^N(\sqrt{r_1m_1}z) \s_{\l,\m}\bigl(\sqrt{r_2m_2}z, x\bigr)\\
										& \quad - [1-(r_1+r_2)]\sqrt{\frac{m_1}{r_1}} z \sin_{\m}^N(\sqrt{r_1m_1}z) \c_{\l,\m}\bigl( \sqrt{r_2m_2}z, x\bigr).
	\end{align*}
	Analogously, \eqref{sqzS2} is proved using \eqref{lrsqz}.
	
	The other two equations are obtained by deriving.
\end{proof}

If the functions $\cos_{\m}^N$, $\sin_{\m}^N$ and $\sin_{\m}^D$ are assumed to be known, then equations \eqref{cpzS1} and \eqref{cpzS2} allow to compute basically all relevant values of the function $\c_{\l,\m}(z,\cdot)$. If, namely, $x$ is a point in the invariant set $K$, then there is a sequence $(x_n)_n$ that converges to $x$ and takes only values of the form 
\[ S_{w_1}\circ S_{w_2} \circ \dotsm \circ S_{w_n} (0) \quad \text{or} \quad S_{w_1}\circ S_{w_2} \circ \dotsm \circ S_{w_n} (1),\]
where $n\in\N$ and $w_1, \dotsc w_n \in \{1,2\}$. For each of these values, \eqref{cpzS1} and \eqref{cpzS2} can be applied $n$ times to get a formula containing only values of $\cos_{\m}^N$, $\sin_{\m}^N$ and $\sin_{\m}^D$. For example
\begin{align*}
\c_{\l,\m} \big(z, S_2(S_1(1))\big)	& = \cos_{\m}^N(\sqrt{r_1m_1}z)\cos_{\m}^N(\sqrt{r_2m_2r_1m_1}z) \\
														& \quad - \sqrt{\frac{r_2m_1}{r_1m_2}}\sin_{\m}^N(\sqrt{r_1m_1}z)\sin_{\m}^D(\sqrt{r_2m_2r_1m_1}z)\\
														& \quad - [1-(r_1+r_2)]\sqrt{\frac{m_1}{r_1}}z \sin_{\m}^N(\sqrt{r_1m_1}z) \cos_{\m}^N(\sqrt{r_2m_2r_1m_1}z).
\end{align*}
The same holds for $\s_{\l,\m}$ and formulas \eqref{sqzS1} and \eqref{sqzS2}. This procedure we will use to compute approximate values of the maxima and to give plots of eigenfunctions in Section~\ref{numerical}.

Therefore we are interested in the functions $\sin_{\m}^D$, $\sin_{\m}^N$, $\cos_{\m}^N$, and $\cos_{\m}^D$. These have power series representations with coefficients
$p_n=p_n(1)$ and $q_n=q_n(1)$. For these numerical sequences we prove a recursion formula in the following.

\begin{pro}\label{pro:pqformel}
\begin{enumerate}
	\item
		For $n\in \N_0$,
		\begin{equation}\label{fo1}
		\begin{split}
		p_{2n+1} & =  \sum_{i=0}^n r_1^i m_1^{i+1}(r_2 m_2)^{n-i} p_{2i+1}\,q_{2n-2i}\\
		& + \sum_{i=0}^n (r_1m_1)^i r_2^{n-i} m_2^{n-i+1} p_{2i}\, p_{2n-2i+1}\\
		& + [1-(r_1+r_2)]\sum_{i=0}^{n-1} r_1^i m_1^{i+1} r_2^{n-i-1}m_2^{n-i} p_{2i+1}\, p_{2n-2i-1}.
		\end{split}
		\end{equation}
	\item
		For $n\in \N$,
		\begin{equation}\label{fo2}
		\begin{split}
		p_{2n}&=\sum_{i=0}^n (r_1 m_1)^{i} (r_2 m_2)^{n-i} p_{2i}\,p_{2n-2i}\\
		&+ \sum_{i=0}^{n-1} r_1^i m_1^{i+1} r_2^{n-i} m_2^{n-i-1} p_{2i+1}\, q_{2n-2i-1}\\
		&+ [1-(r_1+r_2)] \sum_{i=0}^{n-1} r_1^i m_1^{i+1}  (r_2 m_2)^{n-i-1} p_{2i+1}\, p_{2n-2i-2}.
		\end{split}
		\end{equation}
	\item
		For $n\in \N_0$,
		\begin{equation}\label{fo3}
		\begin{split}
		q_{2n+1}&=\sum_{i=0}^n r_1^{i+1}m_1^{i}(r_2m_2)^{n-i} q_{2i+1}\,p_{2n-2i} \\
		& + \sum_{i=0}^{n} (r_1m_1)^{i} r_2^{n-i+1}m_2^{n-i} q_{2i}\, q_{2n-2i+1}\\
		& + [1-(r_1+r_2)] \sum_{i=0}^{n} (r_1m_1)^{i} (r_2m_2)^{n-i} q_{2i}\, p_{2n-2i}.
		\end{split}
		\end{equation}
	\item
	For $n\in \N$,
	\begin{equation}\label{fo4}
		\begin{split}
		q_{2n}&=\sum_{i=0}^n (r_1m_1)^{i} (r_2m_2)^{n-i} q_{2i}\,q_{2n-2i}\\
		&+ \sum_{i=0}^{n-1} r_1^{i+1}m_1^{i} r_2^{n-i-1}m_2^{n-i} q_{2i+1}\, p_{2n-2i-1}\\
		&+ [1-(r_1+r_2)] \sum_{i=0}^{n-1} (r_1m_1)^{i} r_2^{n-i-1} m_2^{n-i} q_{2i}\, p_{2n-2i-1}.
		\end{split}
		\end{equation}
\end{enumerate}
\end{pro}
\begin{rem}
	If we take $r_1=m_1$ and $r_2=m_2$ (and thus $r_1+r_2=1$ and $\m$ is the Lebesgue measure), the above formulas reduce to $\sum_{i=0}^n \binom{n}{i} r_1^i r_2^{n-i} =1$.
\end{rem}
\begin{proof}
We put $x=1$ in Propositions \ref{pro:pformel}, \ref{pro:qformel}, and \ref{pqscaling}. Then we eliminate all terms  of the form $p_n(r_1)$ and $q_n(r_1)$ to obtain formulas that contain only the members of the sequences $(p_n)_n$ and $(q_n)_n$ (as well as $r_1$, $r_2$, $m_1$ and $m_2$).
\end{proof}

To get the desired recursion formulas, we solve the above formulas for the highest order terms.
\begin{cor}\label{rec}
\begin{enumerate}
	\item
		For $n\in \N$,
		\begin{equation}\label{rec1}
		\begin{split}
		p_{2n+1}= & \frac{1}{1-r_1^n m_1^{n+1} - r_2^n m_2^{n+1}}\bigg( \sum_{i=0}^{n-1}r_1^i m_1^{i+1}(r_2 m_2)^{n-i} p_{2i+1}\,q_{2n-2i} \\
		& + \sum_{i=1}^n (r_1m_1)^i r_2^{n-i} m_2^{n-i+1} p_{2i}\, p_{2n-2i+1}	\\
		& + [1-(r_1+r_2)]\sum_{i=0}^{n-1} r_1^i m_1^{i+1} r_2^{n-i-1}m_2^{n-i} p_{2i+1}\, p_{2n-2i-1}\bigg).
		\end{split}
		\end{equation}
	\item
		For $n\in \N$,
		\begin{equation}\label{rec2}
		\begin{split}
		p_{2n}= & \frac{1}{1-(r_1 m_1)^n-(r_2 m_2)^n}\bigg(\sum_{i=1}^{n-1}  (r_1 m_1)^{i} (r_2 m_2)^{n-i} p_{2i}\,p_{2n-2i}\\
		& + \sum_{i=0}^{n-1} r_1^i m_1^{i+1} r_2^{n-i} m_2^{n-i-1} p_{2i+1}\, q_{2n-2i-1}\\
		& +  [1-(r_1+r_2)] \sum_{i=0}^{n-1} r_1^i m_1^{i+1}  (r_2 m_2)^{n-i-1} p_{2i+1}\, p_{2n-2i-2} \bigg).
		\end{split}
		\end{equation}
	\item
		For $n\in \N$,
		\begin{equation}\label{rec3}
		\begin{split}
		q_{2n+1} = &\frac{1}{1-r_1^{n+1}m_1^n-r_2^{n+1}m_2^n}\bigg( \sum_{i=0}^{n-1} r_1^{i+1}m_1^{i}(r_2m_2)^{n-i} q_{2i+1}\,p_{2n-2i}\\
		&+ \sum_{i=1}^{n} (r_1m_1)^{i} r_2^{n-i+1}m_2^{n-i} q_{2i}\, q_{2n-2i+1}\\
		&+  [1-(r_1+r_2)] \sum_{i=0}^{n} (r_1m_1)^{i} (r_2m_2)^{n-i} q_{2i}\, p_{2n-2i} \bigg).
		\end{split}
		\end{equation}
	\item
	For $n\in \N$,
	\begin{equation}\label{rec4}
		\begin{split}
		q_{2n} = & \frac{1}{1- (r_1 m_1)^n - (r_2 m_2)^n}\bigg(\sum_{i=1}^{n-1} (r_1m_1)^{i} (r_2m_2)^{n-i} q_{2i}\,q_{2n-2i} \\
		& + \sum_{i=0}^{n-1} r_1^{i+1}m_1^{i} r_2^{n-i-1}m_2^{n-i} q_{2i+1}\, p_{2n-2i-1}\\
		& + [1-(r_1+r_2)] \sum_{i=0}^{n-1} (r_1m_1)^{i} r_2^{n-i-1} m_2^{n-i} q_{2i}\, p_{2n-2i-1} \bigg).
		\end{split}
		\end{equation}
\end{enumerate}
\end{cor}

\begin{rem}
	Consider two self-similar measures $\m$ and $\m^*$ on $[0,1]$, where $\m^*$ is the reflection of $\m$ with respect to the point $\frac{1}{2}$. Thus, $\m^*$ is described as invariant measure by interchanging the parameters $r_1$, $m_1$ and $r_2$, $m_2$ in the IFS defining $\m$.
	Then the above recursive formulas show that the associated $p$- and $q$-sequences satisfy $p^*_{2n}=q_{2n}$, $q^*_{2n}=p_{2n}$, $p^*_{2n+1}=p_{2n+1}$ and $q^*_{2n+1}=q_{2n+1}$ for all $n\in \N$. Hence, ${\cos_{\m^*}^N}=\cos_{\m}^D$, ${\cos_{\m^*}^D}=\cos_{\m}^N$, ${\sin_{\m^*}^N}=\sin_{\m}^N$ and ${\sin_{\m^*}^D}=\sin_{\m}^D$. This is consistent with the physical intuition that the Neumann as well as the Dirichlet eigenfrequencies do not change when the vibrating string producing them is reversed.
\end{rem}

\begin{exa}
	We take $r_1=r_2=\frac{1}{3}$ and $m_1=m_2=\frac{1}{2}$. Then, $K$ is the middle third Cantor set and $\m$ is the normalized $\frac{\log 2}{\log 3}$-dimensional
	Hausdorff measure restricted to $K$. We calculate the first members of the sequences $(p_n)_n$ and $(q_n)_n$ using formulas \eqref{rec1} and \eqref{rec3}
	for $p_{2n+1}$ and $q_{2n+1}$, which simplify to
	\begin{align*}
	p_{2n+1}	&	= \frac{1}{2 \cdot 6^n - 2}\Big(\sum_{i=1}^{2n} p_i \, p_{2n+1-i} + \sum_{i=0}^{n-1} p_{2i+1}\,p_{2n-2i-1}\Big)\\
	q_{2n+1}	&	= \frac{1}{3 \cdot 6^n - 2}\Big(\sum_{i=1}^{2n} q_i \, q_{2n+1-i} + \sum_{i=0}^n q_{2i}\,q_{2n-2i}\Big).
	\end{align*}
	Since $\m$ is symmetric, we can use for $p_{2n}$ and $q_{2n}$ the simpler formula \eqref{evreceq} 
	\[p_{2n}=q_{2n}= \frac{1}{2}\sum_{i=1}^{2n-1} (-1)^{i+1} p_i\, q_{2n-i}\]
	from Corollary \ref{cor:symm}. Then,
	\begin{align*}
	p_1	&	=	1,													&	q_1	&	= 1,														& p_2	&	=\frac{1}{2},\\
	p_3	&	=\frac{1}{5},									& q_3	&	= \frac{1}{8},									& p_4	&	=\frac{3}{80}, \\
	p_5	& =\frac{27}{2\,800},						& q_5	&	= \frac{21}{4\,240},						& p_6	&	=\frac{311}{296\,800},\\
	p_7	&	=\frac{6\,383}{31\,906\,000},	&	q_7	&	=\frac{33\,253}{383\,465\,600}, & p_8	& = \frac{4\,716\,349}{329\,780\,416\,000}
	\end{align*}
	and therefore
	\begin{align*}
	\sin_{\m}^N (z)	&	= z - \frac{6}{5}  \frac{z^3}{3!} + \frac{81}{70}  \frac{z^5}{5!} - \frac{57\,447}{56\,975} \frac{z^7}{7!}+\cdots  \\[2mm]
	\sin_{\m}^D (z)	&	= z - \frac{3}{4}\frac{z^3}{3!} + \frac{63}{106} \frac{z^5}{5!} - \frac{299\,277}{684\,760} \frac{z^7}{7!} + \cdots
	\end{align*}
	and
	\[ \cos_{\m}^N(z)=\cos_{\m}^D(z)= 1 - \frac{z^2}{2!} + \frac{9}{10} \frac{z^4}{4!} - \frac{2\,799}{3\,710} \frac{z^6}{6!} + \frac{42\,447\,141}{73\,611\,700} \frac{z^8}{8!} - \dots .\]
\end{exa}
More values of the sequences $p_n$ and $q_n$, plots of $\sin_{\m}^N$ and $\sin_{\m}^D$ as well as further examples can be found in Section \ref{numerical}.

The functions $\sin_{\m}^N$, $\sin_{\m}^D$, $\cos_{\m}^N$ and $\cos_{\m}^D$ can be characterized by the following system of functional equations.

\begin{theo}\label{th1}
For $z\in\R$ we have
\begin{align}
\label{th11}
\begin{split}
	\ts \sin_{\m}^N(z)	 &= \ts \sqrt{\frac{m_1}{r_1}}\sin_{\m}^N (\sqrt{r_1m_1}z) \cos_{\m}^D (\sqrt{r_2 m_2}z) \\
	&	\ts \quad + \sqrt{\frac{m_2}{r_2}}\cos_{\m}^N( \sqrt{r_1m_1}z) \sin_{\m}^N (\sqrt{r_2m_2}z)\\
	&	\ts \quad - [1-(r_1+r_2)]\sqrt{\frac{m_1m_2}{r_1r_2}}\, z \sin_{\m}^N( \sqrt{r_1 m_1}z) \sin_{\m}^N (\sqrt{r_2 m_2}z)
\end{split}\\
\label{th12}
\begin{split}
	\ts \sin_{\m}^D(z)	 &= \ts \sqrt{\frac{r_1}{m_1}}\sin_{\m}^D (\sqrt{r_1m_1}z) \cos_{\m}^N (\sqrt{r_2 m_2}z) \\
	& \ts \quad + \sqrt{\frac{r_2}{m_2}}\cos_{\m}^D( \sqrt{r_1m_1}z) \sin_{\m}^D (\sqrt{r_2m_2}z)\\
	&	\ts \quad + [1-(r_1+r_2)]\, z \cos_{\m}^D( \sqrt{r_1 m_1}z) \cos_{\m}^N (\sqrt{r_2 m_2}z)
\end{split}
\end{align}
\begin{align}
\label{th13}
\begin{split}
	\ts \cos_{\m}^N(z)	& = \ts \cos_{\m}^N (\sqrt{r_1m_1}z) \cos_{\m}^N (\sqrt{r_2 m_2}z)\\
	&	\ts \quad - \sqrt{\frac{r_2 m_1}{r_1 m_2}}\sin_{\m}^N( \sqrt{r_1m_1}z) \sin_{\m}^D (\sqrt{r_2m_2}z)\\
	&	\ts	\quad - [1-(r_1+r_2)]\sqrt{\frac{m_1}{r_1}} z \sin_{\m}^N( \sqrt{r_1 m_1}z) \cos_{\m}^N (\sqrt{r_2 m_2}z)
\end{split}\\
\label{th14}
\begin{split}
	\ts \cos_{\m}^D(z)	 &= \ts \cos_{\m}^D (\sqrt{r_1m_1}z) \cos_{\m}^D (\sqrt{r_2 m_2}z)\\
	&	\ts \quad - \sqrt{\frac{r_1 m_2}{r_2 m_1}}\sin_{\m}^D( \sqrt{r_1m_1}z) \sin_{\m}^N (\sqrt{r_2m_2}z)\\
	&	\ts \quad - [1-(r_1+r_2)] \sqrt{\frac{m_2}{r_2}} z \cos_{\m}^D( \sqrt{r_1 m_1}z) \sin_{\m}^N (\sqrt{r_2 m_2}z).
\end{split}
\end{align}
Furthermore, the functions $\sin_{\m}^N$, $\sin_{\m}^D$, $\cos_{\m}^N$ and $\cos_{\m}^D$ are the only analytic functions that solve the above system of functional equations and satisfy the
conditions that $\sin_{\m}^N$ and $\sin_{\m}^D$ are odd, $\cos_{\m}^N$ and $\cos_{\m}^D$ are even, and
\[ \lim_{z\to 0} \frac{\sin_{\m}^N(z)}{z} = \lim_{z\to 0} \frac{\sin_{\m}^D(z)}{z} = 1\]
and
\[ \cos_{\m}^N(0)=\cos_{\m}^D(0)=1.\]
\end{theo}

\begin{rem}
If we would know all the values of all four functions on a given interval, say, $[0,a]$, then,  using the formulas above,  we could calculate all values of all four functions on $[0,(\max_i \sqrt{r_i m_i})^{-1}a]$. Then, iteratively, we get the values on $[0,(\max_i \sqrt{r_i m_i})^{-2} a]$ and so on. So, the functions are determined on $[0,\infty)$ by their values on an arbitrary small interval $[0,a]$.

Furthermore, the theorem describes a kind of ``self-similarity'' of our four functions.
\end{rem}

\begin{proof}
	To show that $\sin_{\m}^N$, $\sin_{\m}^D$, $\cos_{\m}^N$ and $\cos_{\m}^D$ satisfy the equations, put $x=1$ in Proposition \ref{cpzsqzS2}.
	
	Suppose that $f_1, f_2, g_1$ and $g_2$ are real analytic functions that satisfy the above equations, and that $f_1$, $f_2$ are odd, $g_1$, $g_2$ are even, 
	$ \lim\limits_{z\to 0} \frac{f_1(z)}{z} = \lim\limits_{z\to 0} \frac{f_2(z)}{z} = 1$, and $g_1(0)=g_2(0)=1$. Then, power series representations exist, that is, there are real sequences $(a_n)$, $(b_n)$, $(c_n)$ and $(d_n)$ such that for all $z\in\R$ holds
	\begin{align*}
	f_1(z) & = \sum_{n=0}^\infty a_n z^{2n+1}, & f_2(z) & = \sum_{n=0}^\infty b_n z^{2n+1}, & g_1(z) & = \sum_{n=0}^\infty c_n z^{2n},&
	g_2(z) & = \sum_{n=0}^\infty d_n z^{2n},
	\end{align*}
	where $a_0=b_0=c_0=d_0=1$. Since these functions satisfy \eqref{th11}, we get for all $z\in\R$
	\begin{align*}
	\sum_{n=0}^\infty a_n z^{2n+1}	& = \sqrt{\frac{m_1}{r_1}} \sum_{n=0}^\infty z^{2n+1} \sum_{k=0}^n a_k \sqrt{r_1m_1}^{2k+1} d_{n-k} \sqrt{r_2m_2}^{2n-2k}\\
																	& \quad + \sqrt{\frac{m_2}{r_2}} \sum_{n=0}^\infty z^{2n+1} \sum_{k=0}^n c_k \sqrt{r_1m_1}^{2k} a_{n-k} \sqrt{r_2m_2}^{2n+1-2k}\\
																	& \quad - [1-(r_1+r_2)] \sqrt{\frac{m_1m_2}{r_1r_2}}  \sum_{n=0}^\infty z^{2n+3} \sum_{k=0}^n a_k \sqrt{r_1m_1}^{2k+1}
																						a_{n-k} \sqrt{r_2m_2}^{2n+1-2k}.
	\end{align*}
	If we derive this equation $2j+1$ times and put $z=0$, we receive formula \eqref{fo1} for $a_j$. Analogously, one can show that $b_j$ satisfies
	\eqref{fo3}, $c_j$ satisfies \eqref{fo2} and $d_j$ satisfies \eqref{fo4}. Together with the initial condition $a_0=b_0=c_0=d_0=1$ it follows that
	$a_j=p_{2j+1}$, $b_j=q_{2j+1}$, $c_j=p_{2j}$ and $d_j=q_{2j}$ for all $j\in\N$. Thus, $f_1=\sin_{\m}^N$, $f_2=\sin_{\m}^D$, $g_1=\cos_{\m}^N$ and $g_2=\cos_{\m}^D$.
\end{proof}

\begin{exa}
\begin{enumerate}
	\item
		If we take $r_1=m_1$ and $r_2=m_2$ and $r_1+r_2=1$, then $K$ is the unit interval and $\m$ the Lebesgue measure. The functions $\sin_{\m}^N$, $\sin_{\m}^D$, $\cos_{\m}^N$
		and $\cos_{\m}^D$ equal the usual sine and cosine functions, and the formulas in Theorem \ref{th1} simplify to
		\begin{align*} 
		\sin(z) & = \sin(r_1z+r_2z) = \sin(r_1z) \cos (r_2z) + \cos(r_1z) \sin(r_2z),\\
		\cos(z) & = \cos(r_1z+r_2z)  = \cos(r_1z) \cos(r_2 z) - \sin(r_1z) \sin(r_2z).
		\end{align*}
	\item
		Let $r_1=r_2=\frac{1}{3}$ and $m_1=m_2=\frac{1}{2}$.  Then $\m$ is the Cantor measure and the formulas in Theorem \ref{th1} can be rewritten as
\begin{align}
\sin_{\m}^N(\sqrt{6}z)&=\frac{\sqrt{6}}{2}\sin_{\m}^N(z)\Bigl(2\cos_{\m}^N(z) - z\sin_{\m}^N(z)\Bigr)\\
\sin_{\m}^D(\sqrt{6}z)&=\frac{\sqrt{6}}{3}\cos_{\m}^N(z)\Bigl(2\sin_{\m}^D(z) + z\cos_{\m}^N(z)\Bigr)\\
\cos_{\m}^N(\sqrt{6}z)&=\cos_{\m}^N(z)^2-\sin_{\m}^N(z)\sin_{\m}^D(z)-z\cos_{\m}^N(z)\sin_{\m}^N(z).
\end{align}
Since $K$ is symmetric, $\cos_{\m}^N=\cos_{\m}^D$. 
\end{enumerate}
\end{exa}

Observe that Theorem \ref{th1} in combination with the recursive rules in Corollary \ref{rec} supply a technique for investigation of further properties of the eigenvalues. On a given interval $[0,a]$ we can approximate the functions $\sin_{\m}^N$, $\sin_{\m}^D$, $\cos_{\m}^N$ and $\cos_{\m}^D$ arbitrarily exact by polynomials consisting of sufficiently many members of the corresponding power series. Then, by Theorem \ref{th1}, we can extend all four functions successively to larger intervals.

\section{Self-similar measures with $r_1m_1=r_2m_2$}
\label{sec:r1m1gleich}

In this section we suppose $\m$ is a self-similar measure as in the last section but with parameters additionally satisfying $r_1m_1=r_2m_2$. This case is particularly interesting because there we have the following property.

\begin{theo}\label{th:ssnev}
	Let $r_1 m_1=r_2m_2$. If $\l$ is the $m$th Neumann eigenvalue of $-\frac{d}{d\m}\frac{d}{dx}$, then $\frac{1}{r_1m_1}\l$ is the $2m$th Neumann eigenvalue, that is, for all $m\in \N$,
	\[ r_1m_1\, \l_{N,2m} = \l_{N,m}.\]
\end{theo}

This Theorem has been proved with the method of Prüfer angles by Volkmer \cite{volkmer:eigenvalues} for the case $r_1=r_2=\frac{1}{3}$, $m_1=m_2=\frac{1}{2}$ and by Freiberg \cite{freiberg:pruefer} in a more general setting. It delivers the foundation for the statements in this section. An analogous property for Dirichlet eigenvalues does not seem to hold. However, in the symmetric case there is a similar relation between Dirichlet eigenvalues and eigenvalues of the problems (DN) or (ND) posed in Section \ref{NDDNproblem}. Remember, (DN) has boundary conditions $f(0)=f'(1)=0$ and (ND) has $f'(0)=f(1)=0$.

\begin{pro}
Let $\m$ be symmetric, that is $r:=r_1=r_2$ and $m_1=m_2=\frac{1}{2}$ and let $\l$ be an eigenvalue of (DN) or (ND). Then $\frac{2}{r}\l$ is a Dirichlet eigenvalue and if $f$ is a $\frac{2}{r}\l$-Dirichlet eigenfunction, then $f\circ S_1$ is a $\l$-(DN) eigenfunction, and $f\circ S_2$ is a $\l$-(ND) eigenfunction.
\end{pro}
\begin{proof}
 In Corollary \ref{cor:symm}  we showed that since $\m$ is symmetric, we have $\cos_{\m}^N=\cos_{\m}^D$. Then we can factorize \eqref{th12} and get
\[ \sin_{\m}^D(\sqrt{\tfrac{2}{r}}z\bigr)= \cos_{\m}^N(z) \cdot \Bigl[2\sqrt{2r} \sin_{\m}^D(z) + (1-2r)\,z \cos_{\m}^N(z) \Bigr]. \]
Since $\l$ is an eigenvalue of the (DN) and the (ND) problem, $\cos_{\m}^N(\sqrt{\l})=0$. Then, $\sin_{\m}^D(\sqrt{\tfrac{2}{r}\l}\bigr)=0$ and thus, $\tfrac{2}{r}\l$ is a Dirichlet eigenvalue. From Propositions \ref{cpzsqzS1} and \ref{cpzsqzS2} we get for $x\in [0,1]$
\[ \s_{\l,\m}\Bigl(\sqrt{\tfrac{2}{r}\l}, S_1(x)\Bigr) = \sqrt{2r} \s_{\l,\m}(\sqrt{\l},x)\]
and
\[ \s_{\l,\m}\Bigl(\sqrt{\tfrac{2}{r}\l}, S_2(x)\Bigr) = \sqrt{2r} \sin_{\m}^D\bigl(\sqrt{\l}\bigr) \c_{\l,\m}\bigl( \sqrt{\l}, x \bigr),\]
which proves the proposition.
\end{proof}

In the following we treat only the Neumann eigenvalue problem for a (not necessarily symmetric) measure $\mu$ using Theorem \ref{th:ssnev}. With the formula
\begin{equation}\label{cpqspq1}
\cos_{\m}^D(z)\,\cos_{\m}^N(z)+ \sin_{\m}^D(z)\,\sin_{\m}^N(z) = 1,
\end{equation}
which follows from Theorem \ref{pythallg} by setting $x=1$, we rearrange the functional equations from Theorem \ref{th1}. With the
abbreviation
\begin{equation}\label{h1}
h(z):=r_1 \cos_{\m}^N(z)+ r_2\cos_{\m}^D(z)-\bigl[1-(r_1+r_2)\bigr] z \sin_{\m}^N(z)
\end{equation}
we can write
\begin{align}
\label{sinp2} \sin_{\m}^N(z)	&= \frac{\sqrt{r_1m_1}}{r_1r_2}\,\sin_{\m}^N\bigl(\sqrt{r_1m_1}z\bigr) \,h\bigl(\sqrt{r_1m_1}z\bigr),\\
\label{cosp2} \cos_{\m}^N(z)	&=	-\frac{r_2}{r_1}+\frac{1}{r_1}\,\cos_{\m}^N\bigl(\sqrt{r_1m_1}z\bigr) \,h\bigl(\sqrt{r_1m_1}z\bigr),
\intertext{and}
\label{sinq2}\sin_{\m}^D(z)	&= \bigl[1-(r_1+r_2)\bigr]z + \frac{1}{\sqrt{r_1m_1}}\, \sin_{\m}^D\bigl(\sqrt{r_1m_1}z\bigr)\,h\bigl(\sqrt{r_1m_1}z\bigr),\\
\cos_{\m}^D(z)	&= -\frac{r_1}{r_2} + \frac{1}{r_2} \cos_{\m}^D\bigl(\sqrt{r_1m_1}z\bigr) \,h\bigl(\sqrt{r_1m_1}z\bigr).
\end{align}


Employing the above formulas we can calculate the values of $\cos_{\m}^N$, $\cos_{\m}^D$ and $\sin_{\m}^D$ at the zero points of $\sin_{\m}^N$.

\begin{lem}
Let $m\in\N$ and let $v(m)$ be the multiplicity of the prime factor $2$ in $m$. Let $z_m:=\sqrt{\l_{N,m}}$ be the square root of the $m$th Neumann eigenvalue, that is, the $m$th zero point of $\sin_{\m}^N$. Then
\begin{align}
\label{cospvm}\cos_{\m}^N(z_m)	&	= \Bigl(-\frac{r_2}{r_1}\Bigr)^{\ds 2^{v(m)}}\\
\label{cosqvm}\cos_{\m}^D(z_m)	&	= \Bigl(-\frac{r_1}{r_2}\Bigr)^{\ds 2^{v(m)}}\\
\label{sinqvm}\sin_{\m}^D(z_m)	& = a_{v{\scriptscriptstyle (}m{\scriptscriptstyle)}} \cdot z_m 
\end{align}
where $(a_k)_k$ is determined by
\[a_0= 1-(r_1+r_2)\]
and, for $k\in\N$,
\[a_k= 1-(r_1+r_2) + a_{k-1} \biggl(r_1 \Bigl(-\frac{r_2}{r_1}\Bigr)^{2^{k-1}} + r_2 \Bigl(-\frac{r_1}{r_2}\Bigr)^{2^{k-1}}\biggr).\]
\end{lem}
\begin{proof}
Suppose $m$ is odd. Then $\sin_{\m}^N(z_m)=0$ and $\sin_{\m}^N\bigl(\sqrt{r_1m_1}z_m\bigr) \neq 0$. To see this, suppose $\sin_{\m}^N\bigl(\sqrt{r_1m_1}z_m\bigr) = 0$. Then
$r_1m_1 z_m^2$ would be a Neumann eigenvalue, say $r_1m_1 z_m^2=\l_{N,l}$ for some $l\in\N$, and because of Theorem \ref{th:ssnev}, $z_m^2$ would be the eigenvalue $\l_{N,2l}$. Thus, $m=2l$, which is a contradiction.

Hence, it follows by  \eqref{sinp2} that $h\bigl(\sqrt{r_1m_1}z_m\bigr)=0$. Then, by \eqref{cosp2}, $\cos_{\m}^N(z_m)=-\dfrac{r_2}{r_1}$.

By \eqref{th13} follows that, for all $z\in\R$, if
\[ \sin_{\m}^N\bigl(\sqrt{r_1m_1}z\bigr)=0, \]
then
\[ \cos_{\m}^N(z)=\cos_{\m}^N\bigl(\sqrt{r_1m_1}z\bigr)^2.\]
Thus, if $m= 2 l$ for some odd $l$, then $\sqrt{r_1m_1}z_m = z_l$ and hence, 
\[\cos_{\m}^N(z_m)=\cos_{\m}^N (z_l)^2 = \Bigl(-\frac{r_2}{r_1}\Bigr)^2.\]
 Iteratively, we get that, if $m=2^kl$ for some odd $l$,
\[\cos_{\m}^N(z_m) = \Bigl(-\frac{r_2}{r_1}\Bigr)^{2^k},\]
which proves \eqref{cospvm}.

Since $\sin_{\m}^N(z_m)=0$ for all $m\in\N$ we get by \eqref{cpqspq1} that
\[\cos_{\m}^D(z_m) = \frac{1}{\cos_{\m}^N(z_m)} = \Bigl(-\frac{r_1}{r_2}\Bigr)^{2^{v(m)}},\]
which is \eqref{cosqvm}.

Now we show \eqref{sinqvm}. At first, suppose $v(m)=0$, that is, $m$ is odd. Then, as above, $h\bigl(\sqrt{r_1m_1}z_m\bigr)=0$ and thus, by \eqref{sinq2},
\[\sin_{\m}^D(z_m)=\bigl[1-(r_1+r_2)\bigr] z_m.\]
Observe that we have for all $m$
\begin{equation}\label{h1vm}
h(z_m) = r_1 \Bigl(-\frac{r_2}{r_1}\Bigr)^{2^{v(m)}} + r_2 \Bigl(-\frac{r_1}{r_2}\Bigr)^{2^{v(m)}}.
\end{equation}
Suppose $v(m)\geq 1$. Then $\sqrt{r_1m_1}z_m=z_{\scriptscriptstyle\frac{m}{2}}$ and thus,
\begin{align*}
\frac{\sin_{\m}^D(z_m)}{z_m}	&	= 1-(r_1+r_2) + \frac{\sin_{\m}^D\bigl(\sqrt{r_1m_1}z_m\bigr)}{\sqrt{r_1m_1}z_m} h\bigl(\sqrt{r_1m_1}z_m\bigr)\\
												&	= 1-(r_1+r_2) + \frac{\sin_{\m}^D\bigl(z_{\scriptscriptstyle\frac{m}{2}}\bigr)}{z_{\scriptscriptstyle\frac{m}{2}}}
																																										h\bigl(z_{\scriptscriptstyle\frac{m}{2}}\bigr)\\
												& = 1-(r_1+r_2) + \frac{\sin_{\m}^D\bigl(z_{\scriptscriptstyle\frac{m}{2}}\bigr)}{z_{\scriptscriptstyle\frac{m}{2}}}
												\biggl( r_1 \Bigl(-\frac{r_2}{r_1}\Bigr)^{2^{v(m)-1}} + r_2 \Bigl(-\frac{r_1}{r_2}\Bigr)^{2^{v(m)-1}}\biggr).
\end{align*}
Hence, $\dfrac{\sin_{\m}^D(z_m)}{z_m}$ depends only on $v(m)$ and so, with $a_{v(m)} = \dfrac{\sin_{\m}^D(z_m)}{z_m}$, we get
\[a_{v(m)} = 1-(r_1+r_2) + a_{v(m)-1}	\biggl( r_1 \Bigl(-\frac{r_2}{r_1}\Bigr)^{2^{v(m)-1}} + r_2 \Bigl(-\frac{r_1}{r_2}\Bigr)^{2^{v(m)-1}}\biggr),\]
which proves the assertion.
\end{proof}

We use the above computed values of $\cos_{\m}^N(z_m)$ and Propositions \ref{cpzsqzS1} and \ref{cpzsqzS2} to get a relation between the $m$th and the $2m$th  Neumann eigenfunction.

\begin{pro}\label{proreleigfunct2mm}
	Let $m\in\N$ and $v(m)$ be the $2$-multiplicity of $m$. We denote the $m$th Neumann eigenfunction by $f_m:=\c_{\l,\m}(z_m, \cdot )$. Then, for all $x\in [0,1]$,
	\begin{equation}\label{f2mS1fm}
	f_{2m}\bigl(S_1(x)\bigr) = f_m (x)
	\end{equation}
	and
	\begin{equation}\label{f2mS2fm}
	f_{2m}\bigl(S_2(x)\bigr) = \Bigl(-\frac{m_1}{m_2}\Bigr)^{2^{v(m)}} \, f_m(x).
	\end{equation}
\end{pro}
\begin{proof}
Because of Theorem \ref{th:ssnev} we have $\l_m = r_1m_1\l_{2m}$ and thus, 
\[\sin_{\m}^N\bigl(\sqrt{r_1m_1} z_{2m}\bigr) = 0.\]
Since	$f_m=\c_{\l,\m}(z_m, \cdot )$, Propositions \ref{cpzsqzS1} and \ref{cpzsqzS2} give for $x\in[0,1]$
\[f_{2m}\bigl(S_1(x)\bigr) = f_m (x)\]
and
\[f_{2m}\bigl(S_2(x)\bigr) = \cos_{\m}^N (z_m) \, f_m(x).\]
Noting that $\dfrac{r_2}{r_1}=\dfrac{m_1}{m_2}$, we get with \eqref{cospvm} that
\[f_{2m}\bigl(S_2(x)\bigr) = \Bigl(-\frac{m_1}{m_2}\Bigr)^{2^{v(m)}} \, f_m(x).\]
\end{proof}

The above proposition can be employed to work out the relationship between the suprema and the $L_2(\m)$ norms of $f_{m}$ and $f_{2m}$.
\begin{pro}\label{pronormrel}
Let $m\in\N$ and $v(m)$ the $2$-multiplicity of $m$. Then
\begin{align}
\label{l2normrel}\norm{f_{2m}}{L_2(\m)}^2	& = \biggl(m_1+ m_2 \Bigl(\frac{m_1}{m_2}\Bigr)^{2^{v(m)+1}}\biggr) \norm{f_m}{L_2(\m)}^2\\
\intertext{and}
\label{supnorm}\norm{f_{2m}}{\infty}			& = \max\Bigl\{1,\Bigl(\frac{m_1}{m_2}\Bigr)^{2^{v(m)}}\Bigr\} \norm{f_m}{\infty}.
\end{align}
\end{pro}
\begin{proof}
	At first we prove \eqref{l2normrel}. For $m\in\N$ we have
	\begin{align*}
	\norm{f_{2m}}{L_2(\m)}^2	& = \int_{S_1(0)}^{S_1(1)} f_{2m}(t)^2\,d\m(t) + \int_{S_2(0)}^{S_2(1)}  f_{2m}(t)^2\,d\m(t)\\
														& = m_1\int_{S_1(0)}^{S_1(1)} f_{2m}(t)^2\,d(S_1\m)(t) + m_2 \int_{S_2(0)}^{S_2(1)}  f_{2m}(t)^2\,d(S_2\m)(t)\\
														& = m_1\int_0^1 f_{2m}\bigl(S_1(t)\bigr)^2\,d\m(t) + m_2 \int_0^1  f_{2m}\bigl(S_2(t)\bigr)^2\,d\m(t).
	\end{align*}
	By \eqref{f2mS1fm} and \eqref{f2mS2fm} we get
	\begin{align*}
	\norm{f_{2m}}{L_2(\m)}^2	& = m_1\int_0^1 f_{m}(t)^2\,d\m(t) + m_2 \Bigl(-\frac{m_1}{m_2}\Bigr)^{2^{v(m)+1}} \int_0^1  f_m(t)^2\,d\m(t)\\
														& = \Bigl[m_1+ m_2 \Bigl(\frac{m_1}{m_2}\Bigr)^{2^{v(m)+1}}\Bigr] \norm{f_m}{L_2(\m)}^2.
	\end{align*}
	Now we show \eqref{supnorm}. With \eqref{f2mS1fm} and \eqref{f2mS2fm} we have
	\[ \sup_{x\in [S_1(0),S_1(1)]} \abs{f_{2m}(x)} = \sup_{x\in [0,1]} \bigabs{f_{2m}\bigl(S_1(x)\bigr)} = \sup_{x\in [0,1]} \abs{f_{m}(x)} = \norm{f_m}{\infty}\]
	and
	\begin{align*}
	\sup_{x\in [S_2(0),S_2(1)]} \abs{f_{2m}(x)} &= \sup_{x\in [0,1]} \bigabs{f_{2m}\bigl(S_2(x)\bigr)} = \Bigl(\frac{m_1}{m_2}\Bigr)^{2^{v(m)}}\norm{f_m}{\infty}.
	\end{align*}
	Therefore, since $f_{2m}$ is linear on $[S_1(1), S_2(0)]$ and continuous,
	\[\sup_{x\in[0,1]} \abs{f_{2m}(x)} = \max\Bigl\{1,\Bigl(\frac{m_1}{m_2}\Bigr)^{2^{v(m)}}\Bigr\} \norm{f_m}{\infty}.\]
\end{proof}

Now we consider the normalized Neumann eigenfunctions. For $m\in\N_0$ we set 
\[\tilde{f}_m := \norm{f_m}{L_2(\m)}^{-1} f_m.\]
We are interested in the asymptotic behaviour of the sequence $(\norm{\tilde{f}_m}{\infty})_m$. With Proposition \ref{pronormrel} we get some information about certain subsequences stated in the following theorem.

\pagebreak
\begin{theo}
Let $\m$ be a self-similar measure with $r_1m_1=r_2m_2$. Then, for all $m\in\N_0$,
\begin{equation}\label{normsupnormneumann}
 \norm{\tilde{f}_{2m}}{\infty} = \frac{\max\Bigl\{1,\bigl(\frac{m_1}{m_2}\bigr)^{2^{v(m)}}\Bigr\}}{\sqrt{m_1+ m_2 \bigl(\frac{m_1}{m_2}\bigr)^{2^{v(m)+1}}}} \norm{\tilde{f}_m}{\infty}.
\end{equation}
Suppose $m_1 \leq m_2$ and let $l$ be an odd number. Then, for all $k\in\N$,
\begin{equation}\label{supnormproductformula}
\norm{\tilde{f}_{2^k l}}{\infty} = m_1^{-\frac{k}{2}} \prod_{j=1}^k \Bigl( 1+ \Bigl(\frac{m_1}{m_2}\Bigr)^{2^j-1}\Bigr)^{-\frac{1}{2}} \,\norm{\tilde{f}_l}{\infty}.
\end{equation}
\end{theo}
\begin{proof}
\eqref{normsupnormneumann} follows directly from \eqref{l2normrel} and \eqref{supnorm}. Suppose $m_1\leq m_2$ and $l\in \N$ is odd. Then iterative application of \eqref{normsupnormneumann} gives
\[ \norm{\tilde{f}_{2l}}{\infty} = \frac{1}{\sqrt{m_1}}\frac{1}{\sqrt{1+\frac{m_1}{m_2}}} \norm{\tilde{f}_l}{\infty},\]
\[ \norm{\tilde{f}_{2^2 l}}{\infty} = \frac{1}{\sqrt{m_1}}\frac{1}{\sqrt{1+\bigl(\frac{m_1}{m_2}\bigr)^3}}\frac{1}{\sqrt{m_1}}\frac{1}{\sqrt{1+\frac{m_1}{m_2}}} \norm{\tilde{f}_l}{\infty},\]
and so on, and therefore \eqref{supnormproductformula} holds.
\end{proof}

\begin{cor}\label{cor:upperandlowerest}
Let $l\in\N$ be odd. Then the following statements hold.
\begin{enumerate}
	\item\label{item1}
		If $m_1=m_2$, then for all $k\in\N$,
		\[\norm{\tilde{f}_{2^k l}}{\infty}=\norm{\tilde{f}_{l}}{\infty}.\]
	\item\label{item2}
		If $m_1 < m_2$, then $C:= \dfrac{1}{\sqrt{m_1 \bigl(1+\frac{m_1}{m_2} \bigr)}}>1$, and we have for all $k\in\N$,
		\[\norm{\tilde{f}_{2^k l}}{\infty} \geq C^k \norm{\tilde{f}_{l}}{\infty}.\]
		Additionally, for all $k\in \N$,
		\[\norm{\tilde{f}_{2^k l}}{\infty} \leq m_1^{-\frac{k}{2}} \biggl(\frac{m_2}{m_1}\biggr)^{\frac{k}{2}(2^k-1)}\norm{\tilde{f}_{l}}{\infty}.\]
\end{enumerate}
\end{cor}
\begin{proof}
\ref{item1} follows directly from \eqref{supnormproductformula} by putting $m_1=m_2=\frac{1}{2}$.

If $m_1<m_2$, then, for all $j\in\N$,
\[  1+ \Bigl(\frac{m_1}{m_2}\Bigr)^{2^j-1} \leq 1+\frac{m_1}{m_2}.\]
Then,
\[ \norm{\tilde{f}_{2^k l}}{\infty} \geq  m_1^{-\frac{k}{2}} \Bigl(1+\frac{m_1}{m_2}\Bigr)^{-\frac{k}{2}} \,\norm{\tilde{f}_l}{\infty},\]
and since $m_1<m_2$ implies $m_1 < \frac{1}{2}$, we have $m_1 \bigl(1+\frac{m_1}{m_2} \bigr)<1$.

For the upper estimate, we write
\[\prod_{j=1}^k \biggl( 1+ \Bigl(\frac{m_1}{m_2}\Bigr)^{2^j-1}\biggr) \geq \biggl( 1+ \Bigl(\frac{m_1}{m_2}\Bigr)^{2^k-1}\biggr)^k
 \geq \Bigl(\frac{m_1}{m_2}\Bigr)^{k(2^k-1)},\]
which proves \ref{item2}.
\end{proof}

\section{Self-similar measures with $r_1m_1=r_2m_2$ and $r_1+r_2=1$}
\label{r1r2g1}

As in the previous section we have the condition $r_1m_1=r_2m_2$. 
We treat the special case where $r_1+r_2=1$ from which follows that $r_1=m_2$ and $r_2=m_1$. Such measures have been investigated e.g. by Sabot \cite{sabot:density} and \cite{sabot:higher}.

\begin{theo}\label{th:DirandNeusame}
Let $\m$ be a self-similar measure where $r_1=m_2$ and $r_2=m_1$ (and therefore $r_1+r_2=1$). Then the positive eigenvalues of $-\dfrac{d}{d\m}\dfrac{d}{x}$ with Neumann boundary conditions coincide with those with Dirichlet boundary conditions.
\end{theo}
\begin{proof}
Since the eigenvalues are the squares of the zeros of $\sin_{\m}^N$ and $\sin_{\m}^D$, respectively, it is sufficient to show that $\sin_{\m}^N = \sin_{\m}^D$. To do that we show that for all $n\in\N_0$ 
\[ p_{2n+1}=q_{2n+1}.\]
We do this by complete induction using the recursion formulas from Corollary \ref{rec}. By Definition \ref{defpq} we have
\[ p_1 = \int_0^1 \, d\m = 1\]
and
\[ q_1 = \int_0^1 \, dt = 1.\]
Now, let $n\in\N$ and suppose that for $i=0,\dotsc,n-1$ holds $p_{2i+1}=q_{2i+1}$. By \eqref{rec1} and rearrangement of the order of the terms in the sums we get
\begin{align*}
p_{2n+1}		& = \frac{1}{1-m_2^n r_2^{n+1} - m_1^n r_1^{n+1}}\bigg( \sum_{i=0}^{n-1}m_2^i r_2^{i+1}(r_1 m_1)^{n-i} p_{2i+1}\,q_{2n-2i} \\
						& \hspace{2cm} + \sum_{i=1}^n (r_2m_2)^i m_1^{n-i} r_1^{n-i+1} p_{2i}\, p_{2n-2i+1}\bigg)	\\
						& = \frac{1}{1-m_2^n r_2^{n+1} - m_1^n r_1^{n+1}}\bigg( \sum_{i=1}^n m_2^{n-i} r_2^{n+1-i}(r_1 m_1)^{i} p_{2n+1-2i}\,q_{2i} \\
						& \hspace{2cm} + \sum_{i=0}^{n-1} (r_2m_2)^{n-i} m_1^{i} r_1^{i+1} p_{2n-2i}\, p_{2i+1}\bigg).
\end{align*}
Then, by the induction hypothesis and \eqref{rec3},
\begin{align*}
p_{2n+1}		& = \frac{1}{1-m_2^n r_2^{n+1} - m_1^n r_1^{n+1}}\bigg( \sum_{i=1}^n m_2^{n-i} r_2^{n+1-i}(r_1 m_1)^{i} q_{2n+1-2i}\,q_{2i} \\
						& \hspace{2cm} + \sum_{i=0}^{n-1} (r_2m_2)^{n-i} m_1^{i} r_1^{i+1} p_{2n-2i}\, q_{2i+1}\bigg)\\
						& = q_{2n+1}.
\end{align*}
\end{proof}

With the above theorem we can reformulate Theorem \ref{pythallg} to get a property of the Wronskian of $f_{N,m}$ and $f_{D,m}$.

\begin{cor}
	Let $\m$ be as above, let $\l_m$ be the $m$th eigenvalue, let $f_{N,m}= \c_{\l,\m}(\sqrt{\l_m}, \cdot )$ and $f_{D,m} = \s_{\l,\m}(\sqrt{\l_m},\cdot)$ be the corresponding Neumann and Dirichlet eigenfunctions constructed in Section \ref{NandDProb}. Then, for all $x\in [0,1]$,
	\[f_{N,m}(x)\,f_{D,m}'(x) - f_{D,m}(x)\,f_{N,m}'(x) = \sqrt{\l_m}.\]
\end{cor}
\begin{proof}
	We put  $z=\sqrt{\l_m}$ in Theorem \ref{pythallg} and observe that
	\[ f_{N,m}' = \c'_{\l,\m}(\sqrt{\l_m}, \cdot ) = - \sqrt{\l_m} \s_{\m,\l}(\sqrt{\l_m}, \cdot)\]
	and
	\[ f_{D,m}' = \s_{\l,\m}'(\sqrt{\l_m},\cdot)= \sqrt{\l_m} \c_{\m,\l}(\sqrt{\l_m}, \cdot).\]
\end{proof}

Since eigenfunctions can be multiplied with any non-zero number, the above equation states basically that the Wronskian is constant. A similar property of a different Wronskian has been established in Freiberg \cite{freiberg:analytical}*{p. 41}.

\section{Figures and numbers}
\label{numerical}

In this section we give some explicit results and figures calculated by using formulas we developed in the preceding sections for several examples of self-similar measures. For the calculations we used Sagemath cloud \cite{sage}. The program code that we used can be found in the appendix.

\begin{exa}\label{exa1}
Tables \ref{tab:1312p2n+1}, \ref{tab:1312q2n+1} and \ref{tab:1312p2nq2n} collect the first few values of the sequences $(p_n)_n$ and $(q_n)_n$ for the classical Cantor set with evenly distributed measure, that is, for $r_1=r_2=\frac{1}{3}$ and $m_1=m_2=\frac{1}{2}$. We computed these values with the recursion formulas in Corollary \ref{rec} that we implemented for that purpose in Sagemath.

Figures \ref{fig:1312sinpsinqbis49} and \ref{fig:1312sinpsinqbis120} show plots of the functions $\sin_{\m}^N$ and $\sin_{\m}^D$ for $x\in(0,50)$ and for $x\in(0,120)$, respectively, where the first $100$ terms of the series are taken into account.  In the figures, $\sin_{\m}^D$ is drawn in a dash-dot line and $\sin_{\m}^N$ in a solid line. The zero points of these functions squared give the Dirichlet and Neumann eigenvalues, respectively. Observe that the pictures suggest that the eigenvalues are in the order
\[ \l_{N,0} < \l_{N,1} < \l_{D,1} < \l_{D,2} < \l_{N,2} < \l_{N,3} < \l_{D,3} < \l_{D,4} < \dots.\]

Table \ref{tab:nev1312} contains the first $32$ positive Neumann eigenvalues correct to $15$ decimal places (rounded down). These values have been calculated as zero points of the polynomial 
\[  \sum_{n=0}^a (-1)^n \, p_{2n+1}\,z^n \quad \biggl(\approx \frac{\sin_{\m}^N(\sqrt{z})}{\sqrt{z}} \biggr).\]
For that we used the command {\ttfamily findroot} from the {\ttfamily mpmath} library in Sagemath cloud (\cite{sage}) with a starting value that we took from a plot in each case. This way, we computed each zero point of the above polynomial with an accuracy of $100$ digits, where we chose $a$ in each case such that the first $15$ decimals of the zero point remain fixed against any further increase of $a$. Note that by Lemma \ref{convabsch} we have
\[ p_{2n+1} \leq \frac{1}{n!}q_2(1)^n = \frac{1}{n!\cdot 2^n},\] 
from which a more detailed error estimate can be obtained.

Observe that, as stated in Theorem \ref{th:ssnev}, we have that $\l_{N,2m} = 6 \cdot \l_{N,m}$ for all $m$. The distances between eigenvalues differ very much, there are several groups that lie very close together while there are big gaps as well.

In Table \ref{tab:nef1312} we give approximate values of the $L_2(\m)$ norms and the $\sup$ norms of the eigenfunctions $f_{N,m}=\c_{\l,\m}(\sqrt{\l_m}, \cdot )$.

The $L_2$ norms have been calculated with the formula in Corollary \ref{cor:l2norm} where we put in the values for $\l$ from Table \ref{tab:nev1312}. The number of summands had to be chosen higher with bigger eigenvalues, so that the limit value could be approximated with sufficient accuracy.

For the supremum norms we calculated $f_{N,m}\bigl(S_w(0)\bigr)$ and $f_{N,m}\bigl(S_w(1)\bigr)$ for all words $w\in \{1,2\}^n$ for a certain iteration level $n$ and
determined the biggest of these values. We varied $n$ between $5$ and $8$ to get the values. These calculations were made with the formulas in 
Proposition \ref{cpzsqzS2}. For that, the eigenvalue $\l_m$ and values of the functions $\sin_{\m}^N$, $\sin_{\m}^D$ and $\cos_{\m}^N$ were needed. Note that the $\sup$ norm values are rather rough approximations.

Then we determined the $\sup$ norm of the normalized eigenfunctions 
\[\norm{\tilde{f}_{N,m}}{\infty} = \frac{\norm{f_{N,m}}{\infty}}{\norm{f_{N,m}}{L_2(\m)}}.\]
Observe that, as stated in Equation \eqref{normsupnormneumann}, the values for even $m$ are the same as for $\frac{m}{2}$, respectively. 

Figure \ref{fig:eigenfunctions} shows plots of $f_{1,N}$ to $f_{4,N}$ and $f_{1,D}$ to $f_{4,D}$. These were done by iterative use of the formulas in Propositions
\ref{cpzsqzS1} and \ref{cpzsqzS2} as for the calculation of the sup-norms.

In Table \ref{tab:dev1312} we state the first $32$ eigenvalues with Dirichlet boundary conditions exact to 15 decimals. The procedure for the calculations is the same as with the Neumann eigenvalues explained above. Note that two values at a time lie close together, namely  $\l_{D, 2m-1}$ and $\l_{D, 2m}$. Especially close together are pairs of the form $\l_{D, 2^n-1}$ and $\l_{D,2^n}$. Therefore we had to increase the accuracy of $\l_{D,31}$ and $\l_{D,32}$ to $25$ digits to make the difference visible.

Estimates of the Dirichlet eigenvalues have also been obtained by Vladimirov and Sheipak in \cite{vladimirov:singular} and by Etienne \cite{etienne:chains} with completely different methods. 

As in the Neumann case, we calculated norms of Dirichlet eigenfunctions, see Table \ref{tab:def1312}.

\begin{table}[hb]
\renewcommand{\arraystretch}{1.3}
\centering
\begin{tabular}{>{$}l<{$} | >{$}l<{$}}
n		& p_{2n+1}\\
\hline
0		& 1 \\
1		& \n{1}{5}\\
2		& \n{27}{2800}\\
3		& \n{6383}{31906000}\\
4		&	\frac{928\,046\,087}{427\,065\,638\,720\,000}\\
5		&	\frac{18\,312\,146\,532\,699}{1\,290\,321\,173\,531\,252\,800\,000}\\
6		&	\frac{36\,205\,626\,974\,761\,334\,065\,053}{595\,390\,835\,517\,679\,574\,442\,022\,016\,000\,000}\\
7		&	\frac{4\,976\,934\,962\,986\,304\,441\,117\,658\,183}{27\,444\,983\,400\,881\,701\,904\,144\,720\,110\,742\,041\,600\,000}\\
8		& \frac{9\,554\,109\,968\,352\,546\,557\,662\,907\,330\,504\,773\,561\,465\,623}{24\,293\,779\,244\,421\,488\,801\,231\,482\,393\,897\,413\,175\,652\,507\,508\,121\,600\,000\,000}\\
9		& \frac{146\,991\,787\,616\,583\,137\,720\,984\,325\,054\,111\,289\,057\,094\,244\,281\,881\,523\,497}{228\,839\,658\,236\,344\,563\,453\,452\,927\,437\,095\,017\,291\,959\,177\,590\,164\,358\,527\,465\,655\,296\,000\,000\,000}\\
\end{tabular}
\caption{The first ten odd members of $(p_n)$ for $r_1=r_2=\frac{1}{3}$ and $m_1=m_2=\frac{1}{2}$.}
\label{tab:1312p2n+1}
\end{table}

\begin{table}
\renewcommand{\arraystretch}{1.3}
	\centering
	\begin{tabular}{>{$}l<{$} | >{$}l<{$}}
	n		& q_{2n+1}\\
	\hline
	0		&	1\\
	1		&	\n{1}{8}\\
	2		& \n{21}{4240}\\
	3		&	\n{33253}{383465600}\\
	4		&	\n{76118969}{91537621184000}\\
	5		&	\n{20165083798890939}{4103397246999022891520000}\\
	6		&	\n{129726498389261896497}{6714982210971717632658867200000}\\
	7		& \n{2413673468793966201825434809368471}{45210174990342427454327995801851920608256000000}\\
	8		&	\n{1194381655935980000421990244022269580561517}{11036319046998816108771342849627021590229476137440051200000}\\
	9		&	\n{126866175828333349955887526100988154691317901447037378112773}{762232235417372510271600164875680211782266161937386279477493896522956800000000}
	\end{tabular}
	\caption{The first ten odd members of $(q_n)$ for $r_1=r_2=\frac{1}{3}$ and $m_1=m_2=\frac{1}{2}$.}
	\label{tab:1312q2n+1}
\end{table}

\begin{table}
\renewcommand{\arraystretch}{1.3}
	\centering
		\begin{tabular}{>{$}l<{$} | >{$}l<{$}}
n		& p_{2n},q_{2n}\\
\hline
0		&	1\\
1		&	\n{1}{2}\\
2		& \n{3}{80}\\
3		&	\n{311}{296800}\\
4		&	\n{4716349}{329780416000}\\
5		&	\n{186511983201}{1659577072065920000}\\
6		& \n{7179455540679158013}{12761565438166961192627200000}\\
7		&	\n{159906376968352543502900259}{83334473684067539316352053491456000000}\\
8		&	\n{60996703846644308894938372985688873}{13022158544999621792336779426151940728460083200000}\\
9		& \n{55173436475334110717731416972957128310218371151677}{6487868455643720781486892657131701453895648546905230058700800000000}
\end{tabular}
	\caption{The first ten even members of $(p_n)$ and $(q_n)$ for $r_1=r_2=\frac{1}{3}$ and $m_1=m_2=\frac{1}{2}$.}
	\label{tab:1312p2nq2n}
\end{table}

\begin{figure}
	\centering
		\includegraphics[width=\textwidth]{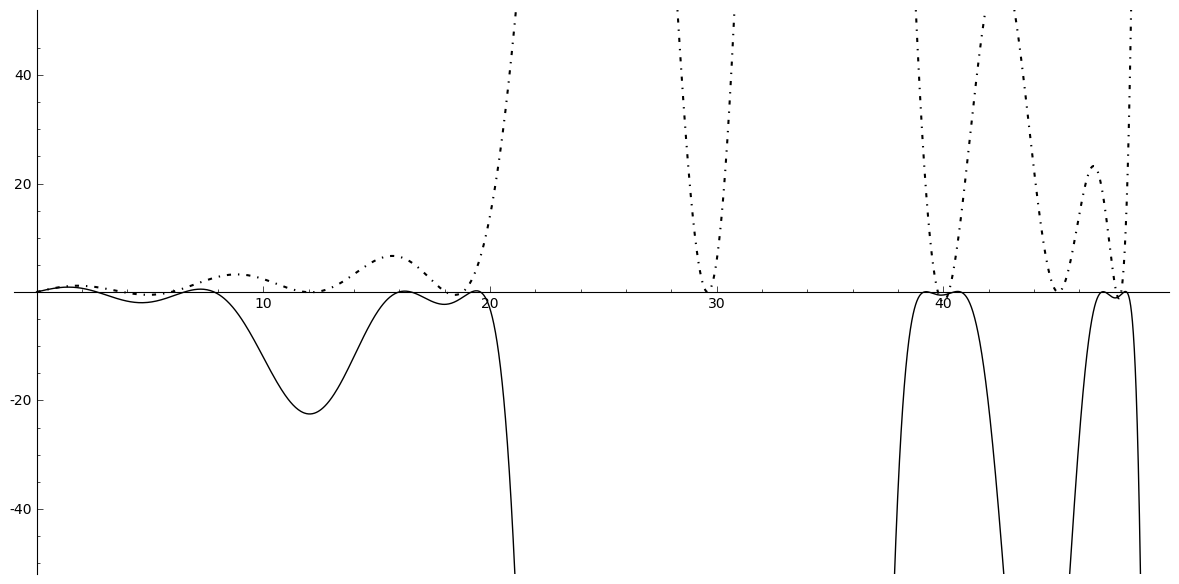}
	\caption{$\sin_{\m}^N$ (solid) and $\sin_{\m}^D$ (dash-dot) for $r_1=r_2=\frac{1}{3}$ and $m_1=m_2=\frac{1}{2}$.}
	\label{fig:1312sinpsinqbis49}
\end{figure}

\begin{figure}
	\centering
		\includegraphics[width=\textwidth]{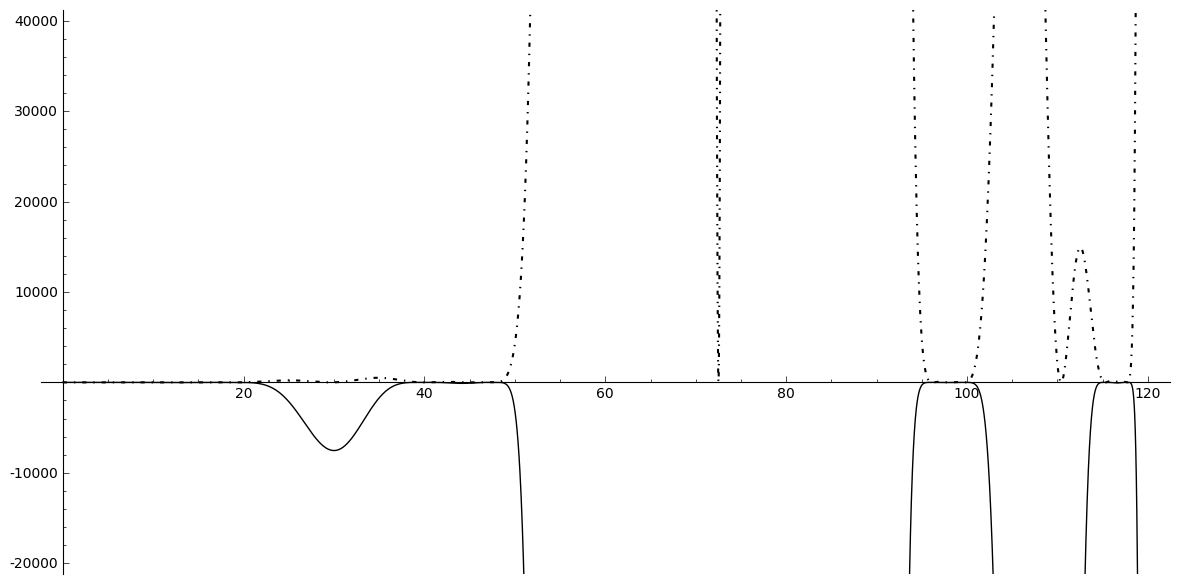}
	\caption{$\sin_{\m}^N$ (solid) and $\sin_{\m}^D$ (dash-dot) for $r_1=r_2=\frac{1}{3}$ and $m_1=m_2=\frac{1}{2}$.}
	\label{fig:1312sinpsinqbis120}
\end{figure}

\noindent
\begin{minipage}{.5\textwidth}
\centering
\includegraphics[width=\textwidth]{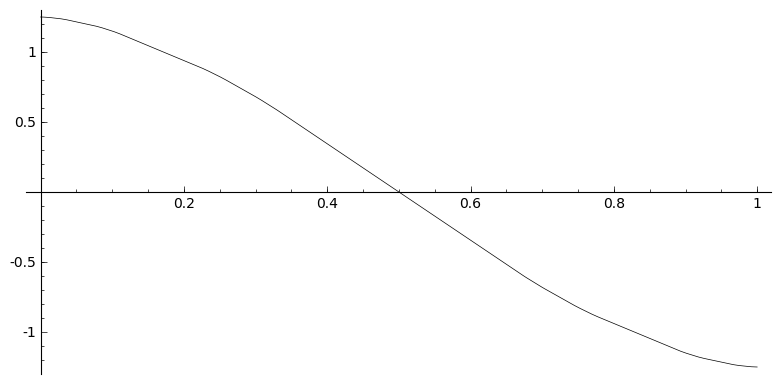}
\end{minipage}
\begin{minipage}{.5\textwidth}
\centering
\includegraphics[width=\textwidth]{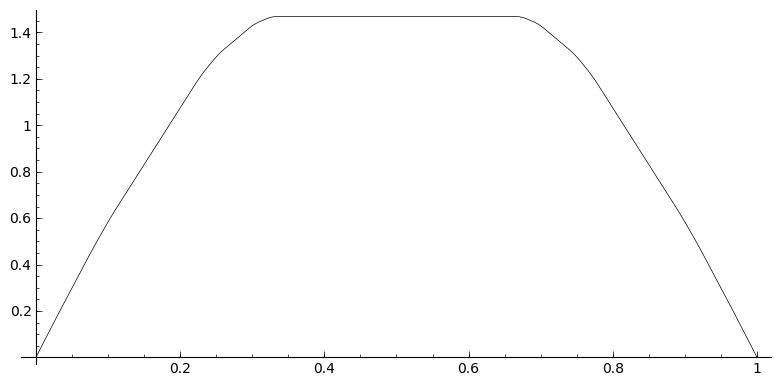}
\end{minipage}
\begin{minipage}{.5\textwidth}
\centering
\includegraphics[width=\textwidth]{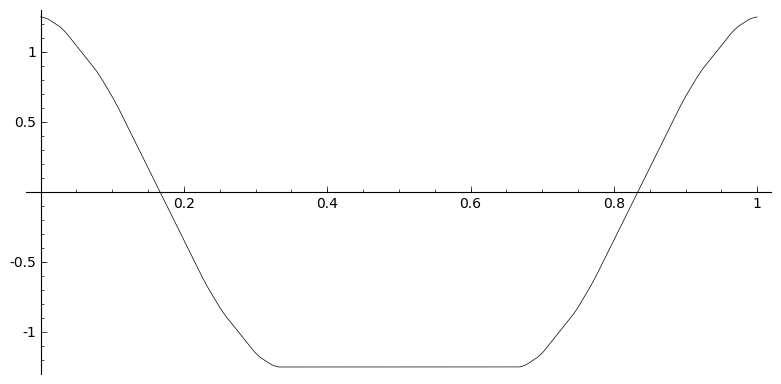}
\end{minipage}
\begin{minipage}{.5\textwidth}
\centering
\includegraphics[width=\textwidth]{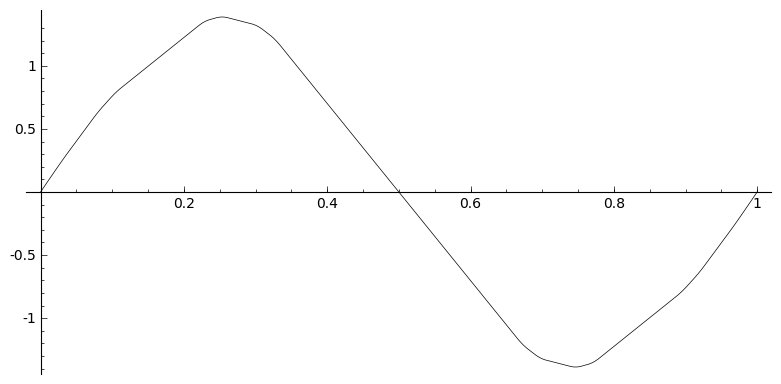}
\end{minipage}
\begin{minipage}{.5\textwidth}
\centering
\includegraphics[width=\textwidth]{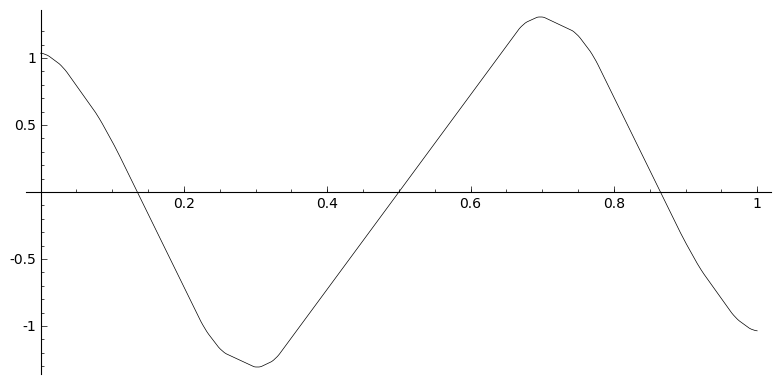}
\end{minipage}
\begin{minipage}{.5\textwidth}
\centering
\includegraphics[width=\textwidth]{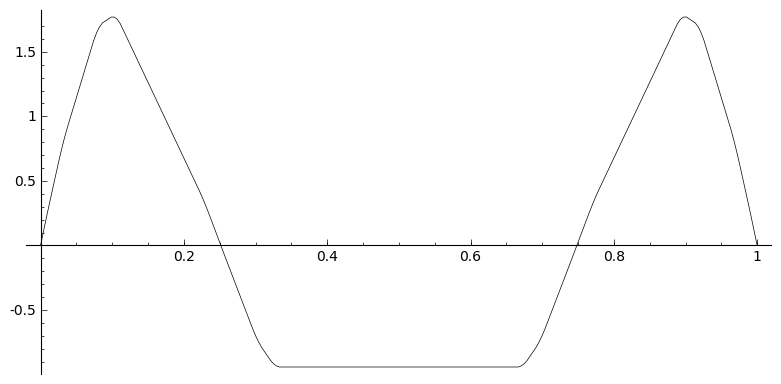}
\end{minipage}
\begin{minipage}{.5\textwidth}
\centering
\includegraphics[width=\textwidth]{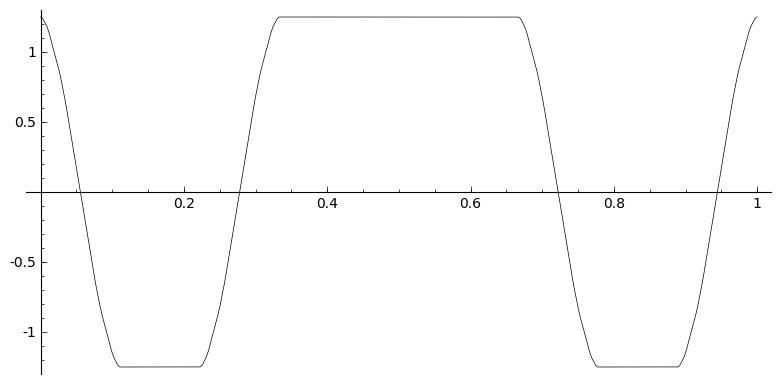}
\end{minipage}
\begin{minipage}{.5\textwidth}
\centering
\includegraphics[width=\textwidth]{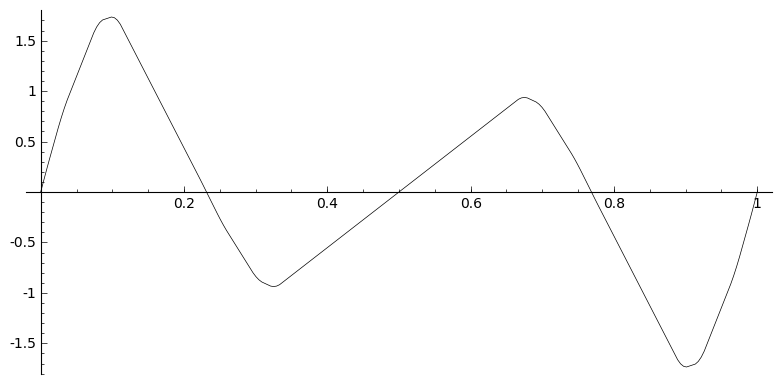}
\end{minipage}
\captionof{figure}{\label{fig:eigenfunctions}The first four Neumann (left) and Dirichlet (right) eigenfunctions\\ for $r_1=r_2=\frac{1}{3}$ and $m_1=m_2=\frac{1}{2}$.}

\begin{table}
\centering
\begin{tabular}{>{$}r<{$} | >{$}r<{$} |>{$}r<{$} p{1cm} >{$}r<{$} | >{$}r<{$}| >{$}r<{$}}
\multicolumn{1}{c|}{$m$}&\multicolumn{1}{c|}{$\l_{N,m}$}&\multicolumn{1}{c}{$a$}&\multicolumn{1}{p{1cm}}{}&\multicolumn{1}{c|}{$m$}&\multicolumn{1}{c|}{$\l_{N,m}$}&\multicolumn{1}{c}{$a$}\\
\cline{1-3} \cline{5-7}
1		&       7.09743\,10981\,41122	&	  12	&&	17	&   9211.73939\,77562\,51229	&	86\\
2		&	     42.58458\,65888\,46733	&   19	&&	18	&   9288.33494\,53277\,71442	& 85\\
3		&	     61.34420\,39227\,01662	&   19	&&	19	&   9316.34702\,24100\,75024	& 85\\
4		&     255.50751\,95330\,80403	&	  28	&&	20	&   9827.40854\,94892\,99413	& 87\\
5		&	    272.98357\,08191\,47205	&   28	&&	21	&   9847.99008\,31996\,68501	& 87\\
6		&     368.06522\,35362\,09975	&	  30	&&	22	&   9975.76460\,05394\,58261	&	87\\
7		&     383.55288\,31276\,93176	&	  31	&&	23	&   9994.03735\,25970\,68208	& 87\\
8		&    1533.04511\,71984\,82423	&		47	&&	24	&  13250.34804\,73035\,59112	& 97\\
9		&    1548.05582\,42212\,95240	&		47	&&	25	&  13260.71659\,87844\,44965	& 96\\
10	&    1637.90142\,49148\,83235	&		48	&&	26	&  13324.61669\,98067\,78407	& 97\\
11	&	   1662.62743\,34232\,43043	&		48	&&	27	&  13342.22766\,88915\,03102	& 97\\
12	&    2208.39134\,12172\,59852	&		53	&&	28	&  13807.90379\,25969\,54355	& 97\\
13	&    2220.76944\,99677\,96401	&		53	&&	29	&  13816.72725\,09206\,34538	& 98\\
14	&    2301.31729\,87661\,59059	&		53	&&	30	&  13875.49272\,43655\,06427	& 98\\
15	&    2312.58212\,07275\,84404	&		53	&&	31	&  13883.67238\,03565\,18424	& 97\\
16	&    9198.27070\,31908\,94542	&		85	&&	32	&  55189.62421\,91453\,67256	& 160
\end{tabular}
\caption{Neumann eigenvalues of $-\frac{d}{d\m}\frac{d}{dx}$ for $r_1=r_2=\frac{1}{3}$ and $m_1=m_2=\frac{1}{2}$.}
\label{tab:nev1312}
\end{table}

\begin{table}
\centering
\begin{tabular}{>{$}r<{$} | >{$}r<{$} |>{$}r<{$}|>{$}r<{$} p{0.3cm} >{$}r<{$} | >{$}r<{$}| >{$}r<{$}|>{$}r<{$}}
\multicolumn{1}{c|}{$m$}&\multicolumn{1}{c|}{$\norm{f_{N,m}}{2}$}&\multicolumn{1}{c|}{$\norm{f_{N,m}}{\infty}$}&\multicolumn{1}{c}{$\norm{\tilde{f}_{N,m}}{\infty}$}&\multicolumn{1}{p{0.3cm}}{}&\multicolumn{1}{c|}{$m$}&\multicolumn{1}{c|}{$\norm{f_{N,m}}{2}$}&\multicolumn{1}{c|}{$\norm{f_{N,m}}{\infty}$}&\multicolumn{1}{c}{$\norm{\tilde{f}_{N,m}}{\infty}$}\\
\cline{1-4} \cline{6-9}
1		&	0.801		&	1.000			&	1.248		&& 17	& 0.666		&	1.001 & 1.503	\\
2		&	0.801		&	1.000			&	1.248		&& 18	& 0.687		& 1.007	& 1.467	\\
3		& 0.966		& 1.261			& 1.306		&& 19	& 0.829		& 1.307 & 1.577	\\
4		&	0.801		&	1.000			&	1.248		&& 20	& 0.746		& 1.049	&	1.405	\\
5		& 0.746		& 1.049			&	1.405		&& 21	& 0.688		& 1.093 & 1.588	\\
6		& 0.966		& 1.261			& 1.306		&& 22	& 0.897		& 1.356	& 1.512	\\
7		& 1.145		& 1.604			& 1.401 	&& 23	& 1.057		& 1.703	& 1.612	\\
8		&	0.801		&	1.000			&	1.248		&& 24	& 0.966		& 1.261	& 1.306	\\
9		& 0.687		& 1.007			& 1.467		&& 25	& 0.826		&	1.262	&	1.529	\\
10	& 0.746		& 1.049			&	1.405		&& 26	& 0.886		& 1.306	& 1.474	\\
11	& 0.897		& 1.356			& 1.512		&& 27	& 1.063		& 1.694	& 1.594	\\
12	& 0.966		& 1.261			& 1.306		&& 28	& 1.145		& 1.604	& 1.401	\\
13	& 0.886		& 1.306			& 1.474		&& 29	& 1.049		& 1.656	&	1.579	\\
14	& 1.145		& 1.604			& 1.401		&& 30	& 1.346		& 2.029	& 1.508	\\
15	& 1.346		& 2.029			& 1.508		&& 31	& 1.579		& 2.563	& 1.625	\\
16	&	0.801		&	1.000			&	1.248		&& 32	& 0.801		&	1.000	&	1.248	
\end{tabular}
\caption{Norms of Neumann eigenfunctions for $r_1=r_2=\frac{1}{3}$ and $m_1=m_2=\frac{1}{2}$.}
\label{tab:nef1312}
\end{table}

\begin{table}
\centering
\begin{tabular}{>{$}r<{$} | >{$}r<{$} |>{$}r<{$} p{0.5cm} >{$}r<{$} | >{$}l<{$}| >{$}r<{$}}
\multicolumn{1}{c|}{$m$}&\multicolumn{1}{c|}{$\l_{D,m}$}&\multicolumn{1}{c}{$a$}&\multicolumn{1}{p{0.5cm}}{}&\multicolumn{1}{c|}{$m$}&\multicolumn{1}{c|}{$\l_{D,m}$}&\multicolumn{1}{c}{$a$}\\
\cline{1-3} \cline{5-7}
1		&      14.43524\,05120\,53874	&	  13	&&	17	&  \phantom{0} 9233.86793\,80086\,63779	&	84\\
2		&	     35.26023\,80242\,77225	&   16	&&	18	&  \phantom{0} 9271.62879\,27212\,74161	& 83\\
3		&	    140.78105\,33845\,56059	&   24	&&	19	&  \phantom{0} 9589.26839\,61415\,98781	& 85\\
4		&     151.29061\,60550\,19631	&	  23	&&	20	&  \phantom{0} 9598.24041\,25849\,12727	& 85\\
5		&	    326.05732\,83577\,53770	&   29	&&	21	&  \phantom{0} 9923.46445\,25858\,18608	& 85\\
6		&     353.41692\,07675\,57756	&	  29	&&	22	&  \phantom{0} 9957.06520\,21538\,29857	&	87\\
7		&     876.27445\,96020\,73755	&	  39	&&	23	&  12190.28558\,35702\,41470	& 93\\
8		&     876.50531\,85096\,60313	&		39	&&	24	&  12190.29241\,90995\,34112	& 94\\
9		&    1581.17702\,42871\,45662	&		46	&&	25	&  13284.12682\,48731\,70732	& 94\\
10	&    1619.40072\,91584\,24238	&		46	&&	26	&  13311.27448\,40460\,62950	& 95\\
11	&	   2029.61356\,34510\,19039	&		51	&&	27	&  13668.53690\,39463\,19748	& 96\\
12	&    2033.85281\,30577\,61437	&		51	&&	28	&  13671.26816\,68726\,11762	& 96\\
13	&    2268.79163\,36445\,60767	&		53	&&	29	&  13851.83951\,26643\,76419	& 96\\
14	&    2289.60406\,94424\,69130	&		52	&&	30	&  13866.93782\,41331\,73771	& 96\\
15	&    5258.33939\,69212\,17309	&		71	&&	31	&  31550.03640\,02815\,21874\,64223\,25788	& 139\\
16	&    5258.33940\,31726\,23308	&		71	&&	32	&  31550.03640\,02815\,21874\,89689\,65410	& 139
\end{tabular}
\caption{Dirichlet eigenvalues of $-\frac{d}{d\m}\frac{d}{dx}$ for $r_1=r_2=\frac{1}{3}$ and $m_1=m_2=\frac{1}{2}$.}
\label{tab:dev1312} 
\end{table}

\begin{table}
\centering
\begin{tabular}{>{$}r<{$} | >{$}r<{$} |>{$}r<{$}|>{$}r<{$} p{0.3cm} >{$}r<{$} | >{$}r<{$}| >{$}r<{$}|>{$}r<{$}}
\multicolumn{1}{c|}{$m$}&\multicolumn{1}{c|}{$\norm{f_{D,m}}{2}$}&\multicolumn{1}{c|}{$\norm{f_{D,m}}{\infty}$}&\multicolumn{1}{c}{$\norm{\tilde{f}_{D,m}}{\infty}$}&\multicolumn{1}{p{0.3cm}}{}&\multicolumn{1}{c|}{$m$}&\multicolumn{1}{c|}{$\norm{f_{D,m}}{2}$}&\multicolumn{1}{c|}{$\norm{f_{D,m}}{\infty}$}&\multicolumn{1}{c}{$\norm{\tilde{f}_{D,m}}{\infty}$}\\
\cline{1-4} \cline{6-9}
1		&	0.627		&	0.920		&	1.469		&&	17	& 8.492		& 15.635		& 1.841	\\
2		&	0.711		&	0.985		&	1.387		&&	18	& 7.115		& 12.394		& 1.742	\\
3		& 0.446		& 0.790		& 1.770		&&	19	& 1.685		&  3.996		& 2.372	\\
4		&	0.457		&	0.793		&	1.734		&& 	20	& 1.679		&  3.985		& 2.374	\\
5		& 1.115		& 1.628		&	1.461		&& 	21	& 5.110		& 10.266		& 2.009	\\
6		& 1.273		& 2.105		& 1.654		&& 	22	& 5.787		& 11.252		& 1.944	\\
7		& 0.262		& 0.646		& 2.469		&& 	23	& 0.415		&  1.195		& 2.883	\\
8		&	0.262		&	0.646		&	2.468		&& 	24	& 0.415		&  1.306		& 3.151	\\
9		& 2.798		& 5.034		& 1.799		&& 	25	& 9.565		& 18.428		& 1.927	\\
10	& 2.656		& 4.694		&	1.767		&& 	26	& 8.944		& 16.597		& 1.856	\\
11	& 0.719		& 1.460		& 2.032		&& 	27	& 1.950		&  4.852		& 2.489	\\
12	& 0.717		& 1.602		& 2.233		&& 	28	& 1.945		&  4.863		& 2.501	\\
13	& 3.048		& 5.661		& 1.857		&& 	29	& 8.777		& 15.403		& 1.755	\\
14	& 3.481		& 6.296		& 1.809		&& 	30	& 9.995		& 19.451		& 1.946	\\
15	& 0.151		& 0.509		& 3.369		&& 	31	& 0.087		&  0.416		& 4.765	\\
16	&	0.151		&	0.528		&	3.491		&& 	32	& 0.087		&  0.416		& 4.765
\end{tabular}
\caption{Norms of Dirichlet eigenfunctions for $r_1=r_2=\frac{1}{3}$ and $m_1=m_2=\frac{1}{2}$.}
\label{tab:def1312} 
\end{table}
\end{exa}

\pagebreak

\begin{exa}\label{exa2}
For the next example, we take the asymmetric self-similar measure with $r_1=1/3$, $r_2=1/4$, $m_1=\dfrac{1}{3^{d_H}}$ and $m_2=\dfrac{1}{4^{d_H}}$ where 
$d_H$ is the Hausdorff dimension of the invariant set. That is, $d_H$ is the solution of the equation
\[\frac{1}{3^{d_H}}+\frac{1}{4^{d_H}}=1.\]
For the calculations we used $0.56049886522386387883902233$ for $d_H$. Variation of this value led to no change in the first 15 digits of the eigenvalues.
Plots of $\sin_{\m}^N$ and $\sin_{\m}^D$ are shown in Figure \ref{fig:1314hd_sinp_sinq_bis_48} and the first eigenvalues exact to 15 decimal places are displayed
in Table \ref{tab:nev1314H}. Note that here $m_1r_1 \neq m_2r_2$. There seem to be no fixed order of Neumann and Dirichlet eigenvalues as in Example \ref{exa1} and there are no clear pairings of the values.

\begin{figure}
	\centering
		\includegraphics[width=\textwidth]{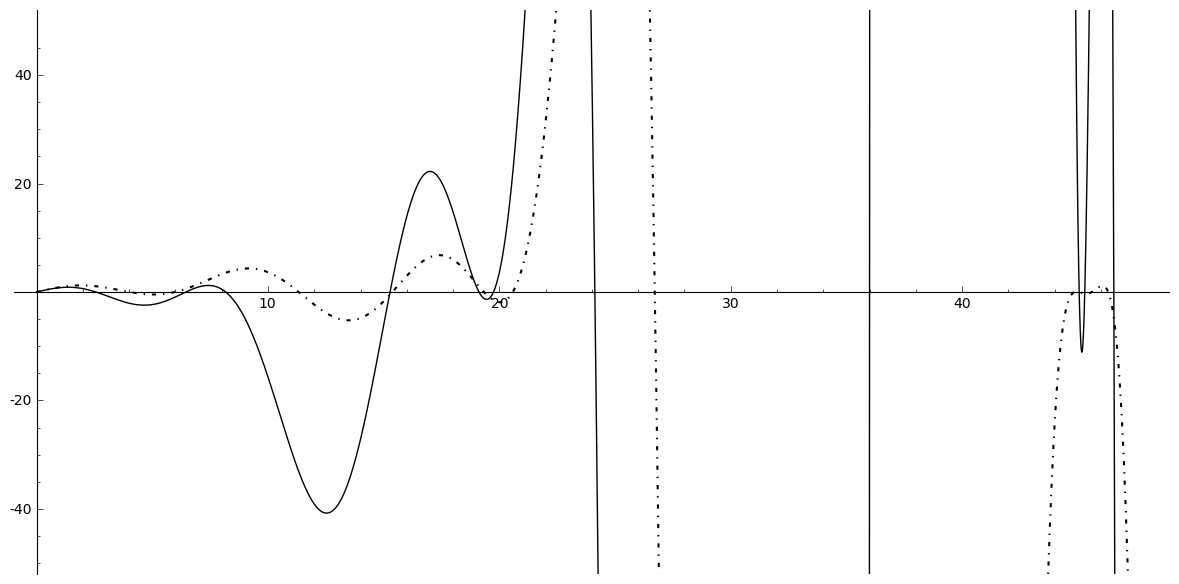}
	\caption{$\sin_{\m}^N$ (solid) and $\sin_{\m}^D$ (dash-dot) for $r_1=1/3$, $r_2=1/4$, $m_1=\dfrac{1}{3^{d_H}}$ and $m_2=\dfrac{1}{4^{d_H}}$}
	\label{fig:1314hd_sinp_sinq_bis_48}
\end{figure}

\begin{table}
\centering
\begin{tabular}{>{$}r<{$} | >{$}r<{$} |>{$}r<{$} p{0.5cm} >{$}r<{$} | >{$}r<{$}| >{$}r<{$}}
\multicolumn{1}{c|}{$m$}&\multicolumn{1}{c|}{$\l_{N,m}$}&\multicolumn{1}{c}{$a$}&\multicolumn{1}{p{0.5cm}}{}&\multicolumn{1}{c|}{$m$}&\multicolumn{1}{c|}{$\l_{D,m}$}&\multicolumn{1}{c}{$a$}\\
\cline{1-3} \cline{5-7}
1		&    6.567037965687942	&	 11		&&	1		&   16.107849410419070	& 12	\\
2		&   41.632795946820830	&	 16		&&  2		&   35.907601066462638	& 15	\\
3		&   66.822767372091789	&  19		&&  3   &  128.330447556120622	& 21	\\
4		&  233.355013145153884	&	 24		&&  4		&  236.463676343561213	& 24	\\
5		&  365.584215801794021	&  27		&&  5		&  373.701929431216995	& 27	\\
6		&  389.945618826510339	&  28		&&  6		&  423.638157028808414	& 28	\\
7		&  582.138208794906725  &  30		&&  7		&  713.786986198043209	& 31	\\
8		& 1295.888937033626505  &  37		&&  8		& 2013.164883016581104	& 44
\end{tabular}
\caption{Neumann and Dirichlet eigenvalues for $r_1=1/3$, $r_2=1/4$, $m_1=\dfrac{1}{3^{d_H}}$ and $m_2=\dfrac{1}{4^{d_H}}$.}
\label{tab:nev1314H} 
\end{table}
\end{exa}

\begin{exa}\label{exa3}

Figure \ref{fig:1314sinp_sinq_bis_59} shows plots of $\sin_{\m}^N$ and $\sin_{\m}^D$ for 
$r_1=\frac{1}{3}$, $r_2=\frac{1}{4}$ and
 $m_1=\frac{3}{7}$, $m_2=\frac{4}{7}$. The invariant set is geometrically the same as in Example \ref{exa2}, but $m_1$ and $m_2$ are chosen 
such that $r_1m_1=r_2m_2=\frac{1}{7}$ and thus, $\l_{N,2m} = 7 \cdot \l_{N,m}$.

Comparing with Example \ref{exa1}, we observe that the Neumann eigenvalues behave qualitatively similar, but the Dirichlet eigenvalues do not appear
 in such close pairs. However, it seems to hold again, that two Neumann and two Dirichlet eigenvalues appear in turns.

\begin{figure}
	\centering
		\includegraphics[width=\textwidth]{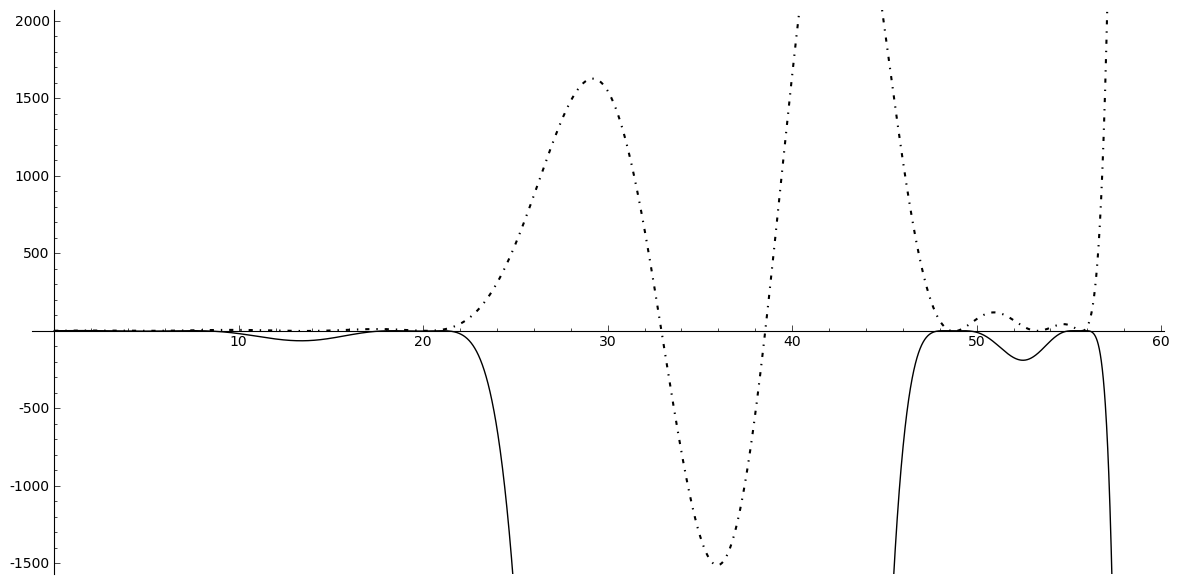}
	\caption{$\sin_{\m}^N$ (solid) and $\sin_{\m}^D$ (dash-dot) for$r_1=1/3$, $r_2=1/4$, $m_1=\frac{3}{7}$ and $m_2=\frac{4}{7}$.}
	\label{fig:1314sinp_sinq_bis_59}
\end{figure}

\begin{table}
\centering
\begin{tabular}{>{$}r<{$} | >{$}r<{$} |>{$}r<{$} p{0.5cm} >{$}r<{$} | >{$}r<{$}| >{$}r<{$}}
\multicolumn{1}{c|}{$m$}&\multicolumn{1}{c|}{$\l_{N,m}$}&\multicolumn{1}{c}{$a$}&\multicolumn{1}{p{0.5cm}}{}&\multicolumn{1}{c|}{$m$}&\multicolumn{1}{c|}{$\l_{D,m}$}&\multicolumn{1}{c}{$a$}\\
\cline{1-3} \cline{5-7}
1		&     6.752284245618646		&	 10		&&	1		&    16.452512161464721	& 12	\\
2		&    47.265989719330522		&  16		&&  2		&    36.904245287406090	& 15	\\
3		&    62.066872795561511		&  18		&&  3		&   154.577520453343494	& 21	\\
4		&   330.861928035313659		&  26		&&  4		&   212.376524344704458	& 23	\\
5		&   345.194670941772007		&  27		&&  5		&   395.526819249411977	& 27	\\
6		&   434.468109568930577		&  28		&&  6		&   417.532700806716224	& 27	\\
7		&   446.407999438501248		&  28		&&  7		&  1083.253271255975735	& 34	\\
8		&  2316.033496247195616		&  45		&&  8		&  1485.470110503836517	& 37	\\
9		&  2332.825185220436900		&  46		&&  9		&  2360.481274606702758	& 44	\\
10	&  2416.362696592404055		&  46		&& 10		&  2397.801276276128236	& 44	\\
11	&  2434.484694248270572		&  46		&& 11		&  2830.491432008378221	& 47	\\
12	&  3041.276766982514042		&  50		&& 12		&  2850.987710468049166	& 47	\\
13	&  3051.736543145083444		&  50		&& 13		&  3093.525406096403347	& 48	\\
14	&  3124.855996069508739		&  50		&& 14		&  3111.593713450879200	& 48	\\
15	&  3133.914016082441210		&  49		&& 15		&  7582.772906434721944	& 58	\\
16	& 16212.234473730369315		&  82		&& 16		& 10398.290767742394136	& 64
\end{tabular}
\caption{Neumann and Dirichlet eigenvalues for $r_1=1/3$, $r_2=1/4$, $m_1=\frac{3}{7}$ and $m_2=\frac{4}{7}$.}
\label{tab:ev1314} 
\end{table}
\end{exa}

\begin{exa}\label{exa4}

We choose the measure with $r_1=0.6$, $r_2=0.4$, $m_1=0.4$ and $m_2=0.6$. This measure is supported on the whole interval $[0,1]$, yet is singular to the Lebesgue measure. In Theorem \ref{th:DirandNeusame} we showed that $\sin_{\m}^N$ and $\sin_{\m}^D$ and thus the Dirichlet and Neumann eigenvalues coincide. In Figure \ref{fig:sinp0604} a plot of $\sin_{\m}^N$ is displayed. It is comparable to the sine function, which we would get for $r_1=r_2=m_1=m_2= 0.5$. Table \ref{tab:ev0604} contains the first 16 eigenvalues. 

\begin{figure}
	\centering
		\includegraphics[width=\textwidth]{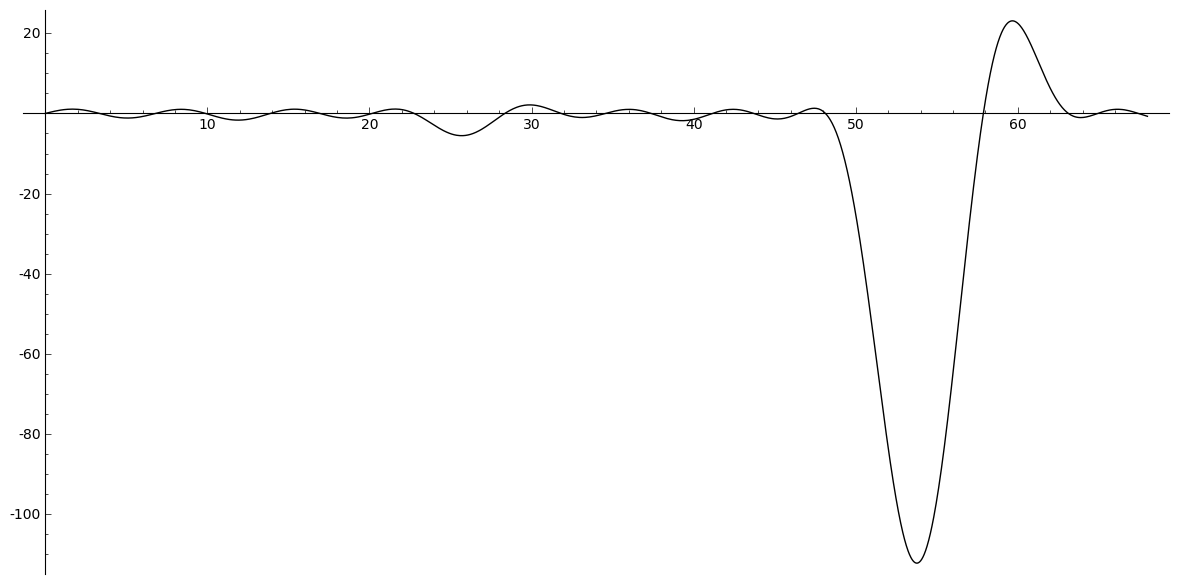}
	\caption{$\sin_{\m}^N$ (coincides with $\sin_{\m}^D$) for $r_1=0.6$, $r_2=0.4$, $m_1=0.4$ and $m_2=0.6$.}
	\label{fig:sinp0604}
\end{figure}

\begin{table}
\centering
\begin{tabular}{>{$}r<{$} | >{$}r<{$} |>{$}r<{$} p{0.5cm} >{$}r<{$} | >{$}l<{$}| >{$}r<{$}}
\multicolumn{1}{c|}{$m$}&\multicolumn{1}{c|}{$\l_{m}$}&\multicolumn{1}{c}{$a$}&\multicolumn{1}{p{0.5cm}}{}&\multicolumn{1}{c|}{$m$}&\multicolumn{1}{c|}{$\l_{m}$}&\multicolumn{1}{c}{$a$}\\
\cline{1-3} \cline{5-7}
1		&     11.113238313123921	&	  13	&&	9		&	  1012.173153820335730	&   54	\\
2		&	    46.305159638016340	&		19	&&	10	&   1194.689209582619744	&		57	\\
3		&     97.600761284513435	&		24	&&	11	&   1396.721102656337624	&		60	\\
4		&    192.938165158401419	&		29	&&	12	&		1694.457661189469363	&		65	\\
5		&    286.725410299828738	&		34	&&	13	&		1910.161155469398890	&		67	\\
6		&    406.669838685472647	&		38	&&	14	&		2157.157626864968439	&   70	\\
7		&		 517.717830447592425	&		41	&&	15	&   2316.668360848575284  &		73	\\
8		&    803.909021493339246	&		48	&&	16	&   3349.620922888913526	&   83
\end{tabular}
\caption{Neumann (and Dirichlet) eigenvalues for $r_1=0.6$, $r_2=0.4$, $m_1=0.4$ and $m_2=0.6$.}
\label{tab:ev0604} 
\end{table}
\end{exa}

\begin{exa}\label{exa5}

We take $r_1=0.9$, $r_2=0.1$, $m_1=0.1$ and $m_2=0.9$. The resulting measure is supported on $[0,1]$ as in Example \ref{exa4}, but in Figure \ref{fig:0901sinp} we see that $\sin_{\m}^N$ looks very different from the sine function. Table \ref{tab:ev0901} contains the first 16 eigenvalues up to 15 decimal places.

\begin{figure}
	\centering
		\includegraphics[width=\textwidth]{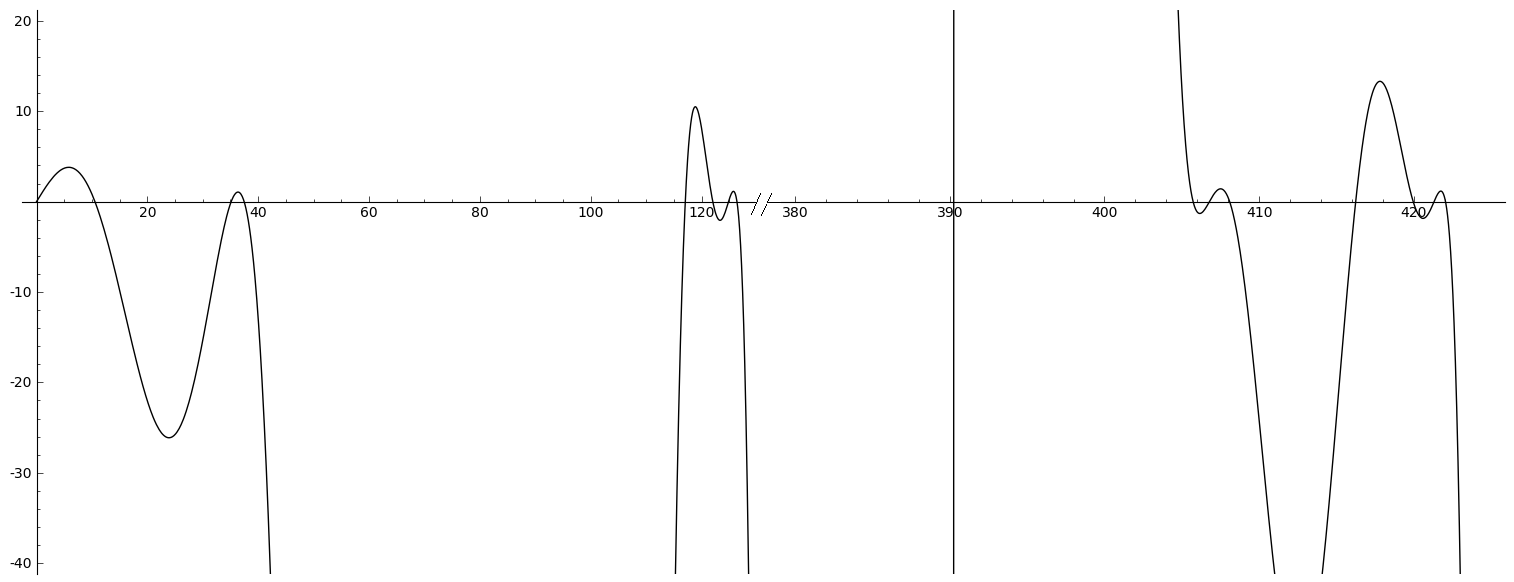}
	\caption{$\sin_{\m}^N$ (coincides with $\sin_{\m}^D$) for $r_1=0.9$, $r_2=0.1$, $m_1=0.1$ and $m_2=0.9$.}
	\label{fig:0901sinp}
\end{figure}

\begin{table}
\centering
\begin{tabular}{>{$}r<{$} | >{$}r<{$} |>{$}r<{$} p{0.5cm} >{$}r<{$} | >{$}r<{$}| >{$}r<{$}}
\multicolumn{1}{c|}{$m$}&\multicolumn{1}{c|}{$\l_{m}$}&\multicolumn{1}{c}{$a$}&\multicolumn{1}{p{0.5cm}}{}&\multicolumn{1}{c|}{$m$}&\multicolumn{1}{c|}{$\l_{m}$}&\multicolumn{1}{c}{$a$}\\
\cline{1-3} \cline{5-7}
1		&    111.021168159382246	&	 9	&&	9		&  164619.744662988877161	& 44	\\
2		&	  1233.568535104247184	&	15	&&	10	&  165477.638319057818349	& 43	\\
3		&   1403.454381590697316	& 15	&&	11	&  166543.871983977116823	& 43	\\
4		&  13706.317056713857601	& 24	&&	12	&  173265.973035888557594	& 43	\\
5		&  14892.987448715203651	& 24	&&	13	&  176376.608221577384951	&	43	\\
6		&  15593.937573229970183	& 25	&&	14	&  177512.483919664028463	&	44	\\
7		&  15976.123552769762561	& 25	&&	15	&  178137.700783469958606	&	44	\\
8		& 152292.411741265084466	& 40	&&	16	& 1692137.908236278716292	&	72
\end{tabular}
\caption{Neumann (and Dirichlet) eigenvalues for $r_1=0.9$, $r_2=0.1$, $m_1=0.1$ and $m_2=0.9$.}
\label{tab:ev0901} 
\end{table}
\end{exa}

\newpage

\section{Remarks and outlook}
\label{sec:rem}

In this section we state several remarks and thoughts that could be subject of future studies. 

\begin{conj}
Due to the examination of several examples (see e.g. Examples \ref{exa1}, \ref{exa3}, \ref{exa4} and \ref{exa5}) we conjecture that in case of a self-similar measure $\m$ with $r_1m_1=r_2m_2$ the Neumann and Dirichlet eigenvalues satisfy
\[ \l_{N,0} < \l_{N,1} < \l_{D,1} < \l_{D,2} < \l_{N,2} < \l_{N,3} < \l_{D,3} < \l_{D,4} < \dots.\]
\end{conj}

\begin{rem2}
It would be very interesting to find out, if there was a relation between our sequences $(p_n)_n$ and $(q_n)_n$ to any known number sequences as e.g. Bernoulli or Euler numbers. Indeed, the definition of $p_n(x)$ or $q_n(x)$ (Definition \ref{defpq}) is reminiscent of the recursive definition of the Euler polynomials $E_n(x)$ by $E_0(x):=1$ and
\[ E_n(x):= \int_c^x n E_{n-1}(t)\,dt,\]
where $c= \frac{1}{2}$ if $n$ is odd and $c=0$ for even $n$. Then the $n$th Euler number is $E_n=2^n E_n(1/2)$.

Furthermore, Equation \eqref{evreceq} has a similar structure as the recursion rule 
\[\alpha_n =\frac{1}{2n} \sum_{j=0}^{n-1} \alpha_j \, \alpha_{n-1-j}\]
with $\alpha_0=\alpha_1=1$, where $\alpha_n = \frac{1}{n!} \abs{E_{2n}}$.
\end{rem2}

\begin{rem2}
One could investigate the functional equations in Theorem \ref{th1} further. In the simple case where $r_1=r_2=\frac{1}{3}$ and $m_1=m_2=\frac{1}{2}$, for instance, we can transform them (after eliminating terms containing $\sin_{\m}^D$ by using $\cos_{\m}^N(z)^2+\sin_{\m}^N(z)\,\sin_{\m}^D(z)=1$)  with the abbreviations $u(z)=z\,\sin_{\m}^N(z)$ and $v(z) = 2 \, \cos_{\m}^N(z)$ to
\begin{align*}
u\bigl(\sqrt{6}z\bigr)	& = 3\, u(z)\,v(z)\ -\hspace{0.6cm} 3\, u(z)^2\\
v\bigl(\sqrt{6}z\bigr)	& = \phantom{3\, u(z)\,} v(z)^2 - v(z)\,u(z)-2.
\end{align*}
From this one can derive recursion formulas for the sequence $(p_n)_n$ that contain only members of $p_n$ and not, as in Corollary \ref{rec}, both $p_n$ and $q_n$.

Furthermore, it could be possible to somehow solve these functional equations to get a more direct representation of $\sin_{\m}^N$ and $\cos_{\m}^N$.
\end{rem2}

\begin{rem2}
We defined our functions $\sin_{\m}^N$, $\sin_{\m}^D$, $\cos_{\m}^N$ and $\cos_{\m}^D$ only for real arguments. However, one can just allow the argument to be complex. Then these power series can be treated with methods of complex analysis.
\end{rem2}

\begin{rem2}
It is also interesting to consider the eigenvalue problem
\[ \frac{d}{d\mu}\frac{d}{d\nu} f = - \l f\]
with appropriate boundary conditions, where both $\m$ and $\nu$ are non-atomic finite Borel measures. This can be done by modifying the above considerations by replacing $\l$ with $\nu$. 

The case where both derivatives are with respect to the same measure, that is, $\mu=\nu$ is much simpler.
There we get
\[ \c_{\m,\m}(z,x)=\cos\bigl(z\,p_1(x)\bigr)\]
and
\[ \s_{\m,\m}(z,x)=\sin\bigl(z\,p_1(x)\bigr).\]
The eigenvalues are $\l_k=k^2\pi^2$, $k\in \N$, as in the classical Lebesgue measure case. This is treated in Arzt and Freiberg  \cite{arzt:ddmddm}. 
See also Freiberg and Zähle \cite{freiberg:harmonic}.

\end{rem2}

\begin{rem2}
Our recursion law for $p_n$ and $q_n$ works only for self-similar measures with $r_1+r_2 \leq 1$. It would be interesting to develop similar formulas for measures with overlaps, i.e. with $r_1+r_2 > 1$. Such measures are treated for example in Ngai \cite{ngai:overlaps} and Chen and Ngai \cite{chen:eigenvalues}, which contains, in particular, numerical solutions of the eigenvalue problem by the finite elements method.
\end{rem2}

\begin{rem2}
In this work, we examined the eigenvalues of $-\frac{d}{d\m}\frac{d}{dx}$ by following the basic lines of the treatment of the classical second derivative operator on the interval. In this classical case all eigenvalues are multiples of $\pi^2$ and have therefore direct representations in many forms, e.g. by using the series expansion of $\arctan$. Maybe one can find a series representation of eigenvalues of the generalized operator, too, by using such functions as $\sin_{\m}^N$, $\sin_{\m}^D$, $\cos_{\m}^N$ and $\cos_{\m}^D$.
\end{rem2}

\begin{rem2}
In Corollary \ref{cor:upperandlowerest} we stated upper and lower estimates for subsequences $\bigl( \norm{\tilde{f}_{2^k l}}{\infty}\bigr)_k$, $l$ odd, of the suprema of the normed eigenfunctions. We have no information about the growth of the sequence $\bigl( \norm{\tilde{f}_{2k+1}}{\infty}\bigr)_k$, though.

Such estimates could be used to prove estimates of the heat kernel
\[K(t,x,y) = \sum_{m=1}^\infty e^{-\l_m t}\tilde{f}_m(x)\,\tilde{f}_m(y)\]
for the corresponding quasi-diffusion process. This process has been investigated for example in Löbus \cite{lobus:diffusions} and 
Küchler \cites{kuchler:asymptotics, kuchler:sojourn}.
\end{rem2}

\begin{rem2}
We used the functions $p_n(x)$ and $q_n(x)$, $x\in [0,1]$, defined in Definition~\ref{defpq} to replace monomials $\frac{1}{n!}x^n$ in the classical case. One could use these functions to build a kind of generalized polynomials that are adjusted to the measure $\m$. For instance, we take the sequence
\[ \tilde{P}_0(x) = 1, \quad \tilde{P}_1(x) = q_1(x), \quad \tilde{P}_2(x) = p_2(x), \quad \tilde{P}_3(x)= q_3(x), \quad \dotsc\]
and orthogonalize it in $L_2(\m)$ by using the Gram-Schmidt process. We take odd numbered $q_n(x)$ and even numbered $p_n(x)$, because they are the building blocks for the eigenfunctions $\s_{\l,\m}(z,\cdot)$ and $\c_{\l,\m}(z,\cdot)$ and they are continuously Lebesgue-differentiable, namely 
\[ q_n'(x)=q_{n-1}(x) \quad \text{for odd $n$}\]
and 
\[ p_n'(x)=p_{n-1}(x) \quad \text{for even $n$.}\]
Furthermore, we can $\m$-integrate them and get 
\[ \int_0^x q_n(t)\,d\m(t) = q_{n+1}(x) \quad \text{for odd $n$}\]
and
\[ \int_0^x p_n(t)\,d\m(t) = p_{n+1}(x) \quad \text{for even $n$}.\]
With that we can apply the generalized integration by parts rule from Lemma \ref{parts} to do the calculations in the Gram-Schmidt algorithm. Note again that we use the notation $p_n:=p_n(1)$ and $q_n:=q_n(1)$ and assume that those numbers are given since we have a recursion rule in the self-similar case. Then we get
\[ P_0(x):= 1,\]
\[ P_1(x):= q_1(x) - \int_0^1 q_1(t)\, d\m(t) = q_1(x) - q_2,\]
\[ P_2(x):= p_2(x) - \int_0^1 p_2(t)\, d\m(t) - \frac{\ds \int_0^1 \bigl(q_1(t)-q_2\bigr)\,p_2(t)\,d\m(t)}{\ds\int_0^1 \bigl(q_1(t)- q_2\bigr)^2\,d\m(t)}
						\bigl(q_1(x)-q_2\bigr).\]
We calculate
\begin{align*}
\int_0^1 \bigl(q_1(t)-q_2\bigr)\,p_2(t)\,d\m(t)	& = \int_0^1 q_1(t)\,p_2(t)\, d\m(t)-\int_0^1 q_2\,p_2(t)\,d\m(t)\\
																								& = q_1(t)\,p_3(t)\Bigr|_0^1 - \int_0^1 p_3(t)\,dt - q_2p_3\\
																								& = q_1p_3 - p_4 - q_2p_3
\end{align*}
and
\begin{align*}
\int_0^1 \bigl(q_1(t)- q_2\bigr)^2\,d\m(t)	& = \int_0^1  q_1(t)^2\,d\m(t) - 2 q_2\int_0^1 q_1(t)\,d\m(t) +  q_2^2 \int_0^1\,d\m(t)\\
																						& = q_1q_2- \int_0^1 q_2(t)\,dt - 2q_2^2 + q_2^2 p_1\\
																						& = q_1q_2 - q_3 - 2 q_2^2 + q_2^2p_1.
\end{align*}
To simplify these expressions a bit we utilize $p_1=q_1=1$ which follows from the definition and $p_2+q_2=1$ which follows from Corollary \ref{sumpqgleich0} by putting $n=1$.
Then we get
\[ P_2(x) = p_2(x) - \frac{p_4-p_2\,p_3}{q_3- p_2\,q_2}\,q_1(x) + \frac{q_2\,p_4- q_3\,p_3}{q_3-p_2\,q_2}.\]
In this fashion one can calculate a sequence of $L_2(\m)$-orthogonal ``polynomials''. 

As an example we take the Lebesgue measure for $\m$ and put $p_n(x)=q_n(x)=\frac{1}{n!}x^n$. Then
\[ P_0(x)=1, \quad P_1(x)= x-\frac{1}{2}, \quad P_2(x) = \frac{1}{2}x^2- \frac{1}{2}x + \frac{1}{12}\]
which are the first Legendre polynomials on $[0,1]$ (not normed).

If $\m$ is the standard Cantor measure, then $p_2=q_2=\frac{1}{2}$, $p_3=\frac{1}{5}$ and $q_3 = \frac{1}{8}$ and we get
\[P_0(x)=1, \quad P_1(x)= q_1(x) - \frac{1}{2}, \quad P_2(x) = p_2(x) - \frac{1}{2} q_1(x) + \frac{1}{20}.\]
Maybe one can use these functions for further analytical studies.
\end{rem2}

\begin{rem2}
With the presented methods one could investigate not only the equation $\frac{d}{d\m}f'=-\l f$, but maybe other differential equations on the interval $[0,1]$ that are generalized involving a self-similar measure $\m$.
\end{rem2}

\begin{rem2}[Fourier series]
It is well known that the normed eigenfunctions $(\tilde{f}_{N,k})_{k=0}^\infty$ and $(\tilde{f}_{D,k})_{k=1}^\infty$ form  orthonormal bases in $L_2(\m)$ (see \cite{freiberg:analytical}).

We denote $n_{N,k}:= \norm{\c_{\l,\m}(\sqrt{\l_{N,k}}, \cdot )}{L_2(\m)}$ and $n_{D,k}:= \norm{\s_{\l,\m}(\sqrt{\l_{D,k}},\cdot)}{L_2(\m)}$ so that
\[ \tilde{f}_{N,k} = \frac{1}{n_{N,k}} \c_{\l,\m}(\sqrt{\l_{N,k}}, \cdot )\]
and
\[ \tilde{f}_{D,k} = \frac{1}{n_{D,k}} \s_{\l,\m}(\sqrt{\l_{D,k}},\cdot).\]
We decompose some functions $f\in L_2(\m)$ into series of eigenfunctions (Fourier series), ignoring questions about convergence for the moment. 
Assume that for $x\in [0,1]$
\[ f(x) = \sum_{k=0}^\infty a_k \tilde{f}_{N,k}(x)\]
with 
\[ a_k = \int_0^1 f(t)\, \tilde{f}_{N,k}(t)\,d\m(t).\]
For reasons of simplicity, we take $\m$ to be a symmetric measure. Then $\cos_{\m}^N = \cos_{\m}^D$ and we have $\cos_{\m}^N(z)^2 + \sin_{\m}^N(z)\,\sin_{\m}^D(z) = 1$. From that follows that $\cos_{\m}^N(\sqrt{\l_{N,k}})^2=1$ and it is heuristically clear that $\cos_{\m}^N(\sqrt{\l_{N,k}})=(-1)^k$. Employing this fact and Lemma~\ref{parts}, the computations can be made explicitly, following the lines of the classical (Euclidean) case.

As a first example, take $f(x)= x$. Then, for $k\in \N$,
\begin{align*}
a_k	& = \frac{1}{n_{N,k}} \int_0^1 t\cdot \c_{\l,\m}(\sqrt{\l_k}, t )\,d\m(t)\\
		& = \frac{1}{n_{N,k}} \biggl[ \frac{1}{\sqrt{\l_{N,k}}}\,t\,\s_{\m,\l}\bigl(\sqrt{\l_{N,k}}, t\bigr)\biggr|_0^1 - \frac{1}{\sqrt{\l_{N,k}}} \int_0^1 \s_{\m,\l}\bigl(\sqrt{\l_{N,k}}, t\bigr)\,dt\biggr]\\
		& = \frac{1}{n_{N,k} \l_{N,k}} \c_{\l,\m}(\sqrt{\l_{N,k}}, t )\Bigr|_0^1\\
		& = \frac{1}{n_{N,k} \l_{N,k}} \Bigl( \cos_{\m}^N(\sqrt{\l_{N,k}})-1\Bigr).
\end{align*}
Thus, $a_k = 0$ for even $k\geq 1$ and $a_k=-\dfrac{2}{n_{N,k}\l_{N,k}}$ for odd $k$. Furthermore, we have
\[ a_0 = \int_0^1 t \, d\m(t) = q_2(1)= q_2.\]
Therefore, we have the decomposition into Neumann eigenfunctions
\[ x = q_2 - 2 \sum_{k=0}^\infty \frac{1}{c_{N,2k+1} \l_{N,2k+1}}\,\tilde{f}_{N,2k+1}(x).\]
Note that the required norms $n_{N,k}$ can be computed with Corollary \ref{cor:l2norm}.

We apply Parseval's identity to this series. This gives
\[ \int_0^1 t^2 \, d\m(t) = q_2^2 + \sum_{k=0}^\infty \frac{4}{c_{N,2k+1}^2\, \l_{N,2k+1}^2},\]
and with
\[ \int_0^1 t^2 \, d\m(t) = t\,q_2(t)\Bigr|_0^1 - \int_0^1 q_2(t)\,dt = q_2 - q_3\]
and $1-q_2=p_2$ we get
\[\sum_{k=0}^\infty \frac{1}{c_{N,2k+1}^2\, \l_{N,2k+1}^2}= \frac{1}{4} (p_2\,q_2 - q_3).\]
If we choose the Lebesgue measure for $\m$  (then $p_2=q_2=\frac{1}{2}$, $q_3=\frac{1}{6}$ and $c_{N,2k+1}^2=\frac{1}{2}$), the above equation becomes the well known identity 
\[ \sum_{k=0}^\infty \frac{1}{(2k+1)^4} = \frac{\pi^4}{96}.\]
In the same fashion we compute the decomposition of some more examples ($\m$ symmetric):
\begin{align*}
 x							& = \sum_{k=1}^\infty \frac{(-1)^{k+1}}{n_{D,k}\,\sqrt{\l_{D,k}}}\tilde{f}_{D,k}(x)	\\
1								& = \sum_{k=0}^\infty \frac{2}{c_{D,2k+1}\,\sqrt{\l_{D,2k+1}}}\, \tilde{f}_{D,2k+1}(x)\\
 f_{D,2n+1}(x)	& = \frac{2}{\sqrt{\l_{D,2n+1}}} - 2\sqrt{\l_{D,2n+1}} \sum_{k=1}^\infty \frac{1}{\bigl(\l_{N,2k}-\l_{D,2n+1}\bigr) c_{N,2k}} \, \tilde{f}_{N,2k}(x),\\
								& \qquad \text{for every $n\in\N_0$}
\end{align*}
which are plotted for the standard middle third Cantor measure in Figures \ref{imgfour1} and \ref{imgfour2}. For the images we computed graphs of the eigenfunctions by iteratively applying the formulas in Propositions \ref{cpzsqzS1} and \ref{cpzsqzS2}. For the normalization we used the norms $n_{N,k}$ and $n_{D,k}$ shown in Tables \ref{tab:nef1312} and \ref{tab:def1312}.

Applying Parseval's identity to these decompositions leads, as above, to
\begin{align*}
\sum_{k=1}^\infty \frac{1}{n_{D,k}^2 \l_{D,k}}	& = q_2 - q_3	\\
\sum_{k=0}^\infty \frac{1}{c_{D,2k+1}^2 \, \l_{D,2k+1}}	& = \frac{1}{4}\\
\sum_{k=1}^\infty \frac{1}{\bigl(\l_{N,2k}- \l_{D,2n+1}\bigr)^2\,c_{N,2k}^2} &= \frac{c_{D,2n+1}^2}{4\l_{D,2n+1}}- \frac{1}{\l_{D,2n+1}^2}.
\end{align*}
If we again  take the Lebesgue measure for $\m$, we receive the well known identities
\begin{align*}
\sum_{k=1}^\infty \frac{1}{k^2}	& = \frac{\pi^2}{6}\\
\sum_{k=0}^\infty \frac{1}{(2k+1)^2}	& = \frac{\pi^2}{8}\\
\sum_{k=1}^\infty \frac{1}{\bigl(4k^2-(2n+1)^2\bigr)^2}&=\frac{\pi^2}{16(2n+1)^2} - \frac{1}{2(2n+1)^4}.
\end{align*}

\begin{figure}
\begin{minipage}{.5\textwidth}
\centering
\includegraphics[width=\textwidth]{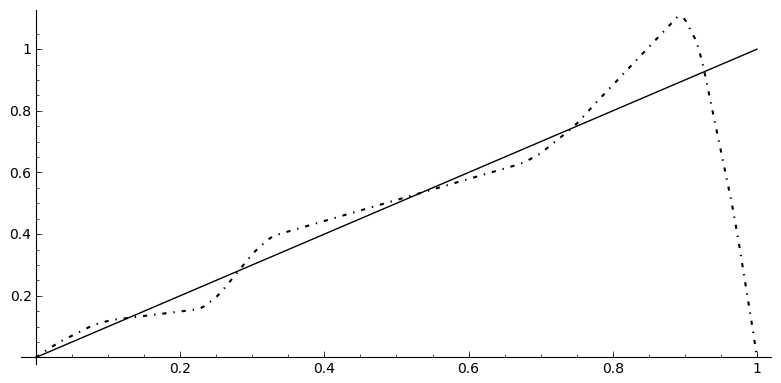}
\end{minipage}
\begin{minipage}{.5\textwidth}
\centering
\includegraphics[width=\textwidth]{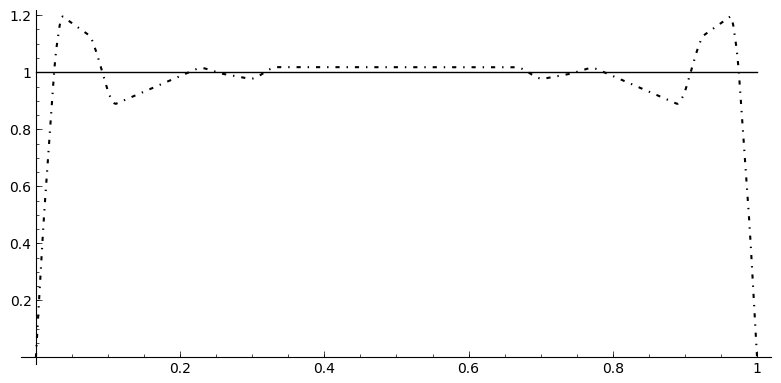}
\end{minipage}
\caption{Left side: approximation of $f(x)=x$ by the first 5 terms of the Dirichlet Fourier series for $r_1=r_2=\frac{1}{3}$ and $m_1=m_2=\frac{1}{2}$. \newline
Right side: approximation of $f(x)=1$ by the first 5 terms of the Dirichlet Fourier series for $r_1=r_2=\frac{1}{3}$ and $m_1=m_2=\frac{1}{2}$.}
\label{imgfour1}
\end{figure}

\begin{figure}
\begin{minipage}{.5\textwidth}
\centering
\includegraphics[width=\textwidth]{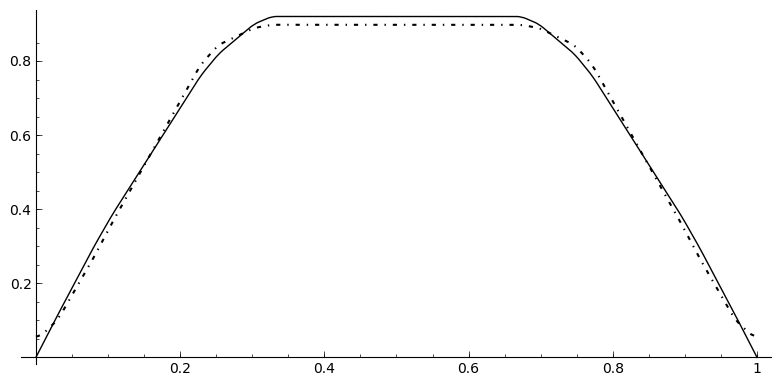}
\end{minipage}
\begin{minipage}{.5\textwidth}
\centering
\includegraphics[width=\textwidth]{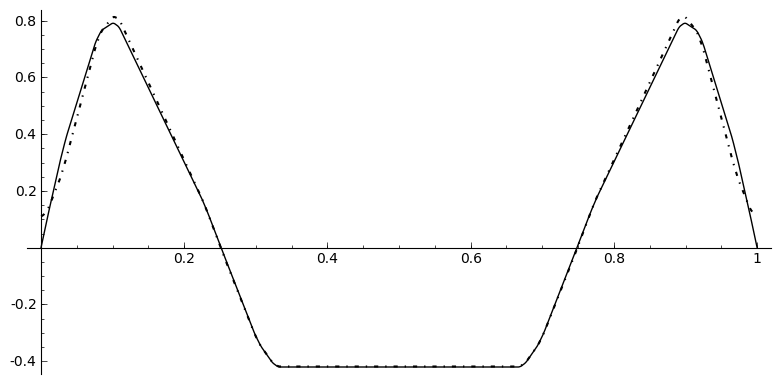}
\end{minipage}
\caption{Left side: approximation of $f_{D,1}$ by the first 3 terms of the Neumann Fourier series for $r_1=r_2=\frac{1}{3}$ and $m_1=m_2=\frac{1}{2}$.\newline
Right side: approximation of $f_{D,3}$ by the first 3 terms of the Neumann Fourier series  for $r_1=r_2=\frac{1}{3}$ and $m_1=m_2=\frac{1}{2}$.}
\label{imgfour2}
\end{figure}
\end{rem2}

\begin{rem2}
The definition of the operator $-\frac{d}{d\m}\frac{d}{dx}$ can be extended to subsets of $\R^d$, $d\in \N$, see, for example, 
Solomyak and Verbitsky \cite{solomyak:spectral}, Naimark and Solomyak \cite{naimark:eigenvalue} and Hu, Lau and Ngai \cite{hu:laplace}. This case, however, is substantially more difficult and the techniques presented here can probably not be readily  extended to it.
\end{rem2}

\begin{rem2}
Analogously to our measure trigonometric functions we can define measure theoretic exponential functions.
For $x\in [0,1]$ and $z\in \C$ put
\[ \e_{\l,\m}(z,x):= \sum_{n=0}^\infty z^{2n}\, p_{2n}(x) + \sum_{n=0}^\infty z^{2n+1}\, q_{2n+1}(x)\]
and
\[ \e_{\m,\l}(z,x):= \sum_{n=0}^\infty z^{2n}\, q_{2n}(x) + \sum_{n=0}^\infty z^{2n+1}\, p_{2n+1}(x).\]
Then, $\e_{\l,\m}(z,\cdot) \in H^2(\l,\m)$ and $\e_{\m,\l}(z,\cdot) \in H^2(\m,\l)$ for every $z \in \C$.
Furthermore, for all $t \in \R$ and $x\in [0,1]$, we have Euler's formula
\[ \e_{\l,\m}(it, x) = \c_{\l,\m}(t,x) + i \s_{\l,\m}(t,x)\]
and
\[ \e_{\m,\l}(it, x) = \c_{\m,\l}(t,x) + i \s_{\m,\l}(t,x).\]

\end{rem2}

\addcontentsline{toc}{chapter}{Bibliography}
\bibliography{ref}

\newpage
\appendix
\section{Program code}

Here we present the code that was used to compute the values and graphics given in Section \ref{numerical}. It was written in Sagemath cloud \cite{sage}, which is based on Python.

Each of the four programs run independently if copied into an empty Sage worksheet, but note that Python is sensitive to indentation.

\paragraph{Program 1} This is used to compute and print $p_n$ and $q_n$.
\begin{verbatim}
# import mpmath for control over the accuracy in the calculations
# mp.dps determines the number of significant digits
# this is not needed if you give the parameters below as fractions, then
# the calculations are done symbolically

from mpmath import mp, mpf
mp.dps = 100
mp.pretty = True

# we set the parameters of the self-similar measure \mu
# r_1 and r_2 are the scaling factors, 
# m_1 and m_2 the weight factors, m_1 + m_2 =1,
# if the values are given as fraction like r1 = 1/3, the computation
# is done symbolically,
# if you give decimal values, write r1=mpf(0.333), then the calculation
# is done with
# the number of significant digits given by mp.dps above

r1=1/3
r2=1/3
m1=1/2
m2=1/2

# we calculate p_n and q_n with the four formulas in Corollary 7.11
# the values are put in a list, the n-th entry is accessed
# by p[n] or q[n], n=0,1,2,...

p=[1,1]
q=[1,1]

for n in range(1,10):

    #p[2n]
    p.append(1/(1-(r1*m1)^n-(r2*m2)^n)*(sum([(r1*m1)^i * (r2*m2)^(n-i)
* p[2*i] * p[2*n-2*i] for i in range(1,n)]) + sum([r1^i * m1^(i+1)
* r2^(n-i) * m2^(n-i-1) * p[2*i+1] * q[2*n-2*i-1] for i in range(0,n)])
+ (1-r1-r2)* sum([r1^i * m1^(i+1) * (r2 * m2)^(n-i-1) * p[2*i+1]
* p[2*n-2*i-2] for i in range(0,n)])))

    #q[2n]
    q.append(1/(1-(r1*m1)^n-(r2*m2)^n)*(sum([(r1*m1)^i * (r2*m2)^(n-i)
* q[2*i] * q[2*n-2*i] for i in range(1,n)]) + sum([r1^(i+1) * m1^i
* r2^(n-i-1) * m2^(n-i) * q[2*i+1] * p[2*n-2*i-1] for i in range(0,n)])
+ (1-r1-r2)* sum([(r1 * m1)^i * r2^(n-i-1) * m2^(n-i) * q[2*i]
* p[2*n-2*i-1] for i in range(0,n)])))

    #p[2n+1]
    p.append(1/(1-r1^n*m1^(n+1)-r2^n*m2^(n+1))*(sum([r1^i*m1^(i+1)
* (r2*m2)^(n-i) * p[2*i+1] * q[2*n-2*i] for i in range(0,n)])
+ sum([r1^i * m1^i * r2^(n-i) * m2^(n-i+1) * p[2*i] * p[2*n-2*i+1]
for i in range(1,n+1)]) + (1-r1-r2)* sum([r1^i * m1^(i+1) * r2^(n-i-1)
* m2^(n-i) * p[2*i+1] * p[2*n-2*i-1] for i in range(0,n)])))

    #q[2n+1]
    q.append(1/(1-r1^(n+1)*m1^n-r2^(n+1)*m2^n)*(sum([r1^(i+1)*m1^i
* (r2*m2)^(n-i) * q[2*i+1] * p[2*n-2*i] for i in range(0,n)])
+ sum([r1^i * m1^i * r2^(n-i+1) * m2^(n-i) * q[2*i] * q[2*n-2*i+1]
for i in range(1,n+1)]) + (1-r1-r2)* sum([r1^i * m1^i * (r2 * m2)^(n-i)
* q[2*i] * p[2*n-2*i] for i in range(0,n+1)])))


#prints a table of the p_n from p_0 to p_19

print(' n  p_n');
for n in range(0,20):
    print('{0:2}  {1}'.format(n,p[n]));


#prints a table of the q_n from q_0 to q_19

print(' n  q_n');
for n in range(0,20):
    print('{0:2}  {1}'.format(n,q[n]));
\end{verbatim}

\paragraph{Program 2}  This first computes $p_n$ and $q_n$ for $n=0,\dotsc, 2\text{an}-1$, then plots the function 
\[ \frac{\sin_{\mu}^N (\sqrt{z})}{\sqrt{z}} \approx \sum_{k=0}^{\text{an}-1}(-1)^k\,  p_{2k+1}\, z^k, \quad z>0,\]
whose zero points are the Neumann eigenvalues. Then, with starting points read off the plot, 
it computes the Neumann eigenvalues numerically. The same is done with the Dirichlet eigenvalues.

\begin{verbatim}
# import mpmath for contol over the accuracy in the calculations
# mp.dps determines the number of significant digits
# findroot enables us to find roots with accuracy given by mp.dps

from mpmath import mp, mpf, findroot
mp.dps = 80
mp.pretty = True

# we set the parameters of the self-similar measure \mu
# r_1 and r_2 are the scaling factors,
# m_1 and m_2 the weight factors, m_1 + m_2 =1,
# for faster computation, write r1=mpf(1/3), ..., this converts the
# fraction in the mpmath-float format with number of significant
# digits given by mp.dps

r1=1/3
r2=1/3
m1=1/2
m2=1/2

# we calculate p_n and q_n with the four formulas in Corollary 7.11
# the values are put in a list, the n-th entry is accessed by p[n] 
# or q[n], n=0,1,2,...
# 'an' gives the number of terms we compute
# that is, we get p_0,..., p_{2*an-1} and q_0, ..., q_{2*an-1}

an = 50;

p=[1,1]
q=[1,1]

for n in range(1,an):

    #p[2n]
    p.append(1/(1-(r1*m1)^n-(r2*m2)^n)*(sum([(r1*m1)^i * (r2*m2)^(n-i)
* p[2*i] * p[2*n-2*i] for i in range(1,n)]) + sum([r1^i * m1^(i+1)
* r2^(n-i) * m2^(n-i-1) * p[2*i+1] * q[2*n-2*i-1] for i in range(0,n)])
+ (1-r1-r2)* sum([r1^i * m1^(i+1) * (r2 * m2)^(n-i-1) * p[2*i+1]
* p[2*n-2*i-2] for i in range(0,n)])))

    #q[2n]
    q.append(1/(1-(r1*m1)^n-(r2*m2)^n)*(sum([(r1*m1)^i * (r2*m2)^(n-i)
* q[2*i] * q[2*n-2*i] for i in range(1,n)]) + sum([r1^(i+1) * m1^i
* r2^(n-i-1) * m2^(n-i) * q[2*i+1] * p[2*n-2*i-1] for i in range(0,n)])
+ (1-r1-r2)* sum([(r1 * m1)^i * r2^(n-i-1) * m2^(n-i) * q[2*i]
* p[2*n-2*i-1] for i in range(0,n)])))

    #p[2n+1]
    p.append(1/(1-r1^n*m1^(n+1)-r2^n*m2^(n+1))*(sum([r1^i*m1^(i+1)
* (r2*m2)^(n-i) * p[2*i+1] * q[2*n-2*i] for i in range(0,n)])
+ sum([r1^i * m1^i * r2^(n-i) * m2^(n-i+1) * p[2*i] * p[2*n-2*i+1]
for i in range(1,n+1)]) + (1-r1-r2)* sum([r1^i * m1^(i+1) * r2^(n-i-1)
* m2^(n-i) * p[2*i+1] * p[2*n-2*i-1] for i in range(0,n)])))

    #q[2n+1]
    q.append(1/(1-r1^(n+1)*m1^n-r2^(n+1)*m2^n)*(sum([r1^(i+1)*m1^i
* (r2*m2)^(n-i) * q[2*i+1] * p[2*n-2*i] for i in range(0,n)])
+ sum([r1^i * m1^i * r2^(n-i+1) * m2^(n-i) * q[2*i] * q[2*n-2*i+1]
for i in range(1,n+1)]) + (1-r1-r2)* sum([r1^i * m1^i * (r2 * m2)^(n-i)
* q[2*i] * p[2*n-2*i] for i in range(0,n+1)])))



# defines the function f(z) = \frac{\sin_{\mu}^N (\sqrt{z})}{\sqrt{z}},
# see the explanation in Example 10.1,
# 'an' denotes the number of considered summands of the series,
# z=mpf(z) converts the argument to mpmath float format for higher
# precision,
# the zeros of this function are the Neumann eigenvalues


def f(z):
    z=mpf(z);
    return sum([(-1)^k * p[2*k+1]*z^k for k in range(0,an)])


# plots the function f on (0,400),
# if 'an' is big enough, you can read approximate values for the Neumann
# eigenvalues from the graph

plot(f,(0,400));


# calculates an approximation of Neumann eigenvalue near the given
# starting point,
# starting points can be read off the plot,
# accuracy depends on 'an' and 'mp.dps'

findroot(f,10);
findroot(f,40);
findroot(f,60);
findroot(f,250);
findroot(f,270);
findroot(f,360);
findroot(f,380);

# defines the function g(z) = \frac{\sin_{\mu}^D (\sqrt{z})}{\sqrt{z}},
# analogously to the Neumann case
# the zeros of this function are the Dirichlet eigenvalues

def g(z):
    z=mpf(z);
    return sum([(-1)^k * q[2*k+1]*mpf(z)^k for k in range(0,an)])


# plots the function g on (0,400)
# if 'an' is big enough, you can read approximate values for the 
# Dirichlet eigenvalues from the graph

plot(g,(0,400));


# calculates an approximation of Dirichlet eigenvalue near the given
# starting point,
# starting points can be read off the plot,
# accuracy depends on 'an' and 'mp.dps'

findroot(g,15);
findroot(g,35);
findroot(g,140);
findroot(g,150);
findroot(g,320);
findroot(g,350);
\end{verbatim}

\paragraph{Program 3} This plots $\sin_{\m}^N$, $\sin_{\m}^D$, $\cos_{\m}^N$ and $\cos_{\m}^D$ as well as the first Neumann and Dirichlet eigenfunctions. Furthermore, it gives approximate values for the suprema of these eigenfunctions.

\begin{verbatim}
from mpmath import mp, mpf, findroot
mp.dps = 80
mp.pretty = True

r1=mpf(1/3)
r2=mpf(1/3)
m1=mpf(1/2)
m2=mpf(1/2)

an = 50;

p=[1,1]
q=[1,1]

for n in range(1,an):

    #p[2n]
    p.append(1/(1-(r1*m1)^n-(r2*m2)^n)*(sum([(r1*m1)^i * (r2*m2)^(n-i)
* p[2*i] * p[2*n-2*i] for i in range(1,n)]) + sum([r1^i * m1^(i+1)
* r2^(n-i) * m2^(n-i-1) * p[2*i+1] * q[2*n-2*i-1] for i in range(0,n)])
+ (1-r1-r2)* sum([r1^i * m1^(i+1) * (r2 * m2)^(n-i-1) * p[2*i+1]
* p[2*n-2*i-2] for i in range(0,n)])))

    #q[2n]
    q.append(1/(1-(r1*m1)^n-(r2*m2)^n)*(sum([(r1*m1)^i * (r2*m2)^(n-i)
* q[2*i] * q[2*n-2*i] for i in range(1,n)]) + sum([r1^(i+1) * m1^i
* r2^(n-i-1) * m2^(n-i) * q[2*i+1] * p[2*n-2*i-1] for i in range(0,n)])
+ (1-r1-r2)* sum([(r1 * m1)^i * r2^(n-i-1) * m2^(n-i) * q[2*i]
* p[2*n-2*i-1] for i in range(0,n)])))

    #p[2n+1]
    p.append(1/(1-r1^n*m1^(n+1)-r2^n*m2^(n+1))*(sum([r1^i*m1^(i+1)
* (r2*m2)^(n-i) * p[2*i+1] * q[2*n-2*i] for i in range(0,n)])
+ sum([r1^i * m1^i * r2^(n-i) * m2^(n-i+1) * p[2*i] * p[2*n-2*i+1]
for i in range(1,n+1)]) + (1-r1-r2)* sum([r1^i * m1^(i+1) * r2^(n-i-1)
* m2^(n-i) * p[2*i+1] * p[2*n-2*i-1] for i in range(0,n)])))

    #q[2n+1]
    q.append(1/(1-r1^(n+1)*m1^n-r2^(n+1)*m2^n)*(sum([r1^(i+1)*m1^i
* (r2*m2)^(n-i) * q[2*i+1] * p[2*n-2*i] for i in range(0,n)])
+ sum([r1^i * m1^i * r2^(n-i+1) * m2^(n-i) * q[2*i] * q[2*n-2*i+1]
for i in range(1,n+1)]) + (1-r1-r2)* sum([r1^i * m1^i * (r2 * m2)^(n-i)
* q[2*i] * p[2*n-2*i] for i in range(0,n+1)])))


# defines \sin_{\mu}^N, \sin_{\mu}^D, \cos_{mu}^N and \cos_{mu}^D,
# again, only the first 'an' terms are considered

# z=mpf(z) converts the argument to mpmath float format for higher 
# precision

def sinN(z):
    z = mpf(z)
    return sum([(-1)^k * p[2*k+1]*z^(2*k+1) for k in range(0,an)])
def sinD(z):
    z = mpf(z)
    return sum([(-1)^k * q[2*k+1]*z^(2*k+1) for k in range(0,an)])
def cosN(z):
    z = mpf(z)
    return sum([(-1)^k * p[2*k]*z^(2*k) for k in range(0,an)])
def cosD(z):
    z = mpf(z)
    return sum([(-1)^k * q[2*k]*z^(2*k) for k in range(0,an)])


# plots the above defined functions

plot(sinN,(0,21))
plot(sinD,(0,21))
plot(cosN,(0,21))
plot(cosD,(0,21))


# next we plot the eigenfunctions (c_{\lambda,\mu}(z,\cdot) and 
# s_{\lambda,\mu}(z,\cdot) where z is the square root of an eigenvalue)
# we do this by iterative use of Propositions 7.7 and 7.8

# we define the IFS S_1, S_2

def S1(x):
    return r1*x
def S2(x):
    return r2*x-r2+1


# 'it' gives the number of iterations

it=4


# x is the list of 'corner points' of the self-similar set, that is,
# of the 'it'-th iteration

x=[0,1]

for i in range(it):
    x = map(S1,x) + map(S2,x)


# with auxiliary function f2, g1, g2, zit we construct values of 
# c_{\lambda, \mu} and s_{\lambda,\mu} iteratively at the points given 
# in 'x'

def f2(a,b,z):
    return (cosN( sqrt(r1*m1)*z ) - (1-(r1+r2))*sqrt(m1/r1)*z
*sinN(sqrt(r1*m1)*z))*a - sqrt(r2*m1/(r1*m2))*sinN(sqrt(r1*m1)*z)*b

def g1(a):
    return sqrt(r1/m1)*a

def g2(a,b,z):
    return (sqrt(r1/m1)*sinD( sqrt(r1*m1)*z ) + (1-(r1+r2))*z
*cosD(sqrt(r1*m1)*z))*a + sqrt(r2/m2)*cosD(sqrt(r1*m1)*z) * b

def zit(z,i):
    if i == 0:
        return [z,z]
    else:
        return zit(z,i-1) + zit(z, i-1)

def clm(z,i):
    if i == 0:
        return [1, cosN(z)]
    else:
        return clm(sqrt(r1*m1)*z,i-1) + map(f2,clm(sqrt(r2*m2)*z,i-1),
				slm( sqrt(r2*m2)*z, i-1 ),zit(z,i-1) )

def slm(z,i):
    if i == 0:
        return [0, sinD(z)]
    else:
        return map(g1, slm(sqrt(r1*m1)*z,i-1)) + map(g2,
				clm(sqrt(r2*m2)*z,i-1), slm( sqrt(r2*m2)*z, i-1 ), zit(z,i-1))


# sets the first Neumann and Dirichlet eigenvalues for the standard
# Cantor measure, calculated in 'determination of eigenvalues'

lN1 = mpf(7.097431098141122)
lN2 = mpf(42.584586588846733)
lN3 = mpf(61.344203922701662)
lN4 = mpf(255.507519533080403)
lN5 = mpf(272.983570819147205)
lN6 = mpf(368.065223536209975)
lN7 = mpf(383.552883127693176)

lD1 = mpf(14.435240512053874)
lD2 = mpf(35.260238024277225)
lD3 = mpf(140.781053384556059)
lD4 = mpf(151.290616055019631)
lD5 = mpf(326.057328357753770)
lD6 = mpf(353.416920767557756)


# plots the first seven Neumann eigenfunctions (not normed)
# the points at x ('corner points') are joined by straight lines, 
# this is alright because these 'gap intervals' do not belong to the
# support of the measure \mu

list_plot(zip(x,clm(sqrt(lN1),it)),plotjoined=true,thickness=0.5)
list_plot(zip(x,clm(sqrt(lN2),it)),plotjoined=true,thickness=0.5)
list_plot(zip(x,clm(sqrt(lN3),it)),plotjoined=true,thickness=0.5)
list_plot(zip(x,clm(sqrt(lN4),it)),plotjoined=true,thickness=0.5)
list_plot(zip(x,clm(sqrt(lN5),it)),plotjoined=true,thickness=0.5)
list_plot(zip(x,clm(sqrt(lN6),it)),plotjoined=true,thickness=0.5)
list_plot(zip(x,clm(sqrt(lN7),it)),plotjoined=true,thickness=0.5)


# plots the first six Dirichlet eigenfunctions (not normed)

list_plot(zip(x,slm(sqrt(lD1),it)),plotjoined=true,thickness=0.5)
list_plot(zip(x,slm(sqrt(lD2),it)),plotjoined=true,thickness=0.5)
list_plot(zip(x,slm(sqrt(lD3),it)),plotjoined=true,thickness=0.5)
list_plot(zip(x,slm(sqrt(lD4),it)),plotjoined=true,thickness=0.5)
list_plot(zip(x,slm(sqrt(lD5),it)),plotjoined=true,thickness=0.5)
list_plot(zip(x,slm(sqrt(lD6),it)),plotjoined=true,thickness=0.5)


# we compute the sup-norm of the eigenfunctions

max(clm(sqrt(lN1),it))
max(clm(sqrt(lN2),it))

max(slm(sqrt(lD1),it))
max(slm(sqrt(lD2),it))
\end{verbatim}

\paragraph{Program 4} This computes the $L2(\mu)$-norms of the Neumann eigenfunctions 
\[ c_{\lambda,\mu}(\sqrt{\l_{N,k}}, \cdot) \]
 and  Dirichlet eigenfunctions
\[ s_{\lambda,\mu}(\sqrt{\l_{D,k}}, \cdot) \]
 with the formulas from Corollary \ref{cor:l2norm}.

\begin{verbatim}
from mpmath import mp, mpf, findroot
mp.dps = 80
mp.pretty = True

r1=mpf(1/3)
r2=mpf(1/3)
m1=mpf(1/2)
m2=mpf(1/2)

an = 50;

p=[1,1]
q=[1,1]

for n in range(1,an):

    #p[2n]
    p.append(1/(1-(r1*m1)^n-(r2*m2)^n)*(sum([(r1*m1)^i * (r2*m2)^(n-i)
* p[2*i] * p[2*n-2*i] for i in range(1,n)]) + sum([r1^i * m1^(i+1)
* r2^(n-i) * m2^(n-i-1) * p[2*i+1] * q[2*n-2*i-1] for i in range(0,n)])
+ (1-r1-r2)* sum([r1^i * m1^(i+1) * (r2 * m2)^(n-i-1) * p[2*i+1]
* p[2*n-2*i-2] for i in range(0,n)])))

    #q[2n]
    q.append(1/(1-(r1*m1)^n-(r2*m2)^n)*(sum([(r1*m1)^i * (r2*m2)^(n-i)
* q[2*i] * q[2*n-2*i] for i in range(1,n)]) + sum([r1^(i+1) * m1^i
* r2^(n-i-1) * m2^(n-i) * q[2*i+1] * p[2*n-2*i-1] for i in range(0,n)])
+ (1-r1-r2)* sum([(r1 * m1)^i * r2^(n-i-1) * m2^(n-i) * q[2*i]
* p[2*n-2*i-1] for i in range(0,n)])))

    #p[2n+1]
    p.append(1/(1-r1^n*m1^(n+1)-r2^n*m2^(n+1))*(sum([r1^i*m1^(i+1)
* (r2*m2)^(n-i) * p[2*i+1] * q[2*n-2*i] for i in range(0,n)])
+ sum([r1^i * m1^i * r2^(n-i) * m2^(n-i+1) * p[2*i] * p[2*n-2*i+1]
for i in range(1,n+1)]) + (1-r1-r2)* sum([r1^i * m1^(i+1) * r2^(n-i-1)
* m2^(n-i) * p[2*i+1] * p[2*n-2*i-1] for i in range(0,n)])))

    #q[2n+1]
    q.append(1/(1-r1^(n+1)*m1^n-r2^(n+1)*m2^n)*(sum([r1^(i+1)*m1^i
* (r2*m2)^(n-i) * q[2*i+1] * p[2*n-2*i] for i in range(0,n)])
+ sum([r1^i * m1^i * r2^(n-i+1) * m2^(n-i) * q[2*i] * q[2*n-2*i+1]
for i in range(1,n+1)]) + (1-r1-r2)* sum([r1^i * m1^i * (r2 * m2)^(n-i)
* q[2*i] * p[2*n-2*i] for i in range(0,n+1)])))

# sets the first Neumann and Dirichlet eigenvalues for the standard 
# Cantor measure, calculated in 'determination of eigenvalues'

lN1 = mpf(7.097431098141122)
lN2 = mpf(42.584586588846733)
lN3 = mpf(61.344203922701662)
lN4 = mpf(255.507519533080403)
lN5 = mpf(272.983570819147205)
lN6 = mpf(368.065223536209975)
lN7 = mpf(383.552883127693176)

lD1 = mpf(14.435240512053874)
lD2 = mpf(35.260238024277225)
lD3 = mpf(140.781053384556059)
lD4 = mpf(151.290616055019631)
lD5 = mpf(326.057328357753770)
lD6 = mpf(353.416920767557756)


# L2(\mu)-norms of the Neumann eigenfunctions 
# c_{\lambda,\mu}(\sqrt{lNk}, \cdot) for k=1, ..., 7
# with Corollary 4.3

sqrt(sum([(-1)^n * lN1^n * sum([(n+1-2*k)*p[2*k]*p[2*n+1-2*k] 
for k in range(0,n+1) ]) for n in range(0,an) ]))
sqrt(sum([(-1)^n * lN2^n * sum([(n+1-2*k)*p[2*k]*p[2*n+1-2*k] 
for k in range(0,n+1) ]) for n in range(0,an) ]))
sqrt(sum([(-1)^n * lN3^n * sum([(n+1-2*k)*p[2*k]*p[2*n+1-2*k] 
for k in range(0,n+1) ]) for n in range(0,an) ]))
sqrt(sum([(-1)^n * lN4^n * sum([(n+1-2*k)*p[2*k]*p[2*n+1-2*k] 
for k in range(0,n+1) ]) for n in range(0,an) ]))
sqrt(sum([(-1)^n * lN5^n * sum([(n+1-2*k)*p[2*k]*p[2*n+1-2*k] 
for k in range(0,n+1) ]) for n in range(0,an) ]))
sqrt(sum([(-1)^n * lN6^n * sum([(n+1-2*k)*p[2*k]*p[2*n+1-2*k] 
for k in range(0,n+1) ]) for n in range(0,an) ]))
sqrt(sum([(-1)^n * lN7^n * sum([(n+1-2*k)*p[2*k]*p[2*n+1-2*k] 
for k in range(0,n+1) ]) for n in range(0,an) ]))


# L2(\mu)-norms of the Dirichlet eigenfunctions
# s_{\lambda,\mu}(\sqrt{lDk}, \cdot) for k = 1, ..., 6

sqrt(sum([(-1)^n * lD1^(n+1) * sum([(n+1-2*k)*q[2*k+1]*q[2*n+2-2*k] 
for k in range(0,n+2) ]) for n in range(0,an-1) ]))
sqrt(sum([(-1)^n * lD2^(n+1) * sum([(n+1-2*k)*q[2*k+1]*q[2*n+2-2*k] 
for k in range(0,n+2) ]) for n in range(0,an-1) ]))
sqrt(sum([(-1)^n * lD3^(n+1) * sum([(n+1-2*k)*q[2*k+1]*q[2*n+2-2*k] 
for k in range(0,n+2) ]) for n in range(0,an-1) ]))
sqrt(sum([(-1)^n * lD4^(n+1) * sum([(n+1-2*k)*q[2*k+1]*q[2*n+2-2*k] 
for k in range(0,n+2) ]) for n in range(0,an-1) ]))
sqrt(sum([(-1)^n * lD5^(n+1) * sum([(n+1-2*k)*q[2*k+1]*q[2*n+2-2*k] 
for k in range(0,n+2) ]) for n in range(0,an-1) ]))
sqrt(sum([(-1)^n * lD6^(n+1) * sum([(n+1-2*k)*q[2*k+1]*q[2*n+2-2*k] 
for k in range(0,n+2) ]) for n in range(0,an-1) ]))
\end{verbatim}

\end{document}